\documentclass[11pt, a4paper,leqno]{amsart}
\usepackage{amsmath,amsthm,amscd,amssymb,amsfonts, amsbsy}
\usepackage{latexsym}
\usepackage{txfonts}
\usepackage{exscale}

\usepackage[colorlinks=true, pdfstartview=FitV, linkcolor=blue, citecolor=red, urlcolor=blue]{hyperref} 

\usepackage{pgf}
\usepackage{color}

\calclayout
\allowdisplaybreaks

\theoremstyle{plain}
\newtheorem{theorem}[equation]{Theorem}
\newtheorem{lemma}[equation]{Lemma}
\newtheorem{corollary}[equation]{Corollary}
\newtheorem{proposition}[equation]{Proposition}

\theoremstyle{definition}
\newtheorem{definition}[equation]{Definition}

\theoremstyle{remark}
\newtheorem{remark}[equation]{Remark}

\numberwithin{equation}{section}

\newcommand{\RR}{{\mathbb{R}}}
\newcommand{\eps}{\varepsilon}
\newcommand{\ep}{\epsilon_1}

\newcommand{\dint}{\int\!\!\!\int}

\newcommand{\dist}{\operatorname{dist}}

\newcommand{\td}{\Delta_{\star}}
\newcommand{\tdelta}{\delta_{\star}}
\newcommand{\tom}{\omega_{\star}}

\newcommand{\re}{\mathbb{R}}

\newcommand{\ree}{\mathbb{R}^{n+1}}

\newcommand{\dd}{\mathbb{D}}

\newcommand{\C}{\mathcal{C}}

\newcommand{\F}{\mathcal{F}}
\newcommand{\E}{\mathcal{E}}
\newcommand{\M}{\mathcal{M}}
\newcommand{\W}{\mathcal{W}}
\newcommand{\T}{\mathcal{T}}
\newcommand{\I}{\mathcal{I}}

\newcommand{\R}{\mathcal{R}}
\newcommand{\G}{\widetilde{G}}

\newcommand{\Fr}{\mathcal{F}(\rho)}

\newcommand{\mut}{\mathfrak{m}}
\newcommand{\pom}{\partial\Omega}
\newcommand{\pomfqo}{\partial\Omega_{\F,Q_0}}

\newcommand{\hm}{\omega}

\renewcommand{\P}{\mathcal{P}}

\renewcommand{\emptyset}{\mbox{\textup{\O}}}

\DeclareMathOperator{\diam}{diam}
\DeclareMathOperator{\interior}{int}

\def\div{\mathop{\operatorname{div}}\nolimits}

\begin{document}
\allowdisplaybreaks

\title[Uniform rectifiability and harmonic measure I]{Uniform rectifiability and harmonic measure I: Uniform rectifiability implies Poisson kernels in $L^p$
\\[.5cm]
Uniforme rectifiabilit\'e et mesure harmonique I: l'uniforme
 rectifiabilit\'e entra\^{\i}ne le noyau de Poisson dans $L^p$.
}

\author{Steve Hofmann}

\address{Steve Hofmann
\\
Department of Mathematics
\\
University of Missouri
\\
Columbia, MO 65211, USA} \email{hofmanns@missouri.edu}

\author{Jos\'e Mar{\'\i}a Martell}

\address{Jos\'e Mar{\'\i}a Martell
\\
Instituto de Ciencias Matem\'aticas CSIC-UAM-UC3M-UCM
\\
Consejo Superior de Investigaciones Cient{\'\i}ficas
\\
C/ Nicol\'as Cabrera, 13-15
\\
E-28049 Madrid, Spain} \email{chema.martell@icmat.es}

\thanks{The first author was supported by NSF grant DMS-0801079.
The second author was supported by MINECO Grant MTM2010-16518 and ICMAT Severo Ochoa project SEV-2011-0087.
\\
\indent
This work has been possible thanks to the support and hospitality of the \textit{University of Missouri-Columbia} (USA),  the \textit{Consejo Superior de Investigaciones Cient{\'\i}ficas} (Spain), the \textit{Universidad Aut\'onoma de Madrid} (Spain), and the \textit{Australian National University} (Canberra, Australia). Both authors express their gratitude to these institutions. Both authors are also grateful to P. Auscher and M. Perrin for their help in the French translations.}

\date{September 25, 2012. \textit{Revised}: \today}


\begin{abstract}
We present a higher dimensional, scale-invariant version of a classical theorem of F. and M. Riesz
\cite{Rfm}.
More precisely,
we establish scale invariant absolute continuity of harmonic measure with respect to surface measure,
along with higher integrability of the Poisson kernel, for a domain $\Omega\subset \re^{n+1},\, n\geq 2$, with a uniformly rectifiable boundary, which satisfies the Harnack Chain condition plus
an interior (but not exterior) corkscrew condition.
In a companion paper to this one \cite{HMU},
we also establish a converse, in which we deduce uniform rectifiability of the boundary, assuming  scale invariant $L^q$ bounds, with $q>1$, on the Poisson kernel.
\end{abstract}

\newcommand{\status}[1]{\vskip-1.2cm\noindent\parbox{\textwidth}{%
\hfill \texttt{#1}}\vskip.2cm\null}

\status{Ann.~Sci.~\'{E}cole~Norm.~Sup.~47 (2014), no.~3, 577-654.}

\maketitle

\begin{quote}\footnotesize
\textsc{R\'esum\'e.~}
On pr\'esente une version invariante par \'echelles et en  dimension sup\'e\-rieure \`a  3 d'un th\'eor\`{e}me classique de F. et M. Riesz \cite{Rfm}. Plus pr\'ecis\'ement, on \'etablit  l'absolue continuit\'e de la mesure harmonique par rapport \`{a} la mesure de surface, ainsi qu'un gain d'int\'egrabilit\'e pour le noyau de Poisson, pour un domaine $\Omega\subset \re^{n+1},\, n\geq 2$, \`a  bord uniform\'ement rectifiable, v\'erifiant une condition de cha\^{i}ne de Harnack et une condition de type ``points d'ancrage'' ou ``corkscrew'' int\'erieure (mais pas  ext\'erieure). L'article associ\'e \cite{HMU} \'etablit une r\'eciproque, c'est-\`{a}-dire l'uniforme rectifiabilit\'e du bord en supposant des estim\'ees invariantes par \'echelle $L^q$ pour $q>1$ sur le noyau de Poisson.

\end{quote}

\newpage

\tableofcontents

\section{Introduction}

In \cite{Rfm}, F. and M. Riesz showed that for a simply connected domain in the complex plane
with a rectifiable boundary, harmonic measure is absolutely continuous with respect to arclength measure.  A quantitative version of this theorem was
obtained by Lavrentiev \cite{Lav}.
More generally, if only a portion of the boundary is rectifiable, Bishop and Jones
\cite{BiJo} have shown that harmonic measure
is absolutely continuous with respect to arclength on that portion.  They also present a counter-example
to show that the result of \cite{Rfm}
may fail in the absence of some topological hypothesis
(e.g., simple connectedness).

In this paper we extend the results of \cite{Rfm} and \cite{Lav} to higher dimensions,
without imposing extra assumptions on either the exterior domain or the boundary,
as has been done previously.
Our extension (Theorem \ref{theor-main-I} below) is  ``scale-invariant'', i.e.,
assuming scale-invariant analogues
of the hypotheses  of \cite{Rfm}, we show that harmonic measure
satisfies a scale-invariant version of absolute continuity,
namely the weak-$A_\infty$ condition (cf. Definition \ref{def1.ainfty} below).
More precisely, let $\Omega \subset \mathbb{R}^{n+1},\,n\geq 2,$
be a connected, open
set.  We establish the
weak-$A_\infty$ property of harmonic measure,
assuming that $\partial \Omega$ is uniformly rectifiable
(cf.   \eqref{eq1.UR} below), and that $\Omega$ satisfies interior (but not necessarily exterior)
Corkscrew and Harnack Chain conditions
(cf. Definitions
\ref{def1.cork} and \ref{def1.hc} below).  Uniform rectifiability is
the scale-invariant version of rectifiability, while
the Corkscrew and Harnack Chain conditions are scale invariant analogues
of the topological properties of
openness and path connectedness, respectively.
We emphasize that in contrast to previous work in this area
in dimensions $n+1\geq 3$, we impose no restriction
on the geometry of the {\it exterior} domain $\Omega_{ext}:= \ree\setminus \overline{\Omega}$, nor
any extra condition on the geometry of the boundary, beyond uniform rectifiability.  In particular, we do
not require that any component of $\Omega_{ext}$ satisfy a Corkscrew
condition (as in \cite{JK},
\cite{S}, \cite{Ba}) or even an $n$-disk condition as in \cite{DJe}; nor do we assume that
$\partial \Omega$ contains ``Big Pieces'' of the boundaries of Lipschitz sub-domains
of $\Omega$, as in \cite{BL}.  The absence of such assumptions is the main advance in the present paper.

In addition, in a companion paper to this one \cite{HMU}, written jointly with
I. Uriarte-Tuero, we establish a converse, Theorem \ref{th2.2}, in which
we deduce uniform rectifiability of the boundary, given a
certain scale invariant local $L^q$ estimate, with $q>1$, for the Poisson kernel
 (cf. \eqref{eq1.lq}).   The method of proof in \cite{HMU} may be of
 independent interest, as it entails a novel
use of ``$Tb$'' theory to obtain a free boundary result.

Taken together, the main results of the present paper and of
\cite{HMU}, namely Theorems \ref{theor-main-I} and \ref{th2.2} below, may be summarized as follows
(the terminology and notation used in the statement will be clarified or cross-referenced immediately afterwards):

\begin{theorem}\label{th1.25} Let $\Omega \subset \ree,\,n\geq2,$
be a connected open set which satisfies interior
Corkscrew and Harnack Chain conditions, and whose boundary
$\partial\Omega$ is $n$-dimensional Ahlfors-David regular.   Then the following are equivalent:
\begin{enumerate}
\item $\partial\Omega$ is uniformly rectifiable.
\smallskip
\item For every surface ball $\Delta=\Delta(x,r) \subset \partial \Omega$, with
radius  $r\lesssim \diam \partial\Omega$,  the harmonic measure
$\omega^{X_\Delta} \in$
weak-$A_\infty(\Delta)$.
\smallskip
\item $\omega<< \sigma$, and there is a
$q>1$ such that
the Poisson kernel $k^{X_\Delta}$
satisfies the scale invariant $L^q$ bound
\eqref{eq1.lq}, for every  $\Delta=\Delta(x,r) \subset \partial \Omega$, with
radius  $r\lesssim \diam \partial\Omega$.
\end{enumerate}
\end{theorem}

\begin{remark}  By the counter-example of \cite{BiJo}, one would not
expect to obtain the implication $(1)\implies (2)$, without some sort of connectivity assumption;
for us, the interior Harnack Chain condition
plays this role.
\end{remark}

Given a domain $\Omega\subset \ree$, a
 ``surface ball'' is a set
 $\Delta= \Delta(x,r):=B(x,r)\cap\partial\Omega$, where $x\in\partial\Omega$, and $B(x,r)$ denotes
the standard $(n+1)$-dimensional Euclidean ball of radius $r$ centered at $x$.
For such a surface ball $\Delta$, we let
$\omega^{X_\Delta}$ denote harmonic measure for $\Omega$,
with pole at the ``Corkscrew point'' $X_\Delta$ (see Definition \ref{def1.cork}).
The Corkscrew and Harnack Chain conditions, as well as the
notions of Ahlfors-David regularity (ADR),
uniform rectifiability  (UR) and weak-$A_\infty$, are described in
Definitions
\ref{def1.cork}, \ref{def1.hc},  \ref{def1.ADR}, \ref{def1.UR}, and \ref{def1.ainfty} below.

The present paper treats the direction (1) implies (2).
That (2) implies (3) is well known (see the discussion following Definition
\ref{def1.ainfty}).   The main result in \cite{HMU} is that (3) implies (1).
We mention also that we obtain in the present paper an extension of
(1) implies (2), in which our hypotheses are assumed to hold only in an
``interior big pieces'' sense (cf. Definition \ref{def1.bp} and Theorem
\ref{theor-main-I:BP} below).

To place Theorem \ref{th1.25} in context, we review previous related work
in dimension $n+1\geq 3$.
We recall that in \cite{JK}, the authors introduce
the notion of  a  ``non-tangentially accessible'' (NTA) domain:
$\Omega$ is said to be NTA if
it satisfies the Corkscrew
and Harnack Chain conditions
(``interior Corkscrew and Harnack Chain conditions''), and also if
the exterior domain, $\Omega_{ext}:= \ree\setminus \overline{\Omega}$ (which need not be connected),
satisfies the Corkscrew condition
(``exterior Corkscrew condition'').
The latter was relaxed
in \cite{DJe} to allow a sort of  ``weak exterior Corkscrew''
condition in which the analogue of the exterior Corkscrew point
is the center merely of an $n$-dimensional
disk in $\Omega_{ext}$, rather than of a full Euclidean ball.
A key observation made in \cite{DJe} was that the weak exterior Corkscrew condition is still
enough to obtain local H\"older continuity at the boundary of
harmonic functions which vanish on a surface ball.   In \cite{DJe}, the authors prove that,
 in the presence of Ahlfors-David regularity of the boundary,
the NTA condition of \cite{JK} or even its relaxed version with ``weak exterior Corkscrews'',
implies that $\Omega$ satisfies an
``interior big pieces'' of Lipschitz sub-domains condition
(cf. Definition \ref{def1.bp} below).
By a simple maximum principle argument (plus the deep result of \cite{Dah}),
one then almost immediately obtains a certain lower bound for harmonic measure, to wit, that
there are constants $\eta\in (0,1)$ and $c_0>0$ such that for
each surface ball $\Delta\subset\partial\Omega$, and any Borel subset $A\subset\Delta$,
we have \begin{equation}\label{eq1.2***}
\omega^{X_\Delta}(A)\geq c_0\,,\qquad {\rm whenever}\,\,\,\sigma(A)\geq \eta \,\sigma(\Delta).
\end{equation}
In turn, still given NTA, or at least the relaxed version of \cite{DJe},
the latter bound self-improves to an $A_\infty$ estimate for harmonic measure,
via the comparison principle.
The same $A_\infty$ conclusion was also obtained
by a different argument in \cite{S},
under the full NTA condition of \cite{JK}.
In \cite{BL}, the authors impose an interior Corkscrew condition, but
in lieu of the Harnack Chain and exterior (or weak exterior) Corkscrew conditions,
the authors assume instead the consequence of these conditions
deduced in \cite{DJe},
namely, that $\Omega$ satisfies the aforementioned
condition concerning
``interior big pieces'' of Lipschitz sub-domains.    The bound
\eqref{eq1.2***} (suitably interpreted)
then holds
almost immediately (again by the maximum principle),
but the self-improvement argument, in the absence of
the Harnack Chain and exterior (or weak exterior) Corkscrew conditions, is now
more problematic
(indeed, the usual proofs of the comparison principle rely on Harnack's inequality
and local H\"older continuity at the boundary),
and the authors conclude in \cite{BL} only that $\omega$ is weak-$A_\infty$.  On the other hand,
they give an example to show that this conclusion is best possible (that is, they construct
a domain which satisfies the ``interior big pieces'' condition, but whose harmonic measure
fails to be doubling).  We mention also in this context the recent paper
\cite{Ba}, in which the geometric conclusion of \cite{DJe}, namely the existence of
``interior big pieces'' of Lipschitz sub-domains, is shown to hold assuming the full
NTA condition (with two-sided Corkscrews), but in which only the lower
(but not the upper) bound
is required in the Ahlfors-David condition (cf.
\eqref{eq1.ADR}).

In the present paper, we improve the results of \cite{BL} and of \cite{DJe} by removing
the ``big pieces of Lipschitz sub-domains'' hypothesis,  as well as all assumptions regarding
the exterior domain.  That is, in Theorem \ref{theor-main-I},
we assume only that $\Omega$ satisfies {\it interior} Corkscrew and Harnack Chain conditions,
and that its boundary is uniformly rectifiable.  More generally, in
Theorem \ref{theor-main-I:BP}, we suppose only
that these hypotheses hold in an appropriate ``interior big pieces'' sense (in particular,  our results include
those of \cite{BL} as
a special case, since their Lipschitz sub-domains clearly satisfy our hypotheses).
The difficulty now, and the heart of the proof, is
to establish \eqref{eq1.2***};
with the latter in hand, the self-improvement to weak $A_\infty$
proceeds as in \cite{BL}.
We mention that by an unpublished example of Hrycak,
UR does not, in general,  imply big pieces of Lipschitz graphs\footnote{
On the other hand, Azzam and Schul \cite{AS}
have recently shown that every UR set contains ``big pieces of big pieces of Lipschitz graphs''
(see \cite[pp. 15-16]{DS2} or \cite{AS} for a precise formulation).   This is a beautiful result, but seems inapplicable to the estimates for harmonic measure considered here:
to enable essential use of the maximum principle,
one would need ``{\bf interior} big pieces (cf. Definition \ref{def1.bp} below) of {\bf interior}
big pieces of Lipschitz subdomains'' (say, in the presence of the 1-sided NTA condition), and it is not clear that the methods of \cite{AS} would yield such a result .  We do expect that the methods of the present paper could be pushed to do so, and we plan to present these arguments, with applications to more general elliptic-harmonic measures, in a forthcoming paper.} (that the opposite
implication does hold for ADR sets is easy, and well known).
Moreover,  in \cite{HMU} we obtain a converse
which shows that the UR hypothesis is optimal. In this connection, we mention also the following observation, which was brought to our attention by M. Badger and T. Toro.  Let $F\subset \re^2$ denote the ``4 corners Cantor set'' of J. Garnett (see, e.g., \cite[p. 4]{DS2}), and let
$F^*:= F \times \re \subset \re^3$ be the ``cylinder'' above $F$.  Then $\Omega:= \re^3 \setminus F^*$
satisfies the (interior) Corkscrew and Harnack Chain conditions,
and has a $2$-dimensional ADR boundary,
but the boundary is not UR, and
therefore its harmonic measure is not weak-$A_\infty$.

We conclude this historical survey
by providing some additional context for our work here and in \cite{HMU},
namely, that our results may
be viewed as a ``large constant'' analogue of the work of
Kenig and Toro \cite{KT1,KT2,KT3}.   The latter,
taken collectively, say that in the presence of a Reifenberg flatness
condition
and Ahlfors-David regularity,  one has that
$\log k \in VMO$ iff $\nu \in VMO$, where $k$ is the Poisson kernel with pole at some fixed point, and
$\nu$ is the unit normal to the boundary.  Moreover, given the same background hypotheses,
the condition that $\nu \in VMO$ is equivalent to
a uniform rectifiability (UR) condition with vanishing trace,
thus $\log k \in VMO \iff vanishing \,\, UR.$  On the other hand, our
large constant version ``almost'' says  ``$\,\log k \in BMO\iff UR\,$'',
given interior Corkscrews and Harnack Chains.
Indeed,  it is well known that the $A_\infty$ condition
(i.e.,  weak-$A_\infty$ plus the doubling property) implies that $\log k \in BMO$, while
if $\log k \in BMO$ with small norm, then $k\in A_\infty$.

In order to state our results precisely, we shall first need to
discuss some preliminary matters.

\subsection{Notation and Definitions}

\begin{list}{$\bullet$}{\leftmargin=0.4cm  \itemsep=0.2cm}

\item We use the letters $c,C$ to denote harmless positive constants, not necessarily
the same at each occurrence, which depend only on dimension and the
constants appearing in the hypotheses of the theorems (which we refer to as the
``allowable parameters'').  We shall also
sometimes write $a\lesssim b$ and $a \approx b$ to mean, respectively,
that $a \leq C b$ and $0< c \leq a/b\leq C$, where the constants $c$ and $C$ are as above, unless
explicitly noted to the contrary.  At times, we shall designate by $M$ a particular constant whose value will remain unchanged throughout the proof of a given lemma or proposition, but
which may have a different value during the proof of a different lemma or proposition.

\item Given a domain $\Omega \subset \ree$, we shall
use lower case letters $x,y,z$, etc., to denote points on $\partial \Omega$, and capital letters
$X,Y,Z$, etc., to denote generic points in $\ree$ (especially those in $\ree\setminus \partial\Omega$).

\item The open $(n+1)$-dimensional Euclidean ball of radius $r$ will be denoted
$B(x,r)$ when the center $x$ lies on $\partial \Omega$, or $B(X,r)$ when the center
$X \in \ree\setminus \partial\Omega$.  A ``surface ball'' is denoted
$\Delta(x,r):= B(x,r) \cap\partial\Omega.$

\item Given a Euclidean ball $B$ or surface ball $\Delta$, its radius will be denoted
$r_B$ or $r_\Delta$, respectively.

\item Given a Euclidean or surface ball $B= B(X,r)$ or $\Delta = \Delta(x,r)$, its concentric
dilate by a factor of $\kappa >0$ will be denoted
by $\kappa B := B(X,\kappa r)$ or $\kappa \Delta := \Delta(x,\kappa r).$

\item For $X \in \ree$, we set $\delta(X):= \dist(X,\partial\Omega)$.

\item We let $H^n$ denote $n$-dimensional Hausdorff measure, and let
$\sigma := H^n\big|_{\partial\Omega}$ denote the ``surface measure'' on $\partial \Omega$.

\item For a Borel set $A\subset \ree$, we let $1_A$ denote the usual
indicator function of $A$, i.e. $1_A(x) = 1$ if $x\in A$, and $1_A(x)= 0$ if $x\notin A$.

\item For a Borel set $A\subset \ree$,  we let $\interior(A)$ denote the interior of $A$.
If $A\subset \partial\Omega$, then $\interior(A)$ will denote the relative interior, i.e., the largest relatively open set in $\partial\Omega$ contained in $A$.  Thus, for $A\subset \partial\Omega$,
the boundary is then well defined by $\partial A := \overline{A} \setminus {\rm int}(A)$.

\item For a Borel set $A$, we denote by $\mathcal{C}(A)$ the space of continuous functions on
$A$, by $\mathcal{C}_c(A)$ the subspace of $\mathcal{C}(A)$
with compact support in $A$, and by $\mathcal{C}_b(A)$ the
space of bounded continuous functions on $A$.  If $A$ is unbounded, we denote by
$\mathcal{C}_0(A)$ the space of continuous functions on $A$ converging to $0$ at infinity.

\item For a Borel subset $A\subset\partial\Omega$, we
set $\fint_A f d\sigma := \sigma(A)^{-1} \int_A f d\sigma$.

\item We shall use the letter $I$ (and sometimes $J$)
to denote a closed $(n+1)$-dimensional Euclidean cube with sides
parallel to the co-ordinate axes, and we let $\ell(I)$ denote the side length of $I$.
We use $Q$ to denote a dyadic ``cube''
on $\partial \Omega$.  The
latter exist, given that $\partial \Omega$ is ADR  (cf. \cite{DS1}, \cite{Ch}), and enjoy certain properties
which we enumerate in Lemma \ref{lemmaCh} below.

\end{list}

\begin{definition} ({\bf Corkscrew condition}).  \label{def1.cork}
Following
\cite{JK}, we say that a domain $\Omega\subset \ree$
satisfies the ``Corkscrew condition'' if for some uniform constant $c>0$ and
for every surface ball $\Delta:=\Delta(x,r),$ with $x\in \partial\Omega$ and
$0<r<\diam(\partial\Omega)$, there is a ball
$B(X_\Delta,cr)\subset B(x,r)\cap\Omega$.  The point $X_\Delta\subset \Omega$ is called
a ``Corkscrew point'' relative to $\Delta.$  We note that  we may allow
$r<C\diam(\pom)$ for any fixed $C$, simply by adjusting the constant $c$.
\end{definition}

\begin{remark}\label{remark1.2}
We note that, on the other hand, every $X\in\Omega$, with $\delta(X)<\diam(\pom)$,
may be viewed as a Corkscrew point,
relative to some surface ball $\Delta\subset\pom$.
Indeed, set $r=K \delta(X)$,  with $K>1$, fix
$x\in\pom$ such that $|X-x|=\delta(X)$, and let
$\Delta:=\Delta(x,r)$.
\end{remark}

\begin{definition}({\bf Harnack Chain condition}).  \label{def1.hc} Again following \cite{JK}, we say
that $\Omega$ satisfies the Harnack Chain condition if there is a uniform constant $C$ such that
for every $\rho >0,\, \Lambda\geq 1$, and every pair of points
$X,X' \in \Omega$ with $\delta(X),\,\delta(X') \geq\rho$ and $|X-X'|<\Lambda\,\rho$, there is a chain of
open balls
$B_1,\dots,B_N \subset \Omega$, $N\leq C(\Lambda)$,
with $X\in B_1,\, X'\in B_N,$ $B_k\cap B_{k+1}\neq \emptyset$
and $C^{-1}\diam (B_k) \leq \dist (B_k,\partial\Omega)\leq C\diam (B_k).$  The chain of balls is called
a ``Harnack Chain''.
\end{definition}

We remark that the Corkscrew condition is a quantitative, scale invariant
version of the fact that $\Omega$ is open, and the Harnack Chain condition is a
scale invariant version of path connectedness.


\begin{definition}\label{def1.ADR}
({\bf Ahlfors-David regular}). We say that a closed set $E \subset \ree$ is $n$-dimensional ADR (or simply ADR) (``Ahlfors-David regular'') if
there is some uniform constant $C$ such that
\begin{equation} \label{eq1.ADR}
\frac1C\, r^n \leq H^n(E\cap B(x,r)) \leq C\, r^n,\,\,\,\forall r\in(0,R_0),x \in E,\end{equation}
where $R_0$ is the diameter of $E$ (which may be infinite).   When $E=\partial \Omega$,
the boundary of a domain $\Omega$, we shall sometimes for convenience simply
say that ``$\Omega$ has the ADR property'' to mean that $\partial \Omega$ is ADR.
\end{definition}

\begin{definition}\label{def1.UR}
({\bf Uniform Rectifiability}). Following
David and Semmes \cite{DS1,DS2}, we say
that a closed set $E\subset \ree$ is $n$-dimensional UR (or simply UR) (``Uniformly Rectifiable''), if
it satisfies the ADR condition \eqref{eq1.ADR}, and if for some uniform constant $C$ and for every
Euclidean ball $B:=B(x_0,r), \, r\leq \diam(E),$ centered at any point $x_0 \in E$, we have the Carleson measure estimate
\begin{equation}\label{eq1.sf}\dint_{B}
|\nabla^2 \mathcal{S}1(X)|^2 \,\dist(X,E) \,dX \leq C r^n,
\end{equation}
where $\mathcal{S}f$ is the single layer potential of $f$, i.e.,
\begin{equation}\label{eq1.layer}
\mathcal{S}f(X) :=c_n\, \int_{E} |X-y|^{1-n} f(y) \,dH^n(y).
\end{equation}
Here, the normalizing constant $c_n$ is chosen so that
$\mathcal{E}(X) := c_n |X|^{1-n}$ is the usual fundamental solution for the Laplacian
in $\ree.$  When $E=\partial \Omega$,
the boundary of a domain $\Omega$, we shall sometimes for convenience simply
say that ``$\Omega$ has the UR property'' to mean that $\partial \Omega$ is UR.
\end{definition}

We note that there are numerous characterizations of uniform rectifiability
given in \cite{DS1,DS2};  the one stated above will be most useful for our purposes, and
appears in \cite[Chapter 3,  Part III]{DS2}.
We remark
that the UR sets
are precisely those for which all ``sufficiently nice'' singular integrals
are bounded on $L^2$ (see \cite{DS1}).

We recall that ``Uniform Rectifiability'' is the scale invariant analogue of
rectifiability;
in particular, using an idea of P. Jones
\cite{J}, one may derive, for UR sets, a quantitative version of the fact that rectifiability may be characterized
in terms of existence a.e. of approximate tangent planes.
For $x \in E$, $t >0$, we set
\begin{equation} \label{eq1.beta}
\beta _2(x,t)\equiv \inf_P\left(\frac{1}{t^{n}}
\int_{B(x,t) \cap E}\left(\frac{\mbox{dist} (y,P)}{t }\right)^2
dH^n (y)\right)^{1/2},\end{equation}
where the infimum runs over all $n$-planes $P.$
Then  a closed, ADR set $E$ is  UR if and only if the following Carleson measure estimate holds on
$E\times \mathbb{R}_+$:
\begin{equation} \label{eq1.UR}
\sup_{x_0 \in E,\, r \, >\, 0} r^{-n} \int_0^r \int_{B(x_0,t) \cap E}
\beta_2(x,t)^2 dH^n (x)
\frac{dt}{t} \, < \, \infty.
\end{equation}
Again see \cite{DS1} for details.

\begin{definition}\label{def1.bp} {\bf (``Interior Big Pieces'')}.
Given a domain $\Omega\subset \ree$, with ADR boundary, and a collection $\mathcal{S}$ of domains in $\ree$, we say that
$\Omega$ has ``interior big pieces of $\mathcal{S}\,$'' (denoted $\Omega\in IBP(\mathcal{S})$)
if there are constants $\alpha>0,\,K>1$ such that for
every $X\in\Omega$, with $\delta(X)<\diam(\pom)$, there is a point $x\in \pom$, with
$|x-X|=\delta(X)$,
and a domain $\Omega' \in \mathcal{S}$
for which, with $r:=K\delta(X)$, we have
\begin{enumerate}
\item $\Omega'\subset\Omega$.
\smallskip
\item $H^n(\pom'\cap \Delta(x,r)) \geq \alpha H^n(\Delta(x,r)) \approx \alpha r^n.$
\smallskip
\item $X$ is  a Corkscrew point for $\Omega'$, relative to $\Delta_\star(y,2r):=B(y,2r)\cap\pom'$, for some
$y\in\pom'\cap\Delta$ (we note that $X$ is also a Corkscrew point for $\Omega$,
relative to $\Delta$, by construction;  cf. Remark \ref{remark1.2}).
\end{enumerate}

\end{definition}

\begin{lemma}\label{lemmaCh}({\bf Existence and properties of the ``dyadic grid''})
\cite{DS1,DS2}, \cite{Ch}.
Suppose that $E\subset \ree$ satisfies
the ADR condition \eqref{eq1.ADR}.  Then there exist
constants $ a_0>0,\, \eta>0$ and $C_1<\infty$, depending only on dimension and the
ADR constants, such that for each $k \in \mathbb{Z},$
there is a collection of Borel sets (``cubes'')
$$
\mathbb{D}_k:=\{Q_{j}^k\subset E: j\in \mathfrak{I}_k\},$$ where
$\mathfrak{I}_k$ denotes some (possibly finite) index set depending on $k$, satisfying

\begin{list}{$(\theenumi)$}{\usecounter{enumi}\leftmargin=.8cm
\labelwidth=.8cm\itemsep=0.2cm\topsep=.1cm
\renewcommand{\theenumi}{\roman{enumi}}}

\item $E=\cup_{j}Q_{j}^k\,\,$ for each
$k\in{\mathbb Z}$.

\item If $m\geq k$ then either $Q_{i}^{m}\subset Q_{j}^{k}$ or
$Q_{i}^{m}\cap Q_{j}^{k}=\emptyset$.

\item For each $(j,k)$ and each $m<k$, there is a unique
$m$ such that $Q_{j}^k\subset Q_{i}^m$.

\item Diameter $\left(Q_{j}^k\right)\leq C_12^{-k}$.

\item Each $Q_{j}^k$ contains some ``surface ball'' $\Delta \big(x^k_{j},a_02^{-k}\big):=
B\big(x^k_{j},a_02^{-k}\big)\cap E$.

\item $H^n\left(\left\{x\in Q^k_j:{\rm dist}(x,E\setminus Q^k_j)\leq \tau \,2^{-k}\right\}\right)\leq
C_1\,\tau^\eta\,H^n\left(Q^k_j\right),$ for all $k,j$ and for all $\tau\in (0,a_0)$.
\end{list}
\end{lemma}

A few remarks are in order concerning this lemma.

\begin{list}{$\bullet$}{\leftmargin=0.4cm  \itemsep=0.2cm}

\item In the setting of a general space of homogeneous type, this lemma has been proved by Christ
\cite{Ch}.  In that setting, the
dyadic parameter $1/2$ should be replaced by some constant $\delta \in (0,1)$.
It is a routine matter to verify that one may take $\delta = 1/2$ in the presence of the Ahlfors-David
property (\ref{eq1.ADR}) (in this more restrictive context, the result already appears in \cite{DS1,DS2}).

\item  For our purposes, we may ignore those
$k\in \mathbb{Z}$ such that $2^{-k} \gtrsim {\rm diam}(E)$, in the case that the latter is finite.

\item  We shall denote by  $\mathbb{D}=\mathbb{D}(E)$ the collection of all relevant
$Q^k_j$, i.e., $$\mathbb{D} := \cup_{k} \mathbb{D}_k,$$
where, if $\diam (E)$ is finite, the union runs
over those $k$ such that $2^{-k} \lesssim  {\rm diam}(E)$.

\item Properties $(iv)$ and $(v)$ imply that for each cube $Q\in\mathbb{D}_k$,
there is a point $x_Q\in E$, a Euclidean ball $B(x_Q,r)$ and a surface ball
$\Delta(x_Q,r):= B(x_Q,r)\cap E$ such that
$r\approx 2^{-k} \approx {\rm diam}(Q)$
and \begin{equation}\label{cube-ball}
\Delta(x_Q,r)\subset Q \subset \Delta(x_Q,Cr),\end{equation}
for some uniform constant $C$.
We shall denote this ball and surface ball by
\begin{equation}\label{cube-ball2}
B_Q:= B(x_Q,r) \,,\qquad\Delta_Q:= \Delta(x_Q,r),\end{equation}
and we shall refer to the point $x_Q$ as the ``center'' of $Q$.

\item Let us now specialize to the case that  $E=\pom$, with $\Omega$ satisfying the Corkscrew condition.  Given $Q\in \mathbb{D}(\partial\Omega)$,
we
shall sometimes refer to a ``Corkscrew point relative to $Q$'', which we denote by
$X_Q$, and which we define to be the corkscrew point $X_\Delta$ relative to the ball
$\Delta:=\Delta_Q$ (cf. \eqref{cube-ball}, \eqref{cube-ball2} and Definition \ref{def1.cork}).  We note that
\begin{equation}\label{eq1.cork}
\delta(X_Q) \approx \dist(X_Q,Q) \approx \diam(Q).
\end{equation}
\end{list}

\begin{list}{$\bullet$}{\leftmargin=0.4cm  \itemsep=0.2cm}

\item For a dyadic cube $Q\in \mathbb{D}_k$, we shall
set $\ell(Q) = 2^{-k}$, and we shall refer to this quantity as the ``length''
of $Q$.  Evidently, $\ell(Q)\approx \diam(Q).$

\item For a dyadic cube $Q \in \mathbb{D}$, we let $k(Q)$ denote the ``dyadic generation''
to which $Q$ belongs, i.e., we set  $k = k(Q)$ if
$Q\in \mathbb{D}_k$; thus, $\ell(Q) =2^{-k(Q)}$.

\end{list}

\begin{definition} ({\bf $A_\infty$, $A_\infty^{\rm dyadic}$ and weak-$A_\infty$}).  \label{def1.ainfty}
Given a surface ball
$\Delta= B\cap\pom$, a Borel measure $\hm$ defined on $\partial \Omega$
is said to belong to the class $A_\infty(\Delta)$ if there are positive constants $C$ and $\theta$
such that for every $\Delta'=B'\cap\pom$ with $B'\subseteq B$,
and every Borel set $F\subset\Delta'$, we have
\begin{equation}\label{eq1.ainfty}
\hm (F)\leq C \left(\frac{\sigma(F)}{\sigma(\Delta')}\right)^\theta\,\hm (\Delta').
\end{equation}
If we replace the surface balls $\Delta$
and $\Delta'$  by a dyadic cube $Q$
and its dyadic subcubes $Q'$, with $F\subset Q'$,
then we say that $\hm\in A_\infty^{\rm dyadic}(Q)$:
\begin{equation}\label{eq1.ainftydyadic}
\hm (F)\leq C \left(\frac{\sigma(F)}{\sigma(Q')}\right)^\theta\,\hm (Q').
\end{equation}
Similarly, $\hm \in$ weak-$A_\infty(\Delta)$, with $\Delta =B\cap\pom$, if for every
$\Delta'=B'\cap\pom$ with $2B'\subseteq B$,
we have
\begin{equation}\label{eq1.wainfty}
\hm (F) \leq C \left(\frac{\sigma(F)}{\sigma(\Delta')}\right)^\theta\,\hm (2 \Delta')
\end{equation}
\end{definition}
As is well known  \cite{CF}, \cite{GR}, \cite{Sa},
the $A_\infty$ (resp. weak-$A_\infty$) condition is equivalent to
the property that the measure
$\hm$ is absolutely continuous with respect to $\sigma$, and that its density
satisfies a reverse H\"older (resp. weak reverse H\"older) condition.
In this paper, we are interested in the case that $\hm = \omega^{X}$,
the harmonic measure with
pole at $X$.  In that setting, we let
$k^X:= d\omega^X/d\sigma$ denote
the Poisson kernel, so that \eqref{eq1.ainfty} is equivalent to the reverse H\"older estimate
\begin{equation}\label{eq1.RH}
\left(\fint_{\Delta'} \left(k^{X}\right)^q d\sigma\right)^{1/q} \leq C\fint_{\Delta'}k^{X}\,d\sigma
\,,
\end{equation}
for some $q>1$
and for some uniform constant $C$.
In particular, when $\Delta'=\Delta$, and $X= X_\Delta$, a Corkscrew point relative to $\Delta$,
the latter estimate reduces to
\begin{equation}\label{eq1.lq}
\int_{\Delta} \left(k^{X_\Delta}\right)^q d\sigma \leq C\, \sigma(\Delta)^{1-q}.
\end{equation}
Similarly, \eqref{eq1.wainfty} is equivalent to
\begin{equation}\label{eq1.WRH}
\left(\fint_{\Delta'} \left(k^{X}\right)^q d\sigma\right)^{1/q} \leq
C\fint_{2\Delta'}k^{X}\,d\sigma
\,.
\end{equation}
Assuming that the latter bound holds with $\Delta'=\Delta$, and with $X=X_\Delta$,
then one again obtains \eqref{eq1.lq}.

\subsection{Statement of the Main Results}

Our main results are as follows.  We shall use the terminology that a connected open set
$\Omega\subset \ree$ is a {\bf 1-sided} NTA domain if it satisfies {\bf interior} (but not necessarily
exterior) Corkscrew and Harnack Chain conditions\footnote{We recall that such
domains are sometimes denoted
``uniform" domains in the literature, but we prefer the terminology ``1-sided NTA", both because it is more descriptive of the actual properties enjoyed by such domains, and to avoid confusion with
the completely different notion of ``uniform rectifiability".}.

\begin{theorem}\label{theor-main-I}  Let $\Omega \subset \ree,\,n\geq2,$
be a 1-sided NTA domain whose boundary
$\partial\Omega$ is $n$-dimensional UR.  Then
for each surface ball $\Delta$, the harmonic measure $\omega^{X_\Delta}$
belongs to weak-$A_\infty(\Delta),$ with uniform weak-$A_\infty$ constants depending only on
dimension and on the constants in the ADR, UR, Corkscrew and Harnack Chain conditions.
\end{theorem}

We emphasize again that we impose no hypothesis
(as in \cite{JK}, \cite{S}, \cite{DJe}) on the geometry of the exterior domain,
nor do we assume as in \cite{BL} that
the boundary has ``Big Pieces''  of boundaries of Lipschitz subdomains
of $\Omega$.

We shall also obtain a certain ``self-improvement'' of Theorem \ref{theor-main-I}, in which the hypotheses
are assumed to hold only in an appropriate ``big pieces'' sense.

\begin{theorem}\label{theor-main-I:BP}  Let $\Omega \subset \ree,\,n\geq2,$
be a connected open set whose boundary
$\partial\Omega$ is $n$-dimensional ADR.  Suppose further that
$\Omega \in IBP(\mathcal{S})$ (cf. Definition \ref{def1.bp}), where
$\mathcal{S}$ is a collection of 1-sided NTA domains with UR boundaries,
with uniform control of all of the relevant Corkscrew, Harnack Chain, ADR and UR constants.
Then for each surface ball $\Delta =\Delta (x,r)$, and for every $X\in\Omega \setminus
B(x,r)$,  the harmonic measure $\hm^X$ belongs to
weak-$A_\infty(\Delta),$ with uniform weak-$A_\infty$ constants that depend only on
dimension, on the constants in the ADR and interior big pieces conditions,
and on the relevant constants for the subdomains.
\end{theorem}

\noindent{\it Remark}.  We note that in Theorem \ref{theor-main-I:BP}, we have obtained
that $\hm^X$ belongs to weak-$A_\infty(\Delta(x,r))$, for all $X \in \Omega\setminus B(x,r)$.
In the presence of the Harnack Chain Condition, as in Theorem \ref{theor-main-I},
one may obtain the same conclusion for $X=X_\Delta$, the Corkscrew point relative to $\Delta$.
On the other hand, in Theorem \ref{theor-main-I}, we of course also obtain
that $\hm^X$ belongs to weak-$A_\infty(\Delta(x,r))$, for all $X \in \Omega\setminus B(x,r)$.

In a companion paper to this one \cite{HMU}, we shall establish the  converse
to Theorem \ref{theor-main-I}:

\begin{theorem}\label{th2.2} Let $\Omega \subset \ree,\,n\geq2,$
be a 1-sided NTA domain, whose boundary
is $n$-dimensional ADR.   Suppose also that  harmonic measure
$\omega$ is absolutely continuous with respect to surface measure
and that there is a $q>1$ such that
for every surface ball $\Delta=\Delta(x,r)$ with radius  $r\lesssim \diam \partial\Omega$,
the Poisson kernel satisfies the scale invariant  estimate
\eqref{eq1.lq}.
Then $\partial \Omega$ is UR.
\end{theorem}

We also mention that in \cite{HMU} we obtain a ``big pieces'' version of the previous result in the following sense. Let $E\subset\ree$ be a closed set and assume that $E$ is $n$-dimensional ADR. Assume that there exists $q>1$ such that $E$ has ``big  pieces of boundaries of $\mathcal{S}$''
(i.e., for every surface ball $B(x,r)\cap E$ there is $\Omega'\in\mathcal{S}$ whose boundary
has  an ``ample'' contact with $E\cap B(x,r)$), where $\mathcal{S}$ is a collection of domains $\Omega'$ each of them satisfying the hypotheses of Theorem \ref{th2.2} (with $q$ fixed) and with uniform control on the relevant constants. Then $E$ is UR. See \cite{HMU} for the precise statement.

\bigskip

\noindent{\bf Acknowledgements}.  The first named author wishes to thank
John Lewis for  helpful comments concerning the paper \cite{BL}.
He also thanks Misha Safonov and Tatiana Toro for bringing to our attention the work of Aikawa \cite{Ai1}, \cite{Ai2}.

\section{Outline of the Strategy of the Proof}

Let us sketch the strategy of the proofs of Theorems \ref{theor-main-I}
and \ref{theor-main-I:BP}.  We shall do most of our analysis
in certain approximating domains which enjoy additional qualitative properties.  Given these qualitative properties, we shall prove some {\it a priori} estimates for the Green function $G$ and for harmonic measure $\hm$,
beginning with Lemma \ref{lemma2.cfms} in Section \ref{section-fund-est}, whose proofs rely on being able to ``hide'' certain small quantities, which must therefore be known in advance to be finite.
An interesting feature of these {\it a priori} estimates is that they permit us to deduce
the doubling property for $\hm$, as well as
a comparison principle for $G$,
in the absence of an exterior disk or Corkscrew condition (the exterior conditions
enable one to prove boundary H\"older continuity of solutions vanishing on a surface ball).  We obtain these properties
for $G$ and $\hm$ without establishing boundary H\"older continuity.
We note that, by the work of Aikawa \cite{Ai1}, \cite{Ai2}, some of the preliminary estimates that we prove in Section \ref{section-fund-est}, in particular, the ``Carleson estimate" Lemma \ref{lemma2.carleson}, and the Comparison Principle (aka ``Boundary Harnack Principle") Lemma \ref{lemma2.comparison},
are known, but we include our own relatively short proofs here for the sake of self-containment.

We also establish several geometric preliminaries as follows.
In Section \ref{spoincare}, we use the
Harnack Chain property  to prove a Poincar\'e inequality
(Lemma \ref{poincare}),
which we use in turn, in Section \ref{s5*}, to obtain a criterion for the existence of
exterior Corkscrew points in the complement of certain ``sawtooth'' regions
(Lemma \ref{lemma4.main}).  This criterion stipulates that
the Carleson measure (cf. \eqref{eq1.sf})
\begin{equation}\label{eq2.1*}|\nabla^2\mathcal{S}1(X)|^2\dist(X,\partial\Omega) dX
\end{equation}
be {\bf sufficiently small} in the relevant sawtooth region.
We then present in Section \ref{s6*} a variant of the ``sawtooth lemma'' of \cite{DJK}
(Lemma \ref{lemma:DJK-dyadic-proj}), which
roughly speaking allows
for a comparison, in the sense of $A_\infty$, between the respective
harmonic measures, $\hm$ and $\hm_{\Omega_{\F}}$,
for the original domain and for the sawtooth domain
(more precisely, our version of the sawtooth lemma
allows us to transfer the dyadic $A_\infty$ property of $\hm_{\Omega_{\F}}$
to  $\P_\F\omega$, where $\P_\F$ is a sort of ``conditional expectation'' projection operator,
with respect to some collection $\F$ of non-overlapping dyadic cubes
from which the sawtooth was constructed).  The arguments of Section \ref{s6*} are an extension, to the present context, of our previous work in the Euclidean setting \cite{HM}.

With these preliminary matters in hand, we proceed to the heart of our proof,
which will exploit the technique of ``extrapolation (i.e., bootstrapping) of Carleson measures'',
as it appears in our previous work \cite{HM} (see also \cite{HM-ANU}), but originating
in \cite{CG} and \cite{LM}.   We now describe the application of this technique in our setting.
By a Corona type stopping time construction delineated in Section \ref{scorona},
plus an induction scheme (formalized in Lemma \ref{lemma:extrapol}),
we reduce matters to verifying that $\P_\F\hm$
(that is, the projection of harmonic measure mentioned above)
enjoys the dyadic $A_\infty$ property, in sawtooth domains $\Omega_\F$
in which the Carleson measure \eqref{eq2.1*} has {\bf sufficiently small} Carleson norm.
In turn, we establish this property for $\P_\F\hm$, by using the preliminary facts noted above:
by the smallness of \eqref{eq2.1*} in the sawtooth, we deduce
that the complement of the sawtooth
enjoys an exterior Corkscrew condition. Thus, we may apply the results of \cite{DJe} to the sawtooth,
to obtain that $\hm_{\Omega_{\F}}$, the harmonic measure for the sawtooth domain,
belongs to $A_\infty$ with respect to surface measure
on the boundary of the sawtooth. Then, invoking our version of the sawtooth lemma, we find that
$\P_\F\hm$ belongs to dyadic $A_\infty$, as desired.   The ``extrapolation'' technology
(i.e., Lemma \ref{lemma:extrapol}) now allows us to conclude that $\hm$ belongs to $A_\infty$
with respect to surface measure, in a local, but scale invariant way.  However, at this point,
we have only reached this conclusion in our approximating domains $\Omega_N$,
albeit with $A_\infty$ constants independent of $N$.  Here $\{\Omega_N\}$ is
a nested increasing sequence of sub-domains of $\Omega$, each of which enjoys the qualitative properties mentioned above, such that $\Omega_N\nearrow\Omega$.  It is not clear
whether the $A_\infty$ property of harmonic measure, or even the doubling property, are transmitted
in the limit to harmonic measure on $\Omega$.  However, a maximum principle argument
(in the case of Theorem \ref{theor-main-I:BP}, there are two separate maximum principle arguments)
allows us to transfer, at least, the property that there are uniform constants $c_0,
\eta\in (0,1)$ such that for any Borel subset  $A\subset \Delta$,
\begin{equation}\label{*}\tag{$*$}
\sigma(A)>\eta\,\sigma(\Delta)\,\,\,\implies\,\, \,\omega^{X_\Delta}(A) \geq c_0\,.
\end{equation}
The fact that  \eqref{*} holds, in the absence of assumptions on the exterior domain $\Omega_{ext}$
or on $\partial\Omega$ (beyond UR), is really the main result of this paper.
Given \eqref{*}, we obtain the conclusion of Theorems \ref{theor-main-I} and \ref{theor-main-I:BP}
by invoking the arguments of \cite{BL}.

\section{Some fundamental estimates}\label{section-fund-est}
In this section we recall or establish certain fundamental estimates for
harmonic measure and the Green function.  In the sequel,
$\Omega \subset \ree,\, n\geq 2,$ will be a connected, open set, $\omega^X$ will denote
harmonic measure for $\Omega$, with pole at $X$, and $G(X,Y)$
will be the Green function.
At least in the case that $\Omega$ is bounded, we may, as usual,
define $\omega^X$ via the maximum principle and the Riesz representation theorem,
after first using the method of
Perron (see, e.g., \cite[pp. 24--25]{GT}) to construct a harmonic function ``associated'' to arbitrary continuous boundary
data.\footnote{Since we have made no assumption as regards
Wiener's regularity criterion, our harmonic function is a generalized solution, which
may not be continuous up to the boundary.}  For unbounded $\Omega$, we may still
define harmonic measure via a standard approximation scheme as follows.
Given $R>0$, set $\Omega_R := \Omega \cap B(x_0,2R)$, where
$x_0$ is a fixed point on $\partial \Omega$.  Define a smooth
cut-off function $\eta\in \C^\infty_0([-2,2])$, with $0\leq \eta \leq 1$, $\eta\equiv 1$ on $[-1,1]$,
and $\eta$ monotone decreasing on $(1,2)$ and monotone increasing on $(-2,-1).$
Suppose now that $0\leq f \in \mathcal{C}_b(\partial\Omega)$ and set
\begin{equation}\label{eq2.fr}f_R(x):= f(x)\, \eta\left(\frac{|x-x_0|}{R}\right).\end{equation}
Extending $f_R$ to be zero
outside of its support defines a continuous function on
$\partial \Omega_R$, so we may construct the corresponding
Perron solution $u_R$ in $\Omega_R$.  By the maximum principle,
$$
u_R\leq u_{R'} \mbox{ in }\Omega_R, \,\, {\rm if} \,\, R'>R,\qquad
{\rm and} \qquad\sup_{\Omega_R} u_R \leq
\sup_{\partial\Omega_R} f_R\leq \sup_{\partial\Omega} f.$$
Consequently, by Harnack's convergence theorem (\cite[p. 22]{GT}),
there is a harmonic function $u$
in $\Omega$ such that
\begin{equation}\label{eq2.limit}\lim_{R\to\infty}u_R = u\,,\end{equation}
with the convergence being uniform on compacta in $\Omega.$
Moreover, $u$ satisfies the maximum principle
$$\sup_\Omega u \leq \sup_{\partial\Omega} f.$$
Thus, we may again define harmonic measure $\omega^X$ for $X\in\Omega$ via the
Riesz representation theorem.  We note for future reference that $\omega^X$
is a non-negative, finite Borel measure which satisfies the outer regularity property
\begin{equation}\label{eq2.outer}
\omega^X(A) := \inf_{A\subseteq O} \omega^X(O),
\end{equation}
for every Borel set $A\subset \partial\Omega$, where the infimum runs over all
(relatively) open $O \subset \partial \Omega$ containing $A$.

The Green function may now be constructed
by setting
\begin{equation}\label{eq2.greendef}
G(X,Y):= \mathcal{E}\,(X-Y) - \int_{\partial\Omega}\mathcal{E}\,(X-z)\,d\omega^Y(z),
\end{equation}
where $\mathcal{E}\,(X):= c_n |X|^{1-n}$ is the usual fundamental solution for the Laplacian
in $\ree$.  We choose the normalization that makes $\mathcal{E}$ positive.
Given this normalization,
we shall also have that $G\geq 0$ (cf. Lemma \ref{lemma2.green} below.)

Before  proceeding further,
let us note one more fact for future reference.  Assuming that $\Omega$ is unbounded,
and using the notation above,
let $\omega^X_R$ and $G_R(X,Y)$ denote, respectively, harmonic measure and Green's function for
the approximating domain $\Omega_R$.  We then have
\begin{equation}\label{eq2.converge}
\lim_{R\to\infty} G_R(X,Y) = G(X,Y),
\end{equation}
with the convergence being uniform on compacta in $\Omega$,
in the $Y$ variable with $X \in \Omega$ fixed.
Indeed,  fixing $X$,  choosing $R$ so large that $R>>|X-x_0|$, and setting
$f:= \mathcal{E}\,(X-\cdot)$, with $f_R$ defined as in \eqref{eq2.fr}, we have that
$$ \int_{\partial\Omega_R} f\, d\omega^Y_R = \int_{\partial\Omega_R} f_R \,d\omega^Y_R
\,+\, O(R^{1-n})  :=u_R(Y) + O(R^{1-n}).$$
We then obtain \eqref{eq2.converge} immediately from \eqref{eq2.limit}
and the definition of the Green function \eqref{eq2.greendef}.

\begin{lemma}[Bourgain \cite{B}]\label{Bourgainhm}  Suppose that
$\partial \Omega$ is $n$-dimensional ADR.  Then there are uniform constants $c\in(0,1)$
and $C\in (1,\infty)$,
such that for every $x \in \partial\Omega$, and every $r\in (0,\diam(\partial\Omega))$,
if $Y \in \Omega \cap B(x,cr),$ then
\begin{equation}\label{eq2.Bourgain1}
\omega^{Y} (\Delta(x,r)) \geq 1/C>0 \;.
\end{equation}
In particular, if $\Omega$ satisfies the Corkscrew and Harnack Chain conditions,
then for every surface ball $\Delta$, we have
\begin{equation}\label{eq2.Bourgain2}
\omega^{X_\Delta} (\Delta) \geq 1/C>0 \;.
\end{equation}
\end{lemma}
We refer the reader to \cite[Lemma 1]{B} for the proof.

We next introduce some notation. We say that a domain $\Omega$ satisfies the \textbf{qualitative exterior
Corkscrew} condition if there exists $N\gg 1$ such that $\Omega$ has exterior corkscrew points at all scales smaller than $2^{-N}$. That is, there exists a constant $c_N$ such that for every surface ball $\Delta=\Delta(x,r)$, with $x\in \pom$  and $r\le 2^{-N}$, there is a ball $B(X_\Delta^{ext},c_N\,r)\subset B(x,r)\cap\Omega_{ext}$.

Given a ball $B_0$ centered on $\pom$, and $X\in\Omega\setminus B_0$, we also introduce the quantity
\begin{equation}\label{defi-upsilon}
\Upsilon_{B_0}(X):=\sup_{B:2B\subseteq B_0} \frac{r_\Delta^{1-n} \omega^X(\Delta)}{G(X_\Delta,X)}.
\end{equation}
where the sup runs over all the balls $B$ centered at $\pom$ with $2B\subseteq B_0$ and where as usual $\Delta=B\cap\pom$. We also set $\|\Upsilon_{B_0}\|=\sup_{X\in\Omega\setminus B_0}\Upsilon_{B_0}(X)$.  The quantity $\Upsilon_{B_0}$ will enter in the proof of Lemma \ref{lemma2.cfms} below.

\begin{remark}\label{remark:qualitative}
Let us observe that if $\Omega$ satisfies the qualitative exterior Corkscrew condition,  then every point in $\pom$ is regular in the sense of Wiener. Moreover, for 1-sided NTA domains, the qualitative
exterior Corkscrew points allow local H\"older continuity
at the boundary (albeit with bounds which may depend badly on $N$),
so that the program of \cite{JK} may be followed to prove that $\Upsilon_{B_0}(X)$ is {\it a priori} finite (possibly depending on $N, X$ and $B_0$).
Eventually, we shall apply Lemmas \ref{lemma2.green} and \ref{lemma2.cfms} below (and several related
lemmas and corollaries) to certain approximating domains $\Omega_N$ which will inherit the stated quantitative hypotheses from the original domain $\Omega$,
but which also satisfy the qualitative exterior corkscrew conditions for scales $\lesssim 2^{-N}$.
We emphasize that all of the {\it quantitative} bounds that we shall establish
will depend only upon dimension and on the
parameters in the 1-sided NTA and UR (including ADR) conditions,
and thus these bounds will hold uniformly for the entire family of approximating domains.
\end{remark}

\begin{lemma} \label{lemma2.green} There are positive, finite constants $C$, depending only on dimension,
and $c(n,\theta)$, depending on dimension and $\theta \in (0,1),$
such that the Green function satisfies
\begin{eqnarray}\label{eq2.green}
&G(X,Y) \leq C\,|X-Y|^{1-n}\\[4pt]\label{eq2.green2}
& c(n,\theta)\,|X-Y|^{1-n}\leq G(X,Y)\,,\quad {\rm if } \,\,\,|X-Y|\leq \theta\, \delta(X)\,, \,\, \theta \in (0,1)\,.
\end{eqnarray}
Moreover, if every point on $\pom$  is regular in the sense of Wiener, then
\begin{equation}
\label{eq2.green3}
G(X,Y)\geq 0\,,\qquad \forall X,Y\in\Omega\,,\, X\neq Y;
\end{equation}
\begin{equation}\label{eq2.green4}
G(X,Y)=G(Y,X)\,,\qquad \forall X,Y\in\Omega\,,\, X\neq Y;
\end{equation}
and
\begin{equation}\label{eq2.14}
\int_{\partial\Omega} \Phi\,d\omega^X = -\iint_\Omega
\nabla_Y G(Y,X) \cdot\nabla\Phi(Y)\, dY,
\end{equation}
for every $X\in\Omega$ and $\Phi \in C_0^\infty(\ree)$ with
$\Phi (X)=0$.
\end{lemma}

\begin{proof}  Some of these facts are standard, but we include the
simple proof here.    Recall
that we have chosen the normalization $\mathcal{E}\,(X):= c_n |X|^{1-n}$ with $c_n >0$.
Inequality  \eqref{eq2.green} is then trivial, by definition \eqref{eq2.greendef},
since $\int_{\partial\Omega}  \mathcal{E}\,(X-z)\,d\omega^Y(z)\geq0.$
We now consider \eqref{eq2.green2}.
Suppose that $0<\theta<1,$ and
that $|X-Y|\leq\theta\, \delta(X)$.
Then,
$$\int_{\partial\Omega}
|X-z|^{1-n}\,d\omega^Y(z)\,\leq\, \delta(X)^{1-n}\,\leq \,\theta^{\,n-1}\,|X-Y|^{1-n}.$$
Thus, $G(X,Y) \geq c_n (1-\theta^{\,n-1})\,|X-Y|^{1-n},$ as desired.

We now assume that every boundary point is regular in the sense of Wiener.  Let us prove
\eqref{eq2.green3}.
Suppose first that $\Omega$ is bounded.
Fix $X\in \Omega$, and observe that by \eqref{eq2.green2},
it is enough to consider the case that $Y\in \Omega':=\Omega\setminus \overline{B(X,\delta(X)/2)}$.
Moreover, by \eqref{eq2.green2}, we have in particular that $G(X,\cdot)>0$ on
$\partial B(X,\delta(X)/2)$.  On the other hand, since every boundary point is regular,
we have by definition \eqref{eq2.greendef}
that  $G(X,\cdot)\equiv 0$ on $\pom$.
Applying the maximum principle in $\Omega'$, we then obtain \eqref{eq2.green3},
at least when $\Omega$ is bounded.
If $\Omega$ is unbounded, we may invoke \eqref{eq2.converge}.

Next, we establish the symmetry condition \eqref{eq2.green4},
again assuming that
every boundary point is regular in the sense of Wiener.
By  \eqref{eq2.converge},
it is enough to treat the case that $\Omega$ is bounded.  Specializing to the case of the Laplacian,
the Green function constructed in \cite{GW}, which we denote temporarily by $\G(X,Y)$,
is symmetric (see \cite[Theorem 1.3]{GW}).  Therefore, it is enough to verify that our Green function is the same as the one constructed in \cite{GW}.  To this end, we first recall that by \cite[Theorem 1.1]{GW}
$\G$ is unique among all those real valued, non-negative functions defined on
$\Omega\times\Omega\setminus \{(X,Y)\in \Omega\times\Omega: X=Y\}$, such that
for each $X\in\Omega$ and $r>0$,
\begin{eqnarray}\label{eq3.green}
&\G(X,\cdot)\in W^{1,2}\big(\Omega\setminus B(X,r)\big)\cap W^{1,1}_0(\Omega)
\\[4pt]\label{eq3.green2}
& \iint_\Omega \nabla_Y \G(X,Y)\cdot\nabla \phi(Y) \,dY = \phi(X)\,,\qquad \forall \phi\in C^\infty_0(\Omega).
\end{eqnarray}
It is clear that \eqref{eq3.green2} holds for our Green function $G(X,Y)$, by Definition
\ref{eq2.greendef}.   Thus, we need only show that $G$ satisfies \eqref{eq3.green}.
As in \cite[p. 5]{Ke}, for $X\in \Omega$ fixed, we may construct $v(X,\cdot)$, the variational solution
to the Dirichlet problem with data $\E(X-\cdot)$.  In particular, $v(X,\cdot)\in W^{1,2}(\Omega)$.
Since $\E(X-\cdot)$ is Lipschitz on $\pom$,
and since every point on $\pom$ is Wiener regular, it follows as in \cite[p. 5]{Ke} that
$v(X,\cdot) \in C(\overline{\Omega})$, and therefore
\begin{equation}\label{eq3.vdef}
v(X,Y):=\int_{\partial\Omega}\mathcal{E}\,(X-z)\,d\omega^Y(z)
\end{equation}
(see, e.g. \cite{GT}, p. 25).  Thus, $G(X,Y) =  \E(X-Y) - v(X,Y)$ (cf.
\eqref{eq2.greendef}), and
since $v\in W^{1,2}(\Omega)$, we obtain \eqref{eq3.green}.

Finally we verify \eqref{eq2.14}.
We begin by reducing matters to
the case that  $\Omega$ is bounded.  Indeed, for the left hand side of \eqref{eq2.14},
we may pass immediately from the bounded to the unbounded case by
splitting $\Phi$ into positive and negative parts, and using
\eqref{eq2.limit}.  To pass to the limit on the right hand side is more delicate,
and we proceed as follows.  As above, given an unbounded domain $\Omega$,
let $G_R$ denote the Green function for
the domain $\Omega_R := \Omega \cap B(x_0,2R)$, for some fixed $x_0\in\pom$.
We claim that
\begin{equation}\label{eq3.converge+}
\lim_{R\to\infty}\iint_{\Omega_R} \nabla_Y G_R(Y,X) \cdot {\bf h}(Y) \, dY
=\iint_\Omega \nabla_Y G(Y,X) \cdot {\bf h}(Y) \, dY\,,
\end{equation}
for all Lipschitz vector-valued ${\bf h}$ with compact support in $\ree$.
Given the claim, and assuming that \eqref{eq2.14} holds for bounded $\Omega$,
we may then pass to the unbounded case by setting ${\bf h}=\nabla\Phi$.

Thus, to reduce the proof of \eqref{eq2.14} to the case that $\Omega$ is bounded,
it remains to prove \eqref{eq3.converge+}.
To this end, we first recall our previous observation that for bounded domains
with Wiener regular boundaries, our Green function is the same as that constructed
in \cite{GW}.  Thus, there is a purely dimensional constant $C_n$ such that
for every $R<\infty$, and $X\in\Omega_R$, $\nabla G_R(\cdot,X)$ enjoys the weak-$L^{(n+1)/n}$
estimate
$$
\left|\big\{Y\in\Omega_R: |\nabla_Y\,G_R(Y,X)|>\lambda\big\}\right | \leq C_n\,\lambda^{-(n+1)/n}\,.
$$
Consequently, if $A\subset\Omega_R$, we have that
\begin{equation}\label{eq3.21+}
\iint_A|\nabla_Y\,G_R(Y,X)|^p\,dY \leq C(n,p,|A|)\,,\qquad \forall\, p<(n+1)/n\,,
\end{equation}
as may be deduced from the weak-type inequality by arguing as in the proof of Kolmogorov's
lemma.   We emphasize that the constant in the last inequality depends only upon $n,p$ and $|A|$,
but not on $R$.  Let us now fix a ball $B_0:= B(X_0,R_0) \subset \ree$, and
consider a Lipschitz function ${\bf h}$ supported in $B_0$.
We note that
\begin{multline}\label{eq3.22+}
\iint_\Omega \nabla_Y G_R(Y,X) \cdot {\bf h}(Y) \, dY
=\iint_\Omega G_R(Y,X) \,\div{\bf h}(Y) \, dY\\[4pt]
\to
\iint_\Omega G(Y,X) \,\div{\bf h}(Y) \, dY\,,
\end{multline}
as $R\to\infty$,
where we have used first
that $G_R \in W^{1,1}_0(\Omega_R)$
(again, because $G_R$ coincides with
the \cite{GW} Green function), and then \eqref{eq2.converge}
(in $\Omega\cap B_0\cap\{\delta(Y)>\epsilon\}$), along with
\eqref{eq2.green} (to control small errors in the ``border strip"
$\Omega\cap B_0\cap\{\delta(Y)\leq\epsilon\}$).
Here we may suppose that $R_0\ll R$
so that $B_0 \cap\Omega \subset  \Omega_R$.  Let us now extend 
$\,G(\cdot,X)$ to be zero
in $\ree\setminus \overline{\Omega}$, and call this extension 
$\mathcal{G}$. 
Then from \eqref{eq3.21+} and \eqref{eq3.22+} it follows that
\begin{equation}\label{eq3.22++}
\left|\iint_{\ree} \mathcal{G}(Y,X) \,\div{\bf h}(Y) \, dY\right|\,\leq C(n,p,|B_0|)\,\|h\|_{p'}\,,\quad
1<p<(n+1)/n.
\end{equation}
Taking a supremum over all Lipschitz ${\bf h}$ supported in $B_0$, with $\|{\bf h}\|_{p'} = 1$,
we obtain that for $p\in(1,(n+1)/n)$,
\begin{equation}\label{eq3.22+++}
\nabla \mathcal{G}(\cdot,X) \in
L^p(B_0)\,,\quad  {\rm with}  \,\,\,
\|\nabla \mathcal{G}\|_{L^p(B_0\setminus\{X\})}\leq C(n,p,|B_0|)\,.
\end{equation}
Now let $\psi\in C^\infty(\RR)$, with $0\leq \psi\leq 1$,
$\psi(t)\equiv 0$ if $t\leq1$,
$\psi(t)\equiv 1$ if $t\geq2$.
We fix ${\bf h}$ as above, let $\epsilon>0$,
and set ${\bf h}_\epsilon(Y):= {\bf h}(Y)\psi(\delta(Y)/\epsilon)$.
Then, by  \eqref{eq2.converge},
\begin{multline}\label{eq3.23++}
\iint_\Omega \nabla_Y G_R(Y,X) \cdot {\bf h}_\epsilon(Y) \, dY
=\iint_\Omega G_R(Y,X) \,\div{\bf h}_\epsilon(Y) \, dY\\[4pt]
\to
\iint_\Omega G(Y,X) \,\div{\bf h}_\epsilon(Y) \, dY=
\iint_\Omega \nabla_Y G(Y,X) \cdot {\bf h}_\epsilon(Y) \, dY\, ,
\end{multline}
as $R\to\infty$.   Also, by \eqref{eq3.21+}, \eqref{eq3.22+++}
and H\"older's inequality, for  $1<p<(n+1)/n$, we have
\begin{multline}\label{eq3.24+}
\left|\iint_\Omega \nabla_YG_R(Y,X) \cdot ({\bf h}-{\bf h}_\epsilon)(Y) \, dY\right|
+\left|\iint_\Omega \nabla_YG(Y,X) \cdot ({\bf h}-{\bf h}_\epsilon)(Y) \, dY\right|\\[4pt]
\leq C(n,p,|B_0|) \,
\|{\bf h}\|_\infty\,\left|\big\{Y\in \Omega\cap
B_0: \delta(Y)<2\epsilon\big\}\right|^{1/p'}\to 0\,,
\end{multline}
as $\epsilon \to 0$, uniformly in $R$.   Then by  \eqref{eq3.23++}-\eqref{eq3.24+}, we have that
\begin{multline}\label{eq3.25+}
\lim_{R\to\infty}\iint_\Omega \nabla_Y G_R(Y,X) \cdot {\bf h}(Y) \, dY=\\[4pt]
\lim_{R\to\infty}\iint_\Omega \nabla_Y G_R(Y,X) \cdot {\bf h}_\epsilon(Y) \, dY\,+\,o(1)
\,=\,\iint_\Omega \nabla_Y G(Y,X) \cdot {\bf h}(Y) \, dY
+o(1),
\end{multline}
as $\epsilon \to 0$. Letting $\epsilon \to 0$, we obtain \eqref {eq3.converge+}.

We may now assume that  $\Omega$ is bounded,
and proceed to prove \eqref{eq2.14} in that case.
As above, Wiener regularity then guarantees
that a given Perron solution, with Lipschitz data, coincides with the corresponding
variational solution with the same data.  This is true for the function $v(X,Y)$ defined in
\eqref{eq3.vdef}, as well as for
$$u(X):= \int_{\pom} \Phi\, d\hm^X\,.$$
Thus, in particular, $u-\Phi \in W^{1,2}_0(\Omega)$,
and we claim that
\begin{equation}\label{eq3.w12}
\iint_\Omega
\nabla_Y G(Y,X) \cdot\big(\nabla u(Y)-\nabla\Phi(Y)\big)\, dY=u(X)-\Phi(X) = u(X)\,.
\end{equation}
If $u-\Phi$ were in $C^\infty_0(\Omega)$, the claim would follow immediately
from \eqref{eq3.green2}.   We will pass from $C^\infty_0(\Omega)$ to $W^{1,2}_0(\Omega)$
by a density argument, with a slight complication since the Green function
is not in $W^{1,2}$ near the pole.
To address this technical issue,  we multiply $(u-\Phi)$ by a smooth cut-off function supported in a small
neighborhood of the pole $X$. For the part near the pole we may invoke \eqref{eq3.green2}:
$u$ is harmonic, therefore smooth in $\Omega$ and $u$ times a smooth cut-off is $C^\infty_0(\Omega)$. For the part away from the pole we use \eqref{eq3.green} and \eqref{eq3.green2},
plus the routine density argument mentioned above.
We leave the details, which are standard,
to the reader.

At this point,  \eqref{eq2.14} follows immediately from  \eqref{eq3.w12} and the fact that
\begin{equation}\label{eq3.23+}
\iint_\Omega \nabla_Y G(Y,X) \,\nabla u(Y)\, dY = 0\,.
\end{equation}
In turn, we may verify the latter identity
as follows.
For $0<\epsilon\ll\delta(X)$,
set
$\phi_\epsilon (Y):= \phi\big((X-Y)/\epsilon\big)$, where $\phi\in C^\infty(\mathbb{R}^{n+1}),$
$\phi \equiv 1$ in $\mathbb{R}^{n+1}\setminus B(0,2)$,
$\phi \equiv 0$ in $B(0,1)$.  Then
\begin{multline*}\iint_\Omega
\nabla_Y G(Y,X) \cdot\nabla u(Y)\, dY \,= \,\iint_\Omega
\nabla_Y \Big(\mathcal{E}(Y-X)\,\phi_\epsilon(Y)-v(Y,X)\Big) \cdot\nabla u(Y)\, dY \\[4pt]+\,
\iint_\Omega
\nabla_Y \Big(\mathcal{E}(Y-X)\big(1-\phi_\epsilon(Y)\big)\Big) \cdot\nabla u(Y)\, dY=\, 0 + O(\epsilon)\,,
\end{multline*}
where in the vanishing term we have used the definition of weak solution, since
$\mathcal{E}(\cdot-X)\,\phi_\epsilon(\cdot)-v(\cdot,X) \in W^{1,2}_0(\Omega)$.
To obtain the $O(\epsilon)$ bound,
we have used standard estimates for the fundamental solution and its gradient, along with
the fact that $\nabla u$ is harmonic and therefore
locally bounded in $\Omega$.
Finally, we obtain \eqref{eq3.23+} by letting $\epsilon\to 0^+$.
\end{proof}

\begin{lemma} \label{lemma2.cfms}
Let $\Omega$ be a 1-sided NTA domain with $n$-dimensional ADR boundary, and suppose that every $x \in \partial \Omega$ is regular in the sense of Wiener. Fix
$B_0:=B(x_0,r_0)$  with $x_0\in\pom$, and $\Delta_0 := B_0\cap\partial\Omega$.
Let $B:=B(x,r)$,  $x\in \pom$, and
$\Delta:=B\cap\pom$, and suppose that $2B\subset B_0.$  Then for
$X\in \Omega\setminus B_0$ we have
\begin{equation}\label{eq2.CFMS1}
r^{n-1}G(X_\Delta,X) \leq C \omega^X(\Delta).
\end{equation}
If, in addition,  $\Omega$ satisfies the qualitative exterior corkscrew condition, then
\begin{equation}\label{eq2.CFMS2}
\omega^X(\Delta)\leq C r^{n-1}G(X_\Delta,X).
\end{equation}
The constants in \eqref{eq2.CFMS1} and \eqref{eq2.CFMS2} depend {\bf only} on
dimension and on the constants in the $ADR$ and 1-sided NTA
conditions.
\end{lemma}

\noindent{\it Remark}.  Let us emphasize that in several results below we will assume that certain domains satisfy the hypotheses of Lemma \ref{lemma2.cfms}; by this we mean that the domains are 1-sided NTA with
$n$-dimensional ADR boundary, which moreover satisfy the qualitative exterior corkscrew condition (in particular then, every boundary point is regular in the sense of Wiener, cf. Remark \ref{remark:qualitative}).
Notice that in such a case \eqref{eq2.CFMS2} holds with $C$ depending {\bf only} on dimension and on the constants in the $ADR$ and 1-sided NTA conditions.  In particular,  $C$ does not depend on the parameter $N$ from the qualitative assumption. This will be crucial when applied to approximating domains.

\begin{proof}  The first estimate \eqref{eq2.CFMS1} may be obtained by a well known argument
(cf.  \cite{CFMS} or \cite[Lemma 1.3.3]{Ke}) using \eqref{eq2.Bourgain2} plus Harnack's inequality,
the upper bound for $G(X,Y)$ in \eqref{eq2.green}, and the maximum principle in
$\Omega\setminus B(X_\Delta, \delta(X_\Delta)/2)$
(the use of the maximum principle is justified even in the case that $\Omega$ is unbounded, by virtue of the decay of the Green function at infinity).   We omit the details.

The proof of the second estimate \eqref{eq2.CFMS2} will require a bit more work. In contrast to
the case of previous proofs of this estimate \cite{CFMS}, \cite{JK}, we do not use local H\"older continuity at the boundary for solutions vanishing on a surface ball
(since this depends on the parameter $N$ in our qualitative assumption).
Instead, we proceed as follows.   Fix $B_0$ centered on $\pom$, and $X\in \Omega\setminus B_0$, and write $\Upsilon_{B_0}=\Upsilon_{B_0}(X)$ (cf. \eqref{defi-upsilon}). As observed in Remark \ref{remark:qualitative},
$\Upsilon_{B_0}$ is {\it a priori} finite (possibly depending on $N$).  Thus,
it will suffice to show that $\Upsilon_{B_0} \leq C_\epsilon + C\epsilon \,\Upsilon_{B_0},$
for every small $\epsilon > 0$.  Choose now $B=B(x,r)$,
with $2B\subseteq B_0$,   such that
$$\frac12 \Upsilon_{B_0} \leq \frac{\omega^X(\Delta)}{r^{n-1}G(X_\Delta,X)},$$
where as usual $\Delta=B\cap \pom$.

Now set $B:= B(x,r)$ and $\widetilde{B}:= B(x,5r/4).$
Taking $\Phi \in C_0^\infty(\widetilde{B})$, with $0\leq\Phi\leq 1$,  $\Phi(Y)\equiv 1$
on $B(x,r)$, and $\|\nabla \Phi\|_\infty \lesssim 1/r,$  we deduce from \eqref{eq2.14} that
\begin{multline}\label{eq2.15}
\omega^X(\Delta) \lesssim \frac1r \iint_{\Omega\cap \widetilde{B}}
\left|\nabla_Y G(Y,X)\right| dY\\[5pt]
=\frac1r \iint_{\Omega\cap \widetilde{B}\cap\{\delta(Y)>\epsilon r\}}
\left|\nabla_Y G(Y,X)\right| dY
+\frac1r \iint_{\Omega\cap \widetilde{B}\cap\{\delta(Y)\leq\epsilon r\}} \left|\nabla_Y G(Y,X)\right| dY
\\[5pt] \lesssim \frac1r \iint_{\Omega\cap \widetilde{B}\cap\{\delta(Y)>\epsilon r\}}
\frac{G(Y,X)}{\delta(Y)}\,dY
+\frac1r \iint_{\Omega\cap \widetilde{B}\cap\{\delta(Y)\leq\epsilon r\}}
\frac{G(Y,X)}{\delta(Y)} \,dY\\[5pt]
=:I \,+\, II,
\end{multline}
where $\epsilon>0$ is at our disposal, and where in the next to last line
we have used standard interior estimates for harmonic functions.
By Harnack's inequality and the Harnack Chain condition, we have that
$$I\,\leq \,C_\epsilon r^{n-1} G(X_\Delta,X)$$
as desired.  To handle term $II$, choose $y\in \partial \Omega$
such that $|Y-y|=\delta(Y)$, and set $\Delta(Y):= \Delta(y,\delta(Y))$.  Then by
\eqref{eq2.CFMS1} and the Harnack Chain condition we have
\begin{multline*}
II\,\leq\, \frac1r \iint_{\Omega\cap  \widetilde{B}\cap\{\delta(Y)\leq\epsilon r\}}
\frac{\omega^X(\Delta(Y))}{(\delta(Y))^n} \,dY\\[5pt]
\lesssim \frac1r \sum_{k: 2^{-k}\lesssim \epsilon r}\iint_{\Omega\cap
\widetilde{B}\cap\{2^{-k-1}<\delta(Y)
\leq 2^{-k}\}}
\frac{\omega^X(\Delta(Y))}{(\delta(Y))^n} \,dY\\[5pt]
\lesssim \frac1r \sum_{k: 2^{-k}\lesssim \epsilon r}
\sum_j\iint_{B^k_j\cap\{2^{-k-1}<\delta(Y)
\leq 2^{-k}\}}
\frac{\omega^X(\Delta(Y))}{(\delta(Y))^n} \,dY,
\end{multline*}
where in the last step $B_j^k:= B(x_j^k, 2^{-k+1})$, and for each $k$ in the sum,
$\mathcal{B}_k:=\{B^k_j\}_j$ is a collection of balls whose doubles have
bounded overlaps, such that $x_j^k\in\pom$,
\begin{equation}\label{eq2.cover}
\left(\Omega\cap \widetilde{B}\cap\left\{2^{-k-1}<\delta(Y) \leq 2^{-k}\right\}\right)\,\subset\,
\bigcup_j B_j^k\,,\quad {\rm and}\quad \bigcup_j 2B_j^k\subset B(x,3r/2).\end{equation}
We leave it to the reader to verify that such a collection exists, by virtue of
the ADR property of $\partial\Omega$,
for all sufficiently small $\epsilon$.
We then have that
\begin{multline} \label{eq2.16} II\, \lesssim \,\frac1r \sum_{k: 2^{-k}\lesssim \epsilon r} 2^{-k}\,
\sum_j \omega^X\left(\Delta(x_j^k,2^{-k+2})\right)\leq C\epsilon\, \omega^X\left( \Delta(x,3r/2)\right)\\[5pt]
\,\leq\, C\epsilon \sum_{\Delta'} \omega^X(\Delta')
\,\leq \,C\epsilon \,\Upsilon_{B_0}\,r^{n-1}\sum_{\Delta'}G(X_{\Delta'},X) \,\leq\,
C\epsilon \,\Upsilon_{B_0}\,r^{n-1}G(X_\Delta,X),
\end{multline}
where in the second, third, fourth and fifth inequalities we have used, respectively, the bounded overlap property of the balls $2B_j^k$; the ADR property to cover $\Delta(x,3r/2)$ by a collection
$\{\Delta'\}$ of bounded cardinality, such that $r_{\Delta'} \approx r$, and
$\Delta'=B'\cap\pom$,
with  $B'$ centered on $\pom$ and $2B'\subset B_0$; the definition of $\Upsilon_{B_0}$; and the Harnack Chain condition.
We then obtain \eqref{eq2.CFMS2} by choosing $\epsilon$ sufficiently small.
\end{proof}

\noindent{\it Remark}.  Let us observe that from the previous proof it follows that we are not really using the full strength of the qualitative exterior corkscrew condition, but only
that every boundary point is regular in the sense of Wiener and that $\Upsilon(B_0)$ is \textit{a priori} finite. Although these relaxed qualitative hypotheses suffice for our purposes, the
qualitative exterior corkscrew condition is cleaner, easier to check in practice and holds for the approximating domains introduced below.

\begin{corollary}\label{cor2.double}
Suppose that $\Omega$ is a  1-sided NTA domain with $n$-dimensional ADR boundary and that it also satisfies the qualitative exterior
Corkscrew condition.  Let $B:=B(x,r)$,  $x\in \pom$,
$\Delta:= B\cap\partial\Omega$ and $X\in \Omega\setminus 4B.$
Then there is a uniform constant $C$ such that
$$\omega^X(2\Delta)\leq C\omega^X(\Delta).$$
\end{corollary}

\begin{proof}The conclusion of the corollary follows immediately from the combination of
\eqref{eq2.CFMS1} and \eqref{eq2.CFMS2},
and Harnack's inequality.  We omit the details.
\end{proof}

Next, we establish a bound of  ``Carleson-type''  for the Green function.
The Carleson estimate is already known for arbitrary non-negative
harmonic functions vanishing on a surface ball
\cite{Ai1}, \cite{Ai2}; however, specializing to
the Green function, one may give a fairly simple direct proof,
based upon that of the previous lemma.

\begin{lemma}\label{lemma2.carleson}
Suppose that $\Omega$ is a  1-sided NTA domain with $n$-dimensional ADR boundary and that it also satisfies the qualitative exterior
Corkscrew condition.
Then there is a uniform constant $C$
such that for each $B:= B(x,r)$,  $x\in\pom$,
$\Delta = B \cap \partial\Omega$, and $X \in \Omega\setminus 2B,$ we have
\begin{equation}\label{eq2.carleson}
\sup_{Y\in B\cap\Omega}G(Y,X) \leq C \, G(X_\Delta,X).
\end{equation}
\end{lemma}

\begin{proof}  Fix $B,\Delta$ and $X$ as in the statement of the lemma, and set
$u(Y):= G(Y,X).$  Extending $u$ to be zero in $\Omega_{ext}$, we obtain from
\eqref{eq2.14} that $u$ is subharmonic in $B(x,3r/2).$  Let $B':= B(x,5r/4).$
By the sub-mean value inequality,
we have that
\begin{multline*}
\sup_{Y\in B} u(Y) \lesssim \frac1{|B'|}\iint_{B'} u
= \frac1{|B'|}\iint_{B'\cap \Omega} u\\[5pt] =
\frac1{|B'|}\iint_{B'\cap \Omega \cap \{\delta(Y)>\epsilon r\}} u(Y)\, dY
\,\,+\,\, \frac1{|B'|}\iint_{B'\cap \Omega \cap \{\delta(Y)\leq\epsilon r\}} u(Y) \, dY\\[5pt]
=:I^* + II^*\,\leq\, C_\epsilon\, u(X_\Delta)\, +\,II^*,
\end{multline*}
where in the last step we have used the Harnack Chain condition to estimate term $I^*$, and we have fixed a small $\epsilon$ as in the proof of Lemma \ref{lemma2.cfms} so that
\eqref{eq2.cover} holds.
Moreover, by definition of $u$,
$$
II^*\, \lesssim \, \frac1{|B'|} \,\epsilon r\iint_{B'\cap \Omega \cap \{\delta(Y)\leq\epsilon r\}}
\frac{G(Y,X)}{\delta(Y)} \, dY \approx \epsilon\, r^{1-n}\, II,
$$
where $II$ is exactly the same term as in \eqref{eq2.15}. In turn, by \eqref{eq2.16} and the fact that
$\Upsilon_{B_0}(X)$  is uniformly bounded (the latter fact is simply a
restatement of \eqref{eq2.CFMS2}), we find that
$$
II^*\,\leq C\epsilon^2 G(X_\Delta,X).
$$
\end{proof}

Under the same hypotheses as in the previous two lemmata, we shall
obtain a comparison principle for
the Green function, again without the use of H\"older continuity at the boundary.
In order to state our comparison principle, we shall need to
introduce the notion of a Carleson region.
Given a ``dyadic cube''
$Q\in \dd(\partial\Omega)$, the {\bf discretized Carleson region} $\dd_Q$ is defined to be
\begin{equation}\label{eq2.discretecarl}
\dd_Q:= \left\{Q'\in \dd: Q'\subseteq Q\right\}.
\end{equation}
For future reference, we also introduce discretized sawtooth regions as follows.
Given a family $\mathcal{F}$ of disjoint cubes $\{Q_j\}\subset \mathbb{D}$, we define
the {\bf global discretized sawtooth} relative to $\F$ by
\begin{equation}\label{eq2.discretesawtooth1}
\dd_{\F}:=\dd\setminus \bigcup_{\F} \dd_{Q_j}\,,
\end{equation}
i.e., $\dd_{\F}$ is the collection of all $Q\in\dd$ that are not contained in any $Q_j\in\F$.
Given some fixed cube $Q$,
the {\bf local discretized sawtooth} relative to $\F$ by
\begin{equation}\label{eq2.discretesawtooth2}
\dd_{\F,Q}:=\dd_Q\setminus \bigcup_{\F} \dd_{Q_j}=\dd_\F\cap\dd_Q.
\end{equation}

We shall also require ``geometric'' Carleson regions and sawtooths.
Let us first recall that we write $k=k(Q)$ if $Q\in \mathbb{D}_k$
(cf. Lemma \ref{lemmaCh}), and in that case the ``length'' of $Q$ is denoted
by $\ell(Q)=2^{-k(Q)}$.
We also recall that there is a Corkscrew point $X_Q$, relative to each $Q\in\dd$
(in fact, there are many such, but we just pick one).
Given such a $Q$, we define an associated
``Whitney region'' as follows.  Let $\mathcal{W}=\W(\Omega)$ denote a collection
of (closed) dyadic Whitney cubes of $\Omega$, so that the cubes in $\mathcal{W}$
form a pairwise non-overlapping covering of $\Omega$, which satisfy
\begin{equation}\label{eq4.1}
4 \diam(I)\leq
\dist(4I,\partial\Omega)\leq \dist(I,\pom) \leq 40\diam(I)\,,\qquad \forall\, I\in \mathcal{W}\,\end{equation}
(just dyadically divide the standard Whitney cubes, as constructed in  \cite[Chapter VI]{St},
into cubes with side length 1/8 as large)
and also
$$(1/4)\diam(I_1)\leq\diam(I_2)\leq 4\diam(I_1)\,,$$
whenever $I_1$ and $I_2$ touch.
Let $\ell(I)$ denote the side length of $I$,
and write $k=k_I$ if $\ell(I) = 2^{-k}$.
We set
\smallskip
\begin{equation}\label{eq2.whitney1}
\mathcal{W}_Q := \left\{I\in \mathcal{W}: \,k(Q) -m_0\,\leq \,k_I\,\leq \,k(Q)+1\,,\,{\rm and}
\dist(I,Q)\leq C_0\, 2^{-k(Q)}\right\}\,,
\end{equation}
where we may (and do) choose
the constant $C_0$ and positive integer $m_0$,
depending only on the constants in the Corkscrew condition and in the
dyadic cube construction (cf. Lemma \ref{lemmaCh}), so that
$X_Q\in I$ for some $I\in\W_Q$,
and for each dyadic child $Q^{j}$ of $Q$,
the respective Corkscrew points $X_{Q^{j}}\in I^{j}$ for some $I^{j}\in\W_Q$.
In particular, the collection $\W_Q$ is non-empty for every $Q\in \dd$.  Moreover
as long as $C_0$ is chosen large
enough depending on the constant $c$ in the Corkscrew condition, then by the properties
of Whitney cubes, we may always find an $I\in \W_Q$ with the slightly more precise property
that $k(Q)-1\leq k_I\leq k(Q)$.
We may further suppose, by choosing $C_0$ large enough,  that
\begin{equation}\label{eq2.whitney2**}
\W_{Q_1}\cap \W_{Q_2}\neq\emptyset\,,\,\, {\rm whenever}\,\,1\leq
\frac{\ell(Q_2)}{\ell(Q_1)}\leq 2\,,\,\, {\rm and}\,\dist(Q_1,Q_2)\leq1000\ell(Q_2)\,.
\end{equation}
We omit the details.
In the sequel, we shall assume
always that $C_0$ has been so chosen, and further that $C_0\geq 1000\sqrt{n}$.

We shall need to augment
$\mathcal{W}_Q$ in order to exploit the Harnack Chain condition.
It will be convenient to introduce the following notation:  given a subset
$A\subset \Omega$, we write
$$ X\rightarrow_A Y$$
if the interior of $A$ contains all the balls in
a Harnack Chain (in $\Omega$), connecting $X$ to $Y$, and if, moreover, for any point $Z$ contained
in any ball in the Harnack Chain, we have
\begin{equation}
\label{eq2.toa}
\dist(Z,\pom) \approx \dist(Z,\Omega\setminus A)\,,
\end{equation}
with uniform control of the implicit constants.   We denote by $X(I)$ the center
of a cube $I\in \ree$, and we recall that $X_Q$ denotes a designated Corkscrew point relative
to $Q$, which we may, from this point on,
assume without loss of generality to be the center of some Whitney cube
$I$ such that $\ell(I)\approx\ell(Q)\approx \dist(I,Q)$.  More precisely, we note the following.
\begin{remark}\label{remark1.18}
Having fixed the collection $\W$  (the Whitney cubes of $\Omega$), by taking the
Corkscrew constant $c$ to be slightly smaller, if necessary, we may assume that
the Corkscrew point $X_Q$ is the center of some $I\in\W$, with
$I\subset B_Q\cap\Omega$ and
$\ell(I) \approx\ell(Q).$
\end{remark}

We now define the augmented collection $\W^*_Q$ as follows.  For each $I\in \W_Q$, we form a Harnack Chain, call it $H(I)$,
from the center $X(I)$ to the Corkscrew point $X_Q$.   We now denote by $\W(I)$
the collection of all Whitney cubes which meet at least one ball in the chain $H(I)$,
and we set
$$\W_Q^* := \bigcup_{I\in \W_Q} \W(I).$$
We also define, for $\lambda\in(0,1)$ to be chosen momentarily,
\begin{equation}\label{eq2.whitney3}
U_Q := \bigcup_{\mathcal{W}^*_Q} (1+\lambda)I=:
\bigcup_{I\in\,\mathcal{W}^*_Q} I^*\,.
\end{equation}
By construction, we then have that
\begin{equation}\label{eq2.whitney2*}
 \mathcal{W}_Q\subset \mathcal{W}^*_Q\subset\mathcal{W}\,\quad {\rm and}\quad
X_Q\in U_Q\,,\,\, X_{Q^j} \in U_Q\,,
\end{equation}
for each child $Q^j$ of $Q$.
It is also clear that there are uniform constants $k^*$ and $K_0$ such that
\begin{eqnarray}\label{eq2.whitney2}
& k(Q)-k^*\leq k_I \leq k(Q) + k^*\,, \quad \forall I\in \mathcal{W}^*_Q\\\nonumber
&X(I) \rightarrow_{U_Q} X_Q\,,\quad\forall I\in \mathcal{W}^*_Q\\ \nonumber
&\dist(I,Q)\leq K_0\,2^{-k(Q)}\,, \quad \forall I\in \mathcal{W}^*_Q\,,
\end{eqnarray}
where $k^*$, $K_0$ and the implicit constants in \eqref{eq2.toa} (which
pertain to the condition $X(I)\to_{U_Q} X_Q$),
depend only  on the ``allowable
parameters'' (since $m_0$ and $C_0$ also have such dependence)
and on $\lambda$.
Thus, by the addition of a few nearby
Whitney cubes of diameter also comparable to that of $Q$,
we can ``augment'' $ \mathcal{W}_Q$ so that the Harnack Chain condition
holds in $U_Q$.

We fix the parameter $\lambda$ so that
for any $I,J\in\W$,
\begin{equation}\label{eq2.29*}
\begin{split}
\dist(I^*,J^*)& \approx \dist(I,J) \\
\interior(I^*)\cap\interior(J^*)\neq \emptyset& \iff \partial I \cap\partial J \neq \emptyset
\end{split}
\end{equation}
(the fattening thus ensures overlap of $I^*$ and $J^*$ for
any pair $I,J \in\W$ whose boundaries touch, so that the
Harnack Chain property then holds locally, with constants depending upon $\lambda$, in $I^*\cup J^*$).  By choosing $\lambda$ sufficiently small,
we may also suppose that there is a $\tau\in(1/2,1)$ such that for distinct $I,J\in\W$,
\begin{equation}\label{eq2.30*}
\tau J\cap I^* = \emptyset\,.
\end{equation}
We remark that any sufficiently small choice of $\lambda$ (say $0<\lambda\le \lambda_0$) will do for our purposes.

Of course, there may be some flexibility in the choice of additional Whitney cubes
which we add to form the augmented collection $\mathcal{W}^*_Q$, but having made such a choice
for each $Q \in \mathbb{D}$, we fix it for all time.   We may then define  the {\bf Carleson box}
associated to $Q$ by
\begin{equation}\label{eq2.box}
T_Q:={\rm int}\left( \bigcup_{Q'\in \dd_Q} U_{Q'}\right).
\end{equation}
Similarly, we may define geometric sawtooth regions as follows.
As above, give a family $\mathcal{F}$ of disjoint cubes $\{Q_j\}\subset \mathbb{D}$,
we define the {\bf global sawtooth} relative to $\mathcal{F}$ by
\begin{equation}\label{eq2.sawtooth1}
\Omega_{\mathcal{F}}:= {\rm int } \left( \bigcup_{Q'\in\dd_\F} U_{Q'}\right)\,,
\end{equation}
and again given some fixed $Q\in \dd$,
the {\bf local sawtooth} relative to $\mathcal{F}$ by
\begin{equation}\label{eq2.sawtooth2}
\Omega_{\mathcal{F},Q}:=  {\rm int } \left( \bigcup_{Q'\in\dd_{\F,Q}} U_{Q'}\right)\,.
\end{equation}

For future reference, we present the following.
\begin{lemma}\label{lemma2.31}  Suppose that $\Omega$ is a 1-sided NTA domain
with an ADR boundary.
Given $Q\in\dd$, let $B_Q:=B(x_Q,r)$, $r\approx\ell(Q)$,
and $\Delta_Q:=B_Q\cap\partial\Omega \subset Q$,
be as in \eqref{cube-ball} and \eqref{cube-ball2}.
Then for each $Q$, there is a ball $B_Q':=B(x_Q,s)\subset B_Q$, with $s\approx\ell(Q)\approx r$,
such that
\begin{equation}\label{eq2.31a}
B_Q'\cap\Omega \subset T_Q.
\end{equation}
Moreover, for a somewhat smaller choice of $s\approx (K_0)^{-1}\ell(Q)$, we have for every
pairwise disjoint family $\F\subset\dd$, and for each $Q_0\in \dd$ containing $Q$, that
\begin{equation}\label{eq4.12a}
B'_Q\cap \Omega_{\F,Q_0} = B'_Q\cap \Omega_{\F,Q}.
\end{equation}
\end{lemma}
\begin{proof} We prove \eqref{eq2.31a} first.
Let $Y\in \Omega$, with $|Y-x_Q|<cr=:s,$ where $c>0$ is to be determined.
Then $Y\in I\in\W$, with
$$\ell(I)\approx \delta(Y)\leq |Y-x_Q| <cr.$$
Fix $Q_I\in \dd$ such that $\ell(Q_I)=\ell(I),$ and $\dist(Q_I,I)\approx \ell(I).$
In particular, $I\in\W_{Q_I}\subset \W^*_{Q_I}$,
so that $I\subset \interior(I^*)\subset \interior(U_{Q_I})$.
By the triangle inequality, for all $x\in Q_I$, we have
$$|x-x_Q|\leq |x-Y|+|Y-x_Q|\leq C \ell(I)+cr\leq C cr<r,$$
if $c$ is chosen small enough.   Hence,
$Q_I\subset \Delta_Q \subset Q$,  so
$Y\in \cup_{Q'\in\dd_Q} \interior(U_{Q'})\subset T_Q.$

We now turn to the proof of \eqref{eq4.12a}.  Since $Q\subset Q_0$, the
``right to left'' containment is trivial, for any choice of $B_Q'$.
We therefore suppose that $Y\in B_Q'\cap
 \Omega_{\F,Q_0}$, where again $B_Q':=B(x_Q,s)$, and $s$ will be chosen momentarily.
 It is enough to show that $Y \in \Omega_{\F,Q}$, for some choice of $s \approx (K_0)^{-1}\ell(Q)$.
 Since $Y\in  \Omega_{\F,Q_0}$, by definition there is some $Q'\in \dd_{Q_0}\cap\dd_\F$,
 for which $Y\in I^*=(1+\lambda)I$, with $I\in \W^*_{Q'}$.  Then
 $$\ell(Q')\approx\ell(I)\approx \delta(Y)\leq
 |Y-x_Q|\leq s\,,$$
 where in the last step we have used that $Y\in B_Q'$.
 Moreover, for every $y'\in Q'$, we have
 $$
 |y'-Y|\lesssim\dist(Y,Q')\lesssim\dist(I,Q')\lesssim K_0 \,\ell(Q') \lesssim K_0\, s.
 $$
 Thus, by the triangle inequality, for every $y'\in Q'$, we have
 $$
 |y'-x_Q|\lesssim K_0\,s  < r \approx \ell(Q),
 $$
 by choice of $s = c (K_0)^{-1}\ell(Q)$ with $c$ sufficiently small.  Thus, $Q'\subset \Delta_Q\subset Q$,
 whence $Y\in \cup_{Q'\in\dd_{\F,Q}}U_{Q'}$, i.e., we have shown that
 $B_Q'\cap \Omega_{\F,Q_0}\subset \cup_{Q'\in\dd_{\F,Q}}U_{Q'}$, and therefore
 $\interior(B_Q')\cap \Omega_{\F,Q_0}\subset  \Omega_{\F,Q}$, by definition.  Choosing $s$ slightly
 smaller, which amounts to replacing $B_Q'$ by a slightly smaller ball, we obtain
 \eqref{eq4.12a}.
\end{proof}

We also define as follows the ``Carleson box'' $T_\Delta$
associated to a surface ball $\Delta:=\Delta(x_\Delta,r)$.
Let $k(\Delta)$ denote the unique $k \in \mathbb{Z}$ such that $2^{-k-1}<200\,r\leq 2^{-k}$,
and set
\begin{equation}\label{eq2.box3}
\mathbb{D}^\Delta:= \{Q \in \mathbb{D}_{k(\Delta)}: Q\cap2\Delta \neq \emptyset\}.
\end{equation}
We then define
\begin{equation}\label{eq2.box2}
T_\Delta :={\rm int}\left(\bigcup_{Q\in \mathbb{D}^\Delta} \overline{T_Q}\right).
\end{equation}

For future reference, we record the following analogue of Lemma
\ref{lemma2.31}.   Set $B_\Delta := B(x_\Delta, r),$ so that
$\Delta = B_\Delta\cap\partial\Omega$.   Then
\begin{equation}\label{eq2.ball-box}
\frac54 B_\Delta\cap\Omega\subset T_\Delta.
\end{equation}
Indeed, let $X\in \Omega$ with $|X-x_\Delta|<5r/4$.  Then $X\in I\in \mathcal{W}$ with
$\ell(I)\approx\delta(X)<5r/4$,
so that $\ell(I) \leq\ell(Q)$, for each $Q\in\dd^\Delta$.

Suppose first that $\delta(X)<3r/4.$  There is an $x_1\in \partial\Omega$ such that
$|X-x_1|=\delta(X)$, so that by the triangle inequality, $|x_1-x_\Delta|< 2r.$ Consequently,
there is a $Q\in \mathbb{D}^\Delta$ for which $x_1\in \overline{Q}$, whence
there is a $Q'\subset Q$, whose closure contains $x_1$,
such that $\ell(Q')=\ell(I),$ and $\dist(Q',I)\leq \delta(X)\leq 41\diam(I)\ll C_0\,\ell(Q')$
(cf \eqref{eq4.1}-\eqref{eq2.whitney1}).
Thus, $I\in \mathcal{W}_{Q'}$, so
$X \in \interior(I^*)\subset \interior(U_{Q'})\subset T_Q\subset T_\Delta$.

Now suppose that
$3r/4\leq\delta(X)<5r/4.$ Then $X\in I$ with $\ell(I)\approx r$, and $\dist(I,Q') \approx r$ for every
$Q'$ contained in any $Q\in \mathbb{D}^\Delta$, with $\ell(Q')\approx\ell(I)$.  In that case, we have that
$I\in \mathcal{W}_{Q'}$, for each such $Q'$, so that $X\in T_Q, \forall Q\in \mathbb{D}_\Delta$.

\begin{lemma}\label{lemma2.30}  Suppose that $\Omega$ is a  1-sided NTA domain
with an ADR boundary.
Then all of its  Carleson boxes
$T_Q$ and $T_\Delta$, and sawtooth regions
$ \Omega_{\mathcal{F}}$, and $\Omega_{\mathcal{F},Q}$ are also 1-sided NTA domains
with ADR boundaries.
If in addition $\partial \Omega$ is also
UR, then so is the boundary of each Carleson box $T_Q$ and $T_\Delta$.
 In all cases, the implicit constants are uniform, and
depend only on dimension and on the corresponding constants for $\Omega$.
\end{lemma}

We defer the proof until Appendix \ref{ainherit}.

We remark that it seems likely that
one could show that the sawtooth regions also inherit the UR property,
but in our case, the only sawtooths that we work with will
enjoy an even stronger property, so we shall not bother to explore this issue here.

\begin{lemma}\label{lemma:inher-quali}
Suppose that $\Omega$ is a 1-sided NTA domain with an ADR boundary and that $\Omega$ also satisfies the qualitative exterior
Corkscrew condition. Then all of its  Carleson boxes
$T_Q$ and $T_\Delta$, and sawtooth regions $ \Omega_{\mathcal{F}}$, and $\Omega_{\mathcal{F},Q}$ satisfy the qualitative exterior
Corkscrew condition. In all cases, the implicit constants are uniform, and
depend only on dimension and on the corresponding constants for $\Omega$.
\end{lemma}

The proof of this result is almost trivial. Consider for instance the domain $\Omega_{\mathcal{F},Q}$, and let $x\in \partial\Omega_{\mathcal{F},Q}$ and $r\le 2^{-N}$, with $N$ corresponding to the qualitative exterior Corkscrew condition assumed on $\Omega$. If either $x\in \pom$ or $x\in \Omega$ with $\delta(x)<r/2$, there exists $y\in \pom$ such that $B(y,r/2)\subset B(x,r)$. Then the exterior corkscrew point relative to $\Delta(y,r/2)$ is also a Corkscrew point relative to $B(x,r)\cap \Omega_{\mathcal{F},Q}$. The case $x\in \Omega$ with $\delta(x)\ge r/2$ is as follows. There exists a Whitney box $I$, with
$\ell(I)\approx \delta(X)$, such that $x\in \partial I^*$ and ${\rm int}(I^*)\subset \Omega_{\mathcal{F},Q}$. Note that $\partial I^*$ can be covered by Whitney boxes $J$ that meet $I$ by \eqref{eq2.29*}.  Since
$x$ is a boundary point of $\Omega_{\F,Q}$, there is a
$J\ni x$, with $J\notin \W_{Q'}^*$ for any $Q'\in \dd_{\F,Q}$.   Consequently,
$B(x,r)$ has an ``ample" intersection with $J\setminus\Omega_{\F,Q}$, wherein we may find the required Corkscrew point.
Further details are left to the reader.

We are now ready to state our comparison principle for the Green function.
The result is already known \cite{Ai1}, but we include the proof
here for the sake of self-containment.
Given a
surface ball $\Delta:=\Delta(x,r)$, let $B_\Delta := B(x,r)$, so that
$\Delta =B_\Delta\cap\partial\Omega$.
We fix
$\kappa_0$ large enough that
\begin{equation}\label{eq2.contain}
\overline{T_\Delta}\subset \kappa_0B_\Delta\cap\overline{\Omega}.\end{equation}

\begin{lemma}\label{lemma2.comparison}
Suppose that $\Omega$ is a  1-sided NTA domain with $n$-dimensional ADR boundary and that it also satisfies the qualitative exterior
Corkscrew condition.
Then there is a uniform constant $C$
such that for each surface ball $\Delta$, and
for every $X,Y \in \Omega \setminus 2\kappa_0B_\Delta,$ and  $Z\in  B_\Delta\cap\Omega$, we have
$$\frac1C\,\frac{G(Z,X)}{G(Z,Y)}\leq \frac{G(X_\Delta,X)}{G(X_\Delta,Y)}\leq C\,\frac{G(Z,X)}{G(Z,Y)}.$$
\end{lemma}

\begin{remark}\label{remark:tent-inherits}
By Lemma \ref{lemma2.30} and Lemma \ref{lemma:inher-quali}, every Carleson box $T_\Delta\subset\Omega$ is a 1-sided NTA domain with $n$-dimensional ADR boundary and also satisfies the qualitative exterior
Corkscrew condition.
\end{remark}

\begin{proof} We follow  \cite[Lemma 1.3.7]{Ke}.  Given a surface ball $\Delta$, fix
$X,Y \in \Omega \setminus 2\kappa_0B_\Delta$, and let
$\omega^Z_{\Delta}$ denote harmonic measure for the sub-domain $T_\Delta$.  Set
$$S_1:= \partial T_\Delta \cap \{Z\in \Omega:\delta(Z)>r/2\},$$
where $r$ is the radius of $B_\Delta$. By Remark \ref{remark:tent-inherits},
Corollary \ref{cor2.double} applies in $T_\Delta$, whence by \eqref{eq2.ball-box},
\begin{equation}\label{eq2.30}
\omega^Z_{\Delta}(\partial T_\Delta \cap\Omega)\,\leq\, C \omega^Z_{\Delta}(S_1),\qquad \forall Z\in
B_\Delta\cap\Omega.
\end{equation}
Now set $B^*:= \kappa_0B_\Delta$ and $\Delta^*:= B^*\cap\partial\Omega$.
By Lemma \ref{lemma2.carleson} and the Harnack Chain condition we have
$$G(Z,X)\leq C\,G(X_{\Delta^*},X)\,\lesssim G(X_\Delta,X),\qquad \forall Z\in
B^*\cap\Omega.$$
By the maximum principle and \eqref{eq2.contain}, we then have
\begin{equation}\label{eq2.31}
G(Z,X)\leq C\,G(X_\Delta,X)\,\omega^Z_{\Delta}(\partial T_\Delta \cap\Omega),\qquad \forall Z\in
T_\Delta.\end{equation}
On the other hand, by the Harnack Chain condition,
$$G(Z,Y) \geq C^{-1} G(X_\Delta,Y),\qquad \forall Z \in S_1,$$
and therefore by the maximum principle we have
\begin{equation}\label{eq2.32}
G(Z,Y) \geq C^{-1} G(X_\Delta,Y)\,\omega^Z_{\Delta}(S_1),\qquad \forall Z \in T_\Delta.\end{equation}
Combining \eqref{eq2.30}, \eqref{eq2.31} and \eqref{eq2.32}, we obtain
$$\frac{G(Z,X)}{G(X_\Delta,X)} \lesssim \frac{G(Z,Y)}{G(X_\Delta,Y)}\,,\qquad \forall Z\in
B_\Delta\cap\Omega.$$
The opposite inequality follows by interchanging the roles of $X$ and $Y$.
\end{proof}

\begin{corollary}\label{cor2.poles} Given the same hypotheses as in Lemma \ref{lemma2.comparison},
there is a uniform constant $C$ such that for every pair of surface balls
$\Delta:=B\cap\partial\Omega$,
and $\Delta':= B'\cap\partial\Omega$,  with $B'\subseteq B,$ and
for every $X\in \Omega\setminus 2\kappa_0B,$
where $\kappa_0$ is the constant in \eqref{eq2.contain}, we have
$$\frac1C\, \omega^{X_\Delta}(\Delta') \,\leq\, \frac{\omega^X(\Delta')}{\omega^X(\Delta)}\,
\leq \,C\, \omega^{X_\Delta}(\Delta').$$
\end{corollary}

\begin{proof} We follow \cite[Corollary 1.3.8]{Ke}.  Fix $\Delta',\, \Delta,\,B',\, B$ and $X$
as in the statement of the present corollary.    Set $B^{**}=\kappa_1B$
and $\Delta^{**}:=B^{**}\cap\partial\Omega$, where we may choose
$\kappa_1$ large enough, depending only on $\kappa_0$ and on the
constants in the Corkscrew condition, such that $X_{\Delta^{**}}\in \Omega \setminus
2\kappa_0 B$.   Let $r'$ and $r$ denote the respective radii of $B'$ and $B$.
 By Lemma \ref{lemma2.cfms}, we have
\begin{eqnarray*}\label{eq2.34}&\omega^X(\Delta') \approx (r')^{n-1} G(X_{\Delta'},X),
\\[4pt]\nonumber
&\omega^X(\Delta) \approx r^{n-1} G(X_{\Delta},X),
\\[4pt]\nonumber
&\omega^{X_\Delta}(\Delta')\approx \omega^{X_{\Delta^{**}}}(\Delta')\approx
(r')^{n-1} G(X_{\Delta'},X_{\Delta^{**}}),
\end{eqnarray*}
where in the third line we have also used  the
Harnack Chain condition.
Moreover, by Lemma \ref{lemma2.green} and the Harnack Chain condition,
we have
$$r^{n-1}G(X_{\Delta},X_{\Delta^{**}}) \approx 1.$$
Note that $X_{\Delta'}\subset B'\subset B.$
Thus,
by Lemma \ref{lemma2.comparison}, we have
$$\frac{G(X_{\Delta'},X)}{G(X_{\Delta'},X_{\Delta^{**}})}\approx
\frac{G(X_{\Delta},X)}{G(X_{\Delta},X_{\Delta^{**}})}\,,$$
and the conclusion of the corollary follows.
\end{proof}

\section{Harnack Chains imply a Poincar\'e inequality}\label{spoincare}

In this section we prove that a certain Poincar\'e inequality holds in any domain
$\Omega$ satisfying the ADR, Corkscrew and
Harnack Chain properties.  We therefore
impose those three hypotheses throughout this section.  It will be convenient to set
some additional notation.   As above, we let $\mathcal{W}$ denote the collection of Whitney cubes
of $\Omega$, and we recall that these have been constructed so that for $I\in \mathcal{W}$, we have
$\dist(4I,\partial\Omega) \approx \ell(I)$ (cf. \eqref{eq4.1}).
Given a pairwise disjoint family $\F\in \dd$,
and a constant $\rho>0$,  we derive from $\F$ another family
$\Fr\subset\dd$, as follows.
We augment $\F$ by adjoining to it all those $Q\in\dd$ of side length $\ell(Q)\leq\rho$, and we
denote this augmented collection by $\mathfrak{C}(\F,\rho)$.
We then let $\Fr$ denote the collection of the maximal cubes of $\mathfrak{C}(\F,\rho)$.
Thus, the corresponding discrete sawtooth $\mathbb{D}_{\Fr}$
consists precisely of those $Q\in\dd_{\F}$ such that $\ell(Q)>\rho$.

Having constructed the family $\Fr$, and given $Q\in\dd$ with $\ell(Q)>\rho$,
we may then define local discrete and geometric sawtooth regions
$\dd_{\Fr,Q}$ and $\Omega_{\Fr,Q}$ with respect to this family as in
\eqref{eq2.discretesawtooth2}-\eqref{eq2.whitney3} and \eqref{eq2.sawtooth2}.

We shall also find it useful to consider certain ``fattened'' versions of the sawtooth regions,
as follows.  Bearing in mind \eqref{eq4.1}, we set
\begin{eqnarray}\label{eq2.whitney3fat}
&U^{fat}_Q := \bigcup_{\mathcal{W}^*_Q} 4I,\\[4pt]\label{eq2.boxfat}
&T^{fat}_Q:=  {\rm int } \left( \bigcup_{Q'\in\dd_{Q}} U^{fat}_{Q'}\right),
\\[4pt]\label{eq2.sawtooth2fat}
&\Omega^{fat}_{\mathcal{F},Q}:=  {\rm int } \left( \bigcup_{Q'\in\dd_{\F,Q}} U^{fat}_{Q'}\right)
\end{eqnarray}
(compare to \eqref{eq2.whitney3}, \eqref{eq2.box} and \eqref{eq2.sawtooth2}).
We note that, by construction,
\begin{eqnarray}\label{eq3.2}
\delta(X)\gtrsim \rho\,, &\quad{\rm if}\,\, X\in  \Omega^{fat}_{\Fr,Q}\\[4pt]\label{eq3.3}
\delta(X)\lesssim \rho\,,&\quad{\rm if}\,\ X\in \Omega_{\F,Q}\setminus \Omega_{\Fr,Q}\,.
\end{eqnarray}

Given a pairwise disjoint family $\F\in \dd$, and a cube $Q\in \dd_{\F}$, we define
\begin{equation}\label{eq3.4}
\W_{\F}:=\bigcup_{Q'\in\dd_{\F}}\W^*_{Q'}\,,\qquad
\W_{\F,Q}:=\bigcup_{Q'\in\dd_{\F,Q}}\W^*_{Q'}\,,\end{equation}
so that in particular, we may write
\begin{equation}\label{eq3.saw}
\Omega_{\mathcal{F},Q}={\rm int }\,\big(\bigcup_{I\in\,\W_{\F,Q}} I^*\big)\,,
\qquad  \Omega^{fat}_{\mathcal{F},Q}
={\rm int }\,\big(\bigcup_{I\in\,\W_{\F,Q}} 4I\big)\,,\end{equation}
where we recall that $I^*:=(1+\lambda)I$.

Suppose now that $Q\in \dd_\F$,  let $r\approx\ell(Q)$ and fix a small $\epsilon>0.$
Then for $I,J\in \W_{\F(\epsilon r),Q}$, we have $\ell(I)\approx \ell(J)$
and $\dist(I,J)\lesssim \ell(I)$ (where the implicit constants depend upon $\epsilon$). By
Lemma \ref{lemma2.30} there is a chain $\{I_1,I_2,\dots,I_N\}\subset\W_{\F(\epsilon r),Q}$,
of bounded cardinality $N$ depending only on dimension,
the Harnack Chain constants, and $\epsilon$, such that
$I_1=J,\, I_N=I$, $\ell(I_j)\approx\ell(I)$ for each $j$
(again the implicit constants depend upon $\epsilon$),  and for which
$\cup_{j=1}^NI^*_j$ contains a Harnack Chain
which connects the centers of $I$ and $J$.  Moreover, for $\lambda$ chosen small enough,
the chain may be constructed so that
for each $j,\, 1\leq j \leq N-1$, either $I^*_j \subset 4I_{j+1}$, or
$I^*_{j+1}\subset 4I_j$.  In the sequel,  we shall refer to such a
chain $\{I_1,I_2,\dots,I_N\}$ as a ``Harnack Chain of Whitney cubes connecting $J$ to $I$'' (we beg the
reader's indulgence for this mild abuse of terminology:  it is of course really the dilates
$\{I_j^*\}$ which form a Harnack Chain).

\begin{lemma}\label{poincare}  Suppose that $\Omega$ is a 1-sided
NTA domain with ADR boundary.
Fix $Q_0\in \dd$, and a pairwise disjoint family $\F\subset \dd_{Q_0},$ and let
$Q\in \dd_{\F,Q_0}.$ If $r\approx \ell(Q)$,
then for every $p,\, 1\leq p<\infty$, and for every small $\epsilon >0$,
there is a constant $C_{\epsilon,p}$ such that
$$ \iint_{\Omega_{\F(\epsilon r),Q}}|f-c_{Q,\epsilon}|^p \leq C_{\epsilon,p}\,
r^p\iint_{\Omega^{fat}_{\F(\epsilon r),Q}}|\nabla f|^p,$$
where $c_{Q,\epsilon}:= |\Omega_{\F(\epsilon r),Q}|^{-1}\iint_{\Omega_{\F(\epsilon r),Q}}f.$
\end{lemma}

\begin{proof}Let $X\in \Omega_{\F(\epsilon r),Q},$ so that in particular, $X\in I^*_X$ where
$I_X\in \W_{\F(\epsilon r),Q}$, and observe that
\begin{multline*}\left|f(X)\,-\,\frac1{|\Omega_{\F(\epsilon r),Q}|}\iint_{\Omega_{\F(\epsilon r),Q}}f\right|
\,=\,\left|\frac1{|\Omega_{\F(\epsilon r),Q}|}\iint_{\Omega_{\F(\epsilon r),Q}}\Big(f(X)-f(Y)\Big)\,dY\right|
\\[4pt]
\leq\,\,\frac1{|\Omega_{\F(\epsilon r),Q}|}\sum_{I\in\ \W_{\F(\epsilon r),Q}}
\iint_{I^*}\left|f(X)-f_{I^*_1} +f_{I^*_1} -\dots +f_{I^*_N} -f(Y)\right|\,dY\,,\end{multline*}
where $I_1,\dots,I_N$ is a Harnack Chain of Whitney cubes connecting $I_X$ to $I$
and $f_{I^*_j}:= |I^*_j|^{-1}\dint_{I^*_j} f$.
Of course, $N$ depends upon $\epsilon$.  It also depends
upon $I$ in the sum,  but in a uniformly bounded manner for $\epsilon$ fixed.
Consequently, it is enough to consider
$$\iint_{J^*}\left(\frac1{|\Omega_{\F(\epsilon r),Q}|}
\iint_{I^*}\left|f(X)-f_{I^*_1} +f_{I^*_1} -\dots+f_{I^*_N} -f(Y)\right|\,dY\right)^p dX\,,$$
where $J,I\in \W_{\F(\epsilon r),Q}$ and are connected by the chain of cubes
$I_1,I_2,\dots,I_N$.    The desired bound may now be obtained from the
standard Poincar\'e inequality as follows.  First,
\begin{multline*}\iint_{J^*}\left(\frac1{|\Omega_{\F(\epsilon r),Q}|}
\iint_{I^*}\left|f(X)-f_{I^*_1}\right|dY\right)^p dX\\[4pt]
\leq\, \iint_{J^*}\left|f(X)-f_{I^*_1}\right|^p dX\,\leq\,
C_{p} \,\ell(J)^p\iint_{J^*}|\nabla f|^p,
\end{multline*}
since $I_1=J.$  Similarly, the contribution of $f_{I^*_N} -f(Y)$ is bounded by
$$\iint_{I^*}\left|f_{I^*_N} -f(Y)\right|^p\,dY\,\leq\,
C_{p} \,\ell(I)^p\iint_{I^*}|\nabla f|^p,$$
since $I_N=I$.  Finally, to handle the contribution of any term $f_{I^*_j} - f_{I^*_{j+1}}$, we observe that
$$\left|f_{I^*_j} - f_{I^*_{j+1}}\right|\leq \left|f_{I^*_j} - f_{I^{**}_{j}}\right|
+\left|f_{I^{**}_j} - f_{I^*_{j+1}}\right|,$$
where $I_j^{**}:= 4I_j$ or $4I_{j+1}$, whichever has the larger diameter.  Then for example,
$$\left|f_{I^*_j} - f_{I^{**}_{j}}\right|\leq \frac1{|I^*_j|}\iint_{I^*_j}|f-f_{I^{**}_{j}}|\lesssim \ell(I_j)\,
\frac1{|I_j|}\iint_{I^{**}_j}|\nabla f|,$$
and similarly for the term $\left|f_{I^{**}_j} - f_{I^*_{j+1}}\right|$ since, as noted above,
$I_j^{**}$ contains both $I^*_j$ and $I^*_{j+1}$.
The Poincar\'e inequality now follows, since each $I \in \W_{\F(\epsilon r),Q}$,
and thus every $I_j$ in any of the chains, has side length proportional to $r$, depending on $\epsilon.$
\end{proof}

\section{A criterion for exterior Corkscrew points}\label{s5*}
We present a criterion for the existence of Corkscrew points in the
domain {\it exterior} to a sawtooth region.   We begin with a series of lemmas in which we establish
some local estimates for the single layer potential operator $\mathcal{S}$ defined in \eqref{eq1.layer},
and also prove some geometric properties of sawtooth regions and of domains with ADR boundaries.

\begin{lemma}\label{lemma4.1}
Suppose that $E\subset \ree$ is $n$-dimensional ADR, and let $\kappa>1$.  If $1\leq q<(n+1)/n$,
there is a constant $C_{q,\kappa}$ depending only on $n,q,\kappa$ and the ADR constants such that
for every $x\in E$, $B:=B(x,r)$ and $\kappa\Delta:=\kappa B\cap E$, we have
\begin{equation}\label{eq4.2}\iint_B |\nabla \mathcal{S} 1_{\kappa\Delta}(X)|^q\, dX \,\leq\, C_{q,\kappa}
\, r^{n+1}.
\end{equation}
\end{lemma}

\begin{proof}The left hand side of \eqref{eq4.2} is crudely dominated by a constant times
\begin{multline*}\dint_B\left|\int_{\kappa\Delta}\frac1{|X-y|^n}\,dH^n(y)\right|^q\,dX
\leq H^n(\kappa\Delta)^{q-1}\,\dint_B\left(\int_{\kappa\Delta}
\frac1{|X-y|^{nq}}\,dH^n(y)\right)\,dX\\[4pt]
\lesssim r^{n(q-1)}
\int_{\kappa\Delta}\left(\dint_{|X-y|<(\kappa + 1)r}\frac1{|X-y|^{nq}}\,dX\right)\,dH^n(y)\\[4pt]
\approx\, r^{n(q-1)} \,r^n \,r^{n+1-nq}\,=\,r^{n+1},\end{multline*}
where of course the implicit constants depend on $\kappa$ and $q$.
\end{proof}

\begin{lemma}\label{lemma4.3} Suppose that $E\subset \ree$ is $n$-dimensional ADR.
For $\rho>0$, define the ``boundary strip'' $\Sigma_\rho:= \{X\in \ree\setminus E: \dist(X,E)<\rho\}$.  Then there is a uniform constant $C$
such that for every ball $B:=B(x,r)$ centered on $E$, and
for $\rho\leq r,$ we have
\begin{equation}\label{eq5.4}
\left|\Sigma_\rho\cap B\right|\leq C\rho r^n.\end{equation}
\end{lemma}

\begin{proof} Let
$\mathcal{W}_E$ denote the collection of cubes in the Whitney decomposition
of $\ree\setminus E$, and for each
$k\in \mathbb{Z}$, set $\mathcal{W}_k:=\{I\in \mathcal{W}_E: \ell(I) = 2^{-k}\}$.
For each $I\in \mathcal{W}_E$, choose $Q_I\in \mathbb{D}(E)$ such that $\ell(Q_I)=\ell(I),$
and $\dist(I,Q_I)= \dist(I,E) \approx\ell(I).$  By ADR, for each $I$ there are at most a bounded
number of $Q\in \mathbb{D}(E)$ having these properties, and we just pick one.  We note that
if $I\cap B$ is non-empty, and if $\ell(I)\lesssim \rho\leq r$, then $Q_I\subset \kappa_1 B$ for some uniform constant $\kappa_1$.  Moreover, the collection
$\{Q_I\}_{I\in \mathcal{W}_k}$ has bounded overlaps for each fixed $k$.
We then have
\begin{multline}\label{eq4.5}
\left|\Sigma_\rho\cap B\right|\,\leq
\,\sum_{k:2^{-k}\lesssim \rho}\,\sum_{I\in\mathcal{W}_k}|I\cap B|\\[4pt]\qquad
\approx\,\sum_{k:2^{-k}\lesssim \rho}\,\sum_{I\in\mathcal{W}_k}|I\cap B|\,\, \ell(I)^{-n}\, H^n(Q_I)
\qquad {\rm (by \,ADR)}\\
\lesssim \,\sum_{k:2^{-k}\lesssim \rho}2^{-k}\sum_{I\in\mathcal{W}_k:\,Q_I\subset \kappa_1 B}\,
H^n(Q_I) \,\lesssim \,\,\rho r^n\,,
\end{multline}
where in the last step we have used the bounded overlap property of the $Q_I$'s.
\end{proof}

\begin{corollary}\label{corsioball} Let $0<\gamma<1/(n+1)$, and
suppose that $E\subset \ree$ is $n$-dimensional ADR.
Then there is a uniform
constant $C_{\gamma,\kappa}$
such that for every ball $B:=B(x,r)$ centered on $E$, $\kappa\Delta :=\kappa B\cap E$,
and every $\rho$ with $0<\rho\leq r$, we have
$$ \iint_{\Sigma_\rho\cap B} |\nabla\mathcal{S}1_{\kappa\Delta}(X)|\, dX
\leq C_{\gamma,\kappa} \,
\rho^\gamma r^{n+1-\gamma},$$
\end{corollary}

\begin{proof} The corollary follows immediately from H\"older's inequality
and the previous two lemmata.  We omit the routine details.
\end{proof}

\begin{lemma}\label{lemma4.7} Suppose that $\partial\Omega$ is ADR, and let $B:=B(x,r)$,
$\Delta:=B\cap\partial\Omega$, with
$r\leq \diam (\partial \Omega)$, and $x\in \partial \Omega$.  If
\begin{equation}\label{eq4.8}|B\cap (\ree\setminus\overline{\Omega})|\geq a r^{n+1},
\end{equation}
for some $a>0$, then there is a point
$X_\Delta^- \in B \setminus
\overline{\Omega}$, and a constant $c_1$ depending only on
$a$,$n$ and the ADR constants such that
$$B(X_\Delta^-,c_1r)\subset \ree\setminus\overline{\Omega}.$$
\end{lemma}

\begin{proof} For notational convenience, we set
$$B^-:= B\cap (\ree\setminus\overline{\Omega}).$$
We apply Lemma \ref{lemma4.3}, with $E=\partial\Omega$, and
$\rho = ar/(2C)$ (notice that without loss of generality we may assume that $a\le 1$, $C\ge 1$), and \eqref{eq4.8} to deduce that
$$\left|B^-\,\setminus \Sigma_{ar/(2C)}\right|
\geq \frac12 a r^{n+1}.$$
In particular $B^-\,\setminus \Sigma_{ar/(2C)}$ is non-empty.  Moreover, by definition
of $\Sigma_\rho$, we have that $\dist(X,\partial\Omega) \geq ar/(2C)$, for every
$X\in B^-\,\setminus \Sigma_{ar/(2C)}$.  Therefore, any such $X$ may be taken as the point
$X_\Delta^-$, with $c_1:= a/(4C).$
\end{proof}

\noindent{\it Remark}.  Given a domain $\Omega$, we shall henceforth refer to a Corkscrew point for
the domain $\ree\setminus\overline{\Omega}$, such as the point
$X_\Delta^-$ in the lemma, as an  ``exterior Corkscrew point''.

\begin{lemma}\label{lemma4.9}
Suppose that $\Omega$ is a 1-sided NTA domain with ADR boundary.   Let $\F\subset \dd$ be a pairwise disjoint family.
Then for every $Q\subseteq Q_j \in \F$, there is a ball
$B'\subset\ree\setminus\Omega_{\F}$,  centered at $\pom$,
with radius $r'\approx\ell(Q)/K_0$, and $\Delta':=B'\cap\partial\Omega\subset Q$.
\end{lemma}

\begin{proof} Recall that there exist $B_Q:=B(x_Q,r)$ and $\Delta_Q :=B_Q\cap\partial \Omega
\subset Q$, as defined in  \eqref{cube-ball} and \eqref{cube-ball2}, where $r \approx \ell(Q)$.
We now set $$B' = B\left(x_Q,(MK_0)^{-1}r\right)\,,$$
where $M$ is a sufficiently large number to be chosen momentarily.
We need only verify that $B'\cap\Omega_\F=\emptyset.$  Suppose not.  Then by definition of
$\Omega_\F$, there is a Whitney cube
$I\in \W_\F$ (cf. \eqref{eq3.4}) such that $I^*$ meets $B'$.  Since $I^*$ meets $B'$, there is a point
$Y_I\in I^*$ such that
$$\ell(I)\approx \dist(I^*,\partial\Omega)\leq |Y_I-x_Q|\leq r/(MK_0)\approx \ell(Q)/(MK_0).$$
On the other hand, since $I\in \W_\F$,
there is a $Q_I\in \dd_\F$ (hence $Q_I$ is not contained in $Q_j$)
with $\ell(I)\approx\ell(Q_I)$, and
$\dist(Q_I,Y_I)\approx\dist(Q_I,I)\lesssim K_0\,\ell(I)\lesssim \ell(Q)/M.$
Then by the triangle inequality,
$$|y-x_Q|\lesssim \ell(Q)/M\,,\qquad \forall y\in Q_I.$$
Thus, if $M$ is chosen large enough,
$Q_I\subset \Delta_Q\subset Q\subset Q_j$,  a contradiction.
\end{proof}

We now come to the main lemma of this section.

\begin{lemma}\label{lemma4.main} Suppose that $\Omega$ is a 1-sided NTA
domain with ADR boundary. Fix $Q_0\in \dd$, and a pairwise disjoint family $\F\subset \dd_{Q_0}$,
and let $\Omega_{\F,Q_0}$ be the corresponding sawtooth domain.  Suppose also that
for some $\eta>0$, we have
\begin{equation}\label{eq4.main1}
\sup_{Q\in\dd_{Q_0}} \frac1{\sigma(Q)}\dint_{\Omega^{fat}_{\F,Q}}|\nabla^2\mathcal{S} 1(X)|^2 \delta(X)\,
dX\,\leq \,\eta.
\end{equation}
If $\eta\le \eta_0$ with $\eta_0$ small enough, depending only on $n,K_0$, and the Corkscrew,
Harnack Chain and ADR constants for $\Omega$, then for every $B:= B(x,r)$ and $\td:=B\cap
\partial\Omega_{\F,Q_0}$, with
$x\in \partial\Omega_{\F,Q_0}$ and $r\leq \diam(Q_0)$, there is  an exterior Corkscrew point $X_{\td}\in B\cap (\ree\setminus
\Omega_{\F,Q_0})$.   Moreover, the exterior Corkscrew constants depend only
upon $\eta_0$, $K_0$, and the
other parameters stated above.
\end{lemma}

To avoid confusion, we note that, as usual, $\delta(X):=\dist(X,\partial\Omega)$,
and $\Delta = B\cap\Omega$ denotes a surface ball on $\partial\Omega$;
we shall use the notation $\tdelta(X):=\dist(X,\partial \Omega_{\F,Q_0}),$
and $\td := B\cap \partial\Omega_{\F,Q_0}$.

\begin{proof}
We fix $x\in \partial\Omega_{\F,Q_0}$, and consider two separate cases.  Let $M$ be a
sufficiently large constant, to be chosen, whose value will remain fixed throughout the proof of the present lemma.

\noindent{\bf Case 1}:  $\dist(x,\partial\Omega)>r/(MK_0)$.

In this case, $x\in \partial I^*\cap J$, where as usual $I^* =(1+\lambda)I$ and  $I\in \W_{\F,Q_0}$ (cf. \eqref{eq3.4}),
and where $J\in\W$ with $\tau J\subset\ree\setminus \Omega_{\F,Q_0}$,
for some $\tau\in(1/2,1)$
(cf. \eqref{eq2.30*}).
By the nature of Whitney cubes, we have
$\ell(I)\approx \ell(J) \approx \dist(x,\partial\Omega) >r/(MK_0)$.
In this case, it is evident that there is a Corkscrew point in $J$, with $c \approx (MK_0)^{-1}$.

\noindent{\bf Case 2}:  $\dist(x,\partial\Omega)\leq r/(MK_0)$.

In this case, either $x\in \partial\Omega\cap\partial\Omega_{\F,Q_0}$, or
else $x$ lies on a face of $I^*$ for some
Whitney cube $I\in \W_{\F,Q_0}$, with $\ell(I)\lesssim r/(MK_0).$
In the former scenario,  by Proposition \ref{prop:sawtooth-contain} below,
we may
choose $Q\in\dd_{Q_0}$, with
$x\in \overline{Q}\subset B$,
and $\ell(Q) \approx r$.  If $Q\subseteq Q_j$, for some $Q_j \in \F$
(which might happen if $x\in\partial Q_j$), then by
Lemma \ref{lemma4.9} we immediately obtain the existence of the desired exterior
Corkscrew point for $\Omega_{\F,Q_0}$,
at the scale $r$.   Thus, in this scenario, it is enough to suppose that
$Q$ is not contained in any $Q_j\in\F$.

Otherwise, if $x\in\partial I^*$ for some $I\in \W_{\F,Q_0}$,
with $\ell(I)\lesssim r/(MK_0)$, then
there is a $Q_I\in \dd_{\F,Q_0}$ such that $\ell(Q_I)\approx \ell(I)$,
and $\dist(Q_I,I)\lesssim K_0\,\ell(I) \lesssim r/M.$  Consequently,
we have $\dist(I,Q)\lesssim  r/M$ for any $Q\in \dd$ with $Q_I\subseteq Q\subseteq Q_0.$
Choosing $M$ large enough, we may then fix such a $Q$ with $\ell(Q)\approx r$, and $Q\subset B$.
If $Q$ is contained in some $Q_j\in \F$, then by Lemma \ref{lemma4.9},
we again obtain the existence of an exterior Corkscrew point exactly as before.

Therefore, in either scenario, we have reduced matters to the following situation:
there is a $Q \in \dd_{\F,Q_0}$ (i.e., not contained in any $Q_j\in \F$), with
$\ell(Q) \approx r$, and $Q\subset B$.
Having fixed this $Q$, we recall that,
by Lemma \ref{lemma2.31},
there is a ball $B'_Q :=B(x_Q,s)$, with radius $s \approx (K_0)^{-1}\ell(Q)$,
such that \eqref{eq4.12a} holds.

By Lemma \ref{lemma2.30}, the sawtooth domain $\Omega_{\F,Q_0}$ inherits
the 1-sided NTA
(i.e., interior Corkscrew and Harnack Chain) and ADR properties from $\Omega$.
Thus, by Lemma \ref{lemma4.7}, applied with $\Omega_{\F,Q_0}$ in place
of  $\Omega$, and $B'_Q$ in place of $B$,
it is enough to establish the analogue of \eqref{eq4.8}
with $a$ depending only on the allowable parameters.

To this end, we proceed by a variant of the argument in \cite[pp. 254--256]{DS2}.
We remind the reader of the definition of the
family $\Fr$ (see the discussion at the beginning of Section \ref{spoincare}), and we also note that,
by construction, there is a purely dimensional constant $C_n$ such that
\begin{equation}\label{eq4.13}
T^{fat}_Q\subset B(x_Q, C_n K_0\, \ell(Q))=: B_Q^*
\end{equation}
Set $\Delta^*_Q :=B_Q^*\cap\partial \Omega$ and
$\Delta_Q':=B_Q'\cap \partial\Omega$, and
let $\Phi \in C^\infty_0(B'_Q)$,
with $\Phi \equiv 1$ on $B(x_Q,s/2)$, $0\leq\Phi\leq 1$, and $\|\nabla\Phi\|_\infty\lesssim s^{-1}.$
Let $\mathcal{L} := \nabla\cdot\nabla$
denote the usual Laplacian in $\ree$.
By the ADR property, and the fact that $s\approx\ell(Q)\approx r$, we have
\begin{multline}\label{eq4.main} r^{n+1} \approx
r\,\sigma\left(\frac12 \Delta'_Q \right) \leq r\int_{\partial\Omega} \Phi\, d\sigma
=r\,\langle -\mathcal{L}\,\mathcal{S}1,\Phi\rangle\\[4pt]
=r \iint_{\ree}\Big(\nabla \mathcal{S}1(X)-
\nabla \mathcal{S}1_{(2\Delta^*_Q)^c}(x_Q)- \vec{\alpha}\Big)\cdot \nabla\Phi(X)\,dX\\[4pt]
\lesssim \,\, \iint_{B'_Q}\big|\nabla \mathcal{S}1(X)-
\nabla \mathcal{S}1_{(2\Delta^*_Q)^c}(x_Q)- \vec{\alpha}\big|\,dX
\,=\,\iint_{\Omega\cap B'_Q}
\,\,+\,\,\iint_{\Omega_{ext}\cap B'_Q}\\[4pt]
=\,\iint_{\Omega_{\F(\epsilon r),Q}\cap B'_Q}\,\,+
\,\,\iint_{\big(\Omega_{\F,Q_0}\setminus \Omega_{\F(\epsilon r),Q}\big)\cap B'_Q}\,\,+
\,\,\iint_{\big(\Omega\setminus \Omega_{\F,Q_0}\big)\cap B'_Q}
\,\,+\,\,\iint_{\Omega_{ext}\cap B'_Q}\\[4pt]
=:\, I+II+III+IV\,,
\end{multline}
where $\vec{\alpha}$ is a constant vector at out disposal, $\epsilon >0$
is a small number to be determined, and where
as above, $\Omega_{ext}:= \ree\setminus \overline{\Omega}.$

We now set
$$\vec{\alpha}:= \frac1{|\Omega_{\F(\epsilon r),Q}|}
\iint_{\Omega_{\F(\epsilon r),Q}}\Big(\nabla \mathcal{S}1(X)-
\nabla \mathcal{S}1_{(2\Delta_Q^*)^c}(x_Q)\Big)\,dX.$$
We note for future reference that by standard Calder\'on-Zygmund estimates,
\begin{equation}\label{eq4.14}
\big|\nabla \mathcal{S}1_{(2\Delta^*_Q)^c}(X)-
\nabla \mathcal{S}1_{(2\Delta_Q^*)^c}(x_Q)\big|\,\leq \, C,\qquad \forall X\in B^*_Q.
\end{equation}
We also note that by Lemma \ref{lemma2.30}, the sawtooth domain $\Omega_{\F(\epsilon r),Q}$,
if non-empty, must contain a Corkscrew point at the scale of $\ell(Q)\approx r$,
so that, in particular,
$$r^{n+1}\lesssim |\Omega_{\F(\epsilon r),Q}|.$$
Consequently, by \eqref{eq4.13} and the fact that $\Omega_{\F,Q}\subset T_Q\subset T_Q^{fat}$
for any pairwise disjoint family $\F$ and every $Q\in \dd$, we have
\begin{equation}\label{eq4.12}
|\vec{\alpha}|\,
\leq \,\frac{C_{K_0}}{|B^*_Q|}
\iint_{B_Q^*}\big|\nabla \mathcal{S}1_{2\Delta^*_Q}(X)\big|\,dX
\,+\, C\,\leq C_{K_0}\,,
\end{equation}
where in the last step we have used Lemma \ref{lemma4.1}.

By the Poincar\'e inequality (Lemma \ref{poincare}), \eqref{eq4.13}, and \eqref{eq3.2} with
$\rho =\epsilon r$, we obtain
\begin{align*}
I
& \leq C_\epsilon r\iint_{\Omega^{fat}_{\F(\epsilon r),Q}}|\nabla^2\mathcal{S} 1(X)| \,
dX 
\\[4pt]
&\leq C_{\epsilon,K_0}\, r^{(n+3)/2}\,
\left(\iint_{\Omega^{fat}_{\F(\epsilon r),Q}}|\nabla^2\mathcal{S} 1(X)|^2 \,
dX \right)^{1/2}
\\[4pt] 
&\leq C_{\epsilon,K_0}\, r^{(n+2)/2}
\left(\iint_{\Omega^{fat}_{\F(\epsilon r),Q}}|\nabla^2\mathcal{S} 1(X)|^2 \delta(X)\,
dX\right)^{1/2}
\\
&\leq C_{\epsilon,K_0}\sqrt{\eta}\,r^{n+1}\,,
\end{align*}
by hypothesis \eqref{eq4.main1}, since $\ell(Q)\approx r$, and $\Omega^{fat}_{\F(\epsilon r),Q}\subset
\Omega^{fat}_{\F,Q}$.

Next, we claim that, for each $\gamma\in (0,1/(n+1))$, we have
\begin{equation}\label{eq4.17} II \leq C_{\gamma,K_0}\, \epsilon^\gamma r^{n+1}.
\end{equation}
We defer the proof of this claim momentarily, and observe that
$$III + IV = \iint_{\big(\ree\setminus\overline{\Omega_{\F,Q_0}}\big)\cap B'_Q}
\big|\nabla \mathcal{S}1(X)-
\nabla \mathcal{S}1_{(2\Delta^*_Q)^c}(x_Q)- \vec{\alpha}\big|\,dX$$
(to avoid possible confusion, we point out that the boundaries of $\Omega$ and all
of its sub-domains that we consider here, have $(n+1)$-dimensional Lebesgue measure
equal to zero).  Then by \eqref{eq4.14}, \eqref{eq4.12}, H\"older's inequality and Lemma
\ref{lemma4.1}, we deduce that for any $q\in (1,(n+1)/n)$,
$$III + IV\,\leq\, C_{K_0} \left(
\big|\big(\ree\setminus\overline{\Omega_{\F,Q_0}}\big)\cap B'_Q\big|^{1/q'} r^{(n+1)/q}
\,+\,\big|\big(\ree\setminus\overline{\Omega_{\F,Q_0}}\big)\cap B'_Q\big|\right).$$
Now, choosing first $\epsilon$, and then $\eta$ sufficiently small, we can hide $I+II$
on the left hand side of \eqref{eq4.main}.  Our estimate for $III+IV$ then implies that
$$r^{n+1} \,\leq\, C_{K_0}\,\big|\big(\ree\setminus\overline{\Omega_{\F,Q_0}}\big)\cap B'_Q\big|.$$
As noted above, the existence of an exterior Corkscrew point now follows
by applying Lemma \ref{lemma4.7}, with $B_Q'$ in place of $B$, and $\Omega_{\F,Q_0}$
in place of $\Omega$.

To complete the proof of Lemma \ref{lemma4.main}, it remains only to prove the claimed estimate
\eqref{eq4.17}.  By \eqref{eq4.12a}, we may replace $\Omega_{\F,Q_0}$ by $\Omega_{\F,Q}$ in the domain of integration which defines $II$.  Consequently, by \eqref{eq3.3} with $\rho =\epsilon r$,
we have that
$$II\leq \int_{\Sigma_{C\epsilon r}\cap B_Q'}\big|\nabla \mathcal{S}1(X)-
\nabla \mathcal{S}1_{(2\Delta^*_Q)^c}(x_Q)- \vec{\alpha}\big|\,dX$$
where $\Sigma_{\rho}:=\{X\in\ree: \delta(X)<\rho\}$.
The desired bound now follows readily from
\eqref{eq4.14}, \eqref{eq4.12}, Lemma \ref{lemma4.3} and Corollary \ref{corsioball}.
We omit the routine details.
\end{proof}

We conclude this section with an estimate for harmonic measure
in ``good'' sawtooth regions (that is, those for which \eqref{eq4.main1} holds for
sufficiently small $\eta$).  Given a subdomain $\Omega'\subset\Omega$,
we shall use the notational convention that
$\tom^X$ denotes harmonic measure for $\Omega'$ with pole at $X$,
when there is no chance for confusion.
\begin{corollary}\label{cor4.18}
Suppose that $\Omega$ is a 1-sided NTA domain with ADR boundary. Suppose also that \eqref{eq4.main1} holds for some $Q_0\in \dd$, and
some pairwise disjoint family $\F\subset \dd_{Q_0}$, with $\eta\le \eta_0$  (cf. Lemma \ref{lemma4.main}). Let $\tom^X$ denote harmonic measure
for $\Omega_{\F,Q_0}$ with pole at $X$. Then, for every $x\in \partial \Omega_{\F,Q_0}$,
every $r\leq \diam(Q_0)$,
and every surface ball $\td= \td(x,r)$, the harmonic measure
$\tom^{X_{\td}}$ belongs to $A_\infty(\td)$  (cf. Definition \ref{def1.ainfty}),
with uniform $A_\infty$ constants depending only upon dimension and
the ADR, Harnack Chain and Corkscrew
constants, including $K_0$.
\end{corollary}

\begin{proof}
By Lemma \ref{lemma2.30}, $\Omega_{\F,Q_0}$ is a 1-sided NTA domain
with ADR boundary. Moreover, by Lemma \ref{lemma4.main},
it also satisfies an exterior Corkscrew condition.  The conclusion of the corollary now
follows immediately by \cite[Theorem 2]{DJe}.
\end{proof}

\section{$\F$-Projections and a Dahlberg-Jerison-Kenig ``Sawtooth Lemma''}\label{s6*}

In this section, we present a dyadic version of the main lemma of \cite{DJK}.
Our approach here is modeled
on an analogous result in the Euclidean case which appeared in our previous work \cite{HM} (see also \cite{HM-ANU}).
As in \cite{HM}, we shall utilize a certain projection operator adapted to a pairwise
disjoint family $\F$.

Consider now such a family $\F=\{Q_j\}\subset \dd$.
The projection operator $\P_\F$ associated to $\F$
(the ``$\F$-projection operator'') is defined by:
$$   \P_\F f(x):=f(x)\,1_{\partial\Omega\setminus (\cup_\F Q_j)}(x)+
\sum_\F\left(\,\fint_{Q_j} f\,d\sigma \right)\,1_{Q_j}(x).$$
We may naturally extend $\P_\F$ to act on non-negative Borel measures on $\partial\Omega$.
Suppose that  $\mu$ is such a measure, and let $A\subset \partial \Omega$.   We then define the measure $\P_\F\,\mu$ as follows:
$$\P_\F\, \mu(A) : = \int_{\partial\Omega} \P_\F\left(1_A\right) d\mu
=  \mu (A\setminus \cup_\F Q_j)+\sum_\F \frac{\sigma(A\cap Q_j)}{\sigma (Q_j)}\,\mu(Q_j).$$
In particular, we have that $\P_\F\, \mu(Q)=\mu(Q)$, for every $Q\in\dd_\F$ (i.e.,
for $Q$ not contained in any $Q_j\in\F$), and also that
$\P_\F\, \mu(Q_j)=\mu(Q_j)$ for every $Q_j\in\F$.

We shall prove a version of the main lemma in \cite{DJK} which is valid
for $\F$-projections of harmonic measure.
Our proof follows the idea of the argument in \cite{DJK}, but is technically  simpler
(given certain geometric preliminaries), owing to the dyadic setting in which we work here.  In more precise detail, we follow our earlier Euclidean version of this
lemma, which appears in \cite[Lemma A.1]{HM}.

Let us set a bit of notation:  given  $Q_0\in\dd$, a pairwise disjoint family $\F\subset\dd$, and
the corresponding sawtooth domain $\Omega_{\F,Q_0}$
(cf. \eqref{eq2.discretecarl}-\eqref{eq2.whitney3} and
\eqref{eq2.sawtooth2}; also \eqref{eq3.4} and \eqref{eq3.saw}), we let
$\td$, $\tdelta$, and $\tom^X$ denote, respectively, a surface ball
on $\partial\Omega_{\F,Q_0}$, the distance to the boundary of $\partial\Omega_{\F,Q_0}$,
and harmonic measure for the domain $\Omega_{\F,Q_0}$ with pole at $X$; i.e.,
for $x\in \partial\Omega_{\F,Q_0}$,
$\td(x,r)=B(x,r)\cap\partial\Omega_{\F,Q_0}$, and for $X\in \Omega_{\F,Q_0}$,
$\tdelta(X):=\dist(X,\partial\Omega_{\F,Q_0})$.
We continue to use $\Delta=\Delta(x,r), \delta(X)$ and
$\omega^X$ to denote the analogous objects in reference to
the original domain $\Omega$ and its boundary.

Before stating our sawtooth lemma, let us record some useful geometric observations.
We recall that by Lemma \ref{lemma2.30}, the sawtooth domain $\Omega_{\F,Q_0}$ inherits the
1-sided NTA (i.e., interior Corkscrew and Harnack Chain)
and ADR properties from $\Omega$.
We begin with the following.
\begin{proposition}\label{prop:sawtooth-contain}  Suppose that $\Omega$ is a
1-sided NTA domain with ADR boundary.
Fix $Q_0\in \dd$, and let $\F\subset\dd_{Q_0}$ be a disjoint family.  Then
\begin{equation}\label{eq5.0}
Q_0\setminus \left(\cup_\F Q_j\right)\subset\partial\Omega\cap\partial\Omega_{\F,Q_0}
\subset \overline{Q_0} \setminus \left(\cup_\F \,\,{\rm int}\!\left(Q_j\right)\right)
\end{equation}
\end{proposition}

\begin{proof} We first prove the right hand containment. Suppose that
$x\in\partial\Omega\cap\partial\Omega_{\F,Q_0}$.  Then there is a sequence
$X^k\in\Omega_{\F,Q_0}$, with $X^k\to x$.  By definition of $\Omega_{\F,Q_0}$,
each $X^k$ is contained in $I^*_k$ for some $I_k\in \W_{\F,Q_0}$ (cf. \eqref{eq3.4}-\eqref{eq3.saw}),
so that $\ell(I_k) \approx \delta(X^k)\to 0$.  Moreover, again by definition, each $I_k$
belongs to some $\W^*_{Q^k}$, $Q^k\in\dd_\F$ so that,
$$\dist(Q^k,I_k)\lesssim K_0\,\ell(Q^k)
\approx K_0\,\ell(I_k) \to 0.$$
Consequently, $\dist(Q^k,x)\to 0$.  Since each $Q^k\subset Q_0$, we have $x\in \overline{Q_0}$.
On the other hand, if $x\in {\rm int}(Q_j)$, for some $Q_j\in\F$,  then there is an $\epsilon >0$ such that
$\dist(x,Q)>\epsilon$  for every $Q\in \dd_{\F,Q_0}\,$ with $\ell(Q)\ll\epsilon$, because no
$Q\in\dd_{\F,Q_0}$ can be contained in any $Q_j$.  Since this cannot happen if
$\ell(Q^k) +\dist(Q^k,x)\to 0$, the right hand containment is established.

Now suppose that $x\in Q_0\setminus (\cup_\F Q_j)$.  By definition,
if $x\in Q\in\dd_{Q_0}$, then
$Q\in \dd_{\F,Q_0}$.   Therefore, we may choose a sequence $\{Q^k\}\subset\dd_{\F,Q_0}$
shrinking to $x$,
whence there exist $I_k\in \W^*_{Q^k}\subset\W_{\F,Q_0}$ with $\dist(I_k,x)\to 0$.
The left hand containment now follows.
\end{proof}

\begin{proposition}\label{prop5.0a}  Suppose that $\partial\Omega$ is ADR, and that $\mu$ is a doubling measure on
$\partial\Omega$; i.e, there is a uniform constant $M_0$ such that
$\mu(2\Delta)\leq M_0\,\mu(\Delta)$ for every surface ball $\Delta$.
Then $\partial Q:=\overline{Q}\setminus {\rm int}(Q)$
has $\mu$-measure 0, for every $Q\in\dd$.  In particular,
the sets in \eqref{eq5.0} have the same $\mu$ measure.
\end{proposition}

\begin{proof}  The argument is a refinement of that in \cite[p. 403]{GR}, where the Euclidean case was
treated.
Fix an integer $k$, a cube $Q\in \dd_k$,
and a positive integer $m$ to be chosen.   We set $$\{Q^1_i\}:=\dd^1:=\dd_Q\cap \dd_{k+m}\,,$$
and make the
disjoint decomposition
$Q=\cup Q^1_i.$  We then split $\dd^1=\dd^{1,1} \cup \dd^{1,2}$, where
$Q^1_i\in \dd^{1,1}$ if $\partial Q_i^1$ meets $\partial Q$, and
$Q^1_i\in \dd^{1,2}$ otherwise.  We then write $\overline{Q}=R^{1,1}\cup R^{1,2}$, where
$$R^{1,1}:= \cup_{\dd^{1,1}} \widehat{Q}^1_i\,,\qquad R^{1,2}:=
\cup_{\dd^{1,2}} Q^1_i,$$
and for each cube $Q^1_i\in\dd^{1,1}$, we construct $\widehat{Q}^1_i$ as follows.
We enumerate the elements in $\dd^{1,1}$ as $Q^1_{i_1},Q^1_{i_2},\dots,Q^1_{i_N}$, and then set $(Q^1_i)^*=Q^1_i\cup(\partial Q^1_i\cap \partial Q)$ and
$$
\widehat{Q}^1_{i_1}:=(Q^1_{i_1})^*,
\quad
\widehat{Q}^1_{i_2}:=(Q^1_{i_2})^*\setminus(Q^1_{i_1})^*,
\quad
\widehat{Q}^1_{i_3}:=(Q^1_{i_3})^*\setminus((Q^1_{i_1})^*\cup (Q^1_{i_2})^*),\ \dots
$$
so that $R^{1,1}$ covers $\partial Q$ and the modified cubes $\widehat{Q}^1_i$
are pairwise disjoint.

We recall the surface ball $\Delta_Q=\Delta(x_Q,r)\subset Q$, with $r\approx \ell(Q)$
as in \eqref{cube-ball}-\eqref{cube-ball2}.  Then
$$\dist\Big(\Delta(x_Q,r/2),\partial Q\Big) \geq \frac{r}2\geq c_0\,\ell(Q)=c_0\,2^{-k}\,,$$
for some uniform constant $c_0$.  By Lemma \ref{lemmaCh}, there is a uniform constant
$C_1$ such that $\diam(Q')\leq C_1\ell(Q')$, for every $Q'\in\dd$.
We may therefore choose $m$ depending only on the ADR constants and dimension so that
$2^{-m}<c_0/C_1$, whence
$$\diam(Q^1_i)\leq C_12^{-k-m}< c_0\,2^{-k}.$$  Consequently, $R^{1,1}$ misses $\Delta(x_Q,r/2)$,
so that by the doubling property,
$$\mu(\overline{Q})\leq C_{M_0}\,\mu(\Delta(x_Q,r/2)\leq C_{M_0}\,\mu(R^{1,2}).$$
Since $R^{1,1}$ and $R^{1,2}$ are disjoint, the latter estimate yields
$$\mu(R^{1,1})\leq \Big(1-\frac1{C_{M_0}}\Big)\,\mu(\overline{Q})=:\theta\,\mu(\overline{Q}),$$
where we note that $\theta<1$.

Let us now repeat this procedure, decomposing $\widehat{Q}^1_i$ for each $Q_i^1\in\dd^{1,1}$. We set $\dd^2(Q^1_i)=\dd_{Q^1_i}\cap \dd_{k+2m}$ and split it into $\dd^{2,1}(Q^1_i)$ and $\dd^{2,2}(Q^1_i)$ where $Q'\in \dd^{2,1}(Q^1_i)$ if
$\partial Q'$ meets $\partial Q\cap \widehat{Q}^1_i$ (this set plays the role of $\partial Q$ in the previous step). Associated to any $Q'\in \dd^{2,1}(Q^1_i)$ we set $(Q')^*=(Q'\cap \widehat{Q}^1_i)\cup(\partial Q'\cap (\partial Q\cap \widehat{Q}^1_i))$. Then we make these sets disjoint as before and we have that $R^{2,1}(Q^1_i)$ is defined as the disjoint union of the corresponding $\widehat{Q'}$. Note that $\widehat{Q}^1_i= R^{2,1}(Q^1_i)\cup R^{2,2}(Q^1_i)$ and this a disjoint union. As before, $R^{2,1}(Q^1_i)$ misses
$(1/2)\Delta_{Q_i^1}$
so that  by the doubling property
$$
\mu(\widehat{Q}^1_i)\leq C_{M_0}\,\mu\left(\frac12\Delta_{Q_i^1}
\right)\leq C_{M_0}\,\mu(R^{2,2}(Q^1_i))
$$
and then $\mu(R^{2,1})\leq \theta\,\mu(\widehat{Q}^1_i).$ Next we set $R^{2,1}$ and $R^{2,2}$ as the union of the corresponding $R^{2,1}(Q^1_i)$ and $R^{2,2}(Q^1_i)$ with $Q_i^1\in\dd^{1,1}$. Then,
\begin{multline*}
\mu\big(R^{2,1}\big)
:=
\mu \Big(\bigcup_{Q_i^1\in\dd^{1,1}} R^{2,1}(Q^1_i)\Big)
=
\sum_{Q_i^1\in\dd^{1,1}} \mu\big(R^{2,1}(Q^1_i)\big)
\\
\leq
\theta
\sum_{Q_i^1\in\dd^{1,1}} \mu(\widehat{Q}^1_i)
=
\theta\,\mu(R^{1,1})
\le
\theta^2\,\mu(\overline{Q}).
\end{multline*}
A straightforward iteration argument now yields that $\mu(\partial Q) = 0$.   We omit the details.
\end{proof}

\begin{proposition}\label{prop:cork-both}
Suppose that $\Omega$ is a 1-sided NTA domain with ADR boundary.
Fix $Q_0\in \dd$, and let $\F\subset\dd_{Q_0}$ be a disjoint family.  Then for each $Q\in\dd_{\F,Q_0}$,
there is a radius $r_Q \approx K_0\,\ell(Q)$, and a point $A_Q\in\Omega_{\F,Q_0}$ which serves as a Corkscrew point simultaneously for $\Omega_{\F,Q_0}$,
with respect to the surface ball $\td(y_Q,r_Q)$, for some $y_Q\in \partial \Omega_{\F,Q_0}$,
and for  $\Omega$, with respect to each surface ball
$\Delta(x,r_Q)$, for every $x\in Q$.  \end{proposition}

\begin{proof}
Let $Q\in \dd_{\F,Q_0}.$  Recall that by construction, $\W_Q$ is non-empty.
It follows that there is an $I$ for which
$\ell(I) \approx \ell(Q)$ and $\dist(Q,I)\leq C_0\,\ell(Q)$.  Furthermore,
$I\subset \Omega_{\F,Q_0}$, and
$\dist(I,\partial\Omega_{\F,Q_0})\leq \dist(I,\partial\Omega)\approx\ell(I)$.
We let $A_Q$ denote the center of this particular $I$, so that
\begin{equation}\label{eq5.1aa}
\ell(Q)\approx \dist(A_Q,\partial\Omega_{\F,Q_0}) \leq \dist(A_Q,Q)\lesssim C_0\,\ell(Q).
\end{equation}
Fix $y_Q\in \partial\Omega_{\F,Q_0}$ so that $\dist(A_Q,\partial\Omega_{\F,Q_0})=
|A_Q-y_Q|$.
Then $A_Q$ is the promised simultaneous Corkscrew point, for
$r_Q \approx K_0\,\ell(Q)\geq C_0\,\ell(Q)$.
\end{proof}

\begin{corollary}\label{cor5.5} The point $A_{Q_0}$
is a Corkscrew point with respect to $\td(x,r_{Q_0})$, for all $x\in \partial\Omega_{\F,Q_0}$, and for
$\Delta(x,r_{Q_0})$, for all $x\in Q_0$, with $r_{Q_0} \approx K_0\,\ell(Q_0)$.
\end{corollary}

The proof is almost immediate, since $\diam(\Omega_{\F,Q_0})\lesssim K_0 \,\ell(Q_0)$, and we omit it.

\begin{proposition}\label{prop:Pj}  Suppose that
$\Omega$ is a 1-sided NTA domain with ADR boundary.
Fix $Q_0\in \dd$, and let $\F\subset\dd_{Q_0}$ be a disjoint family.  Then for each
$Q_j\in\F$, there is an $n$-dimensional cube $P_j\subset \partial\Omega_{\F,Q_0}$, which is
contained in a face of $I^*$, for some $I\in\W$, and which satisfies
\begin{equation}\label{eq5.1a}
\ell(P_j)\approx \dist(P_j,Q_j)\approx \dist(P_j,\partial\Omega) \approx \ell(I)\approx \ell(Q_j),
\end{equation}
where the uniform implicit constants are allowed to depend upon $K_0$.
\end{proposition}

\begin{proof}  Fix $Q_j\in\F$. It follows from Lemma \ref{lemma4.9} (with $Q=Q_j$)
and the Corkscrew condition that there is an $I_1\in\W$ with
$I_1\subset \Omega\setminus\Omega_{\F,Q_0}$,
$\ell(I_1)\approx\ell(Q_j)/K_0$, and
$\dist(I_1,Q_j)\lesssim \ell(Q_j)/K_0$.  On the other hand,
the dyadic parent $\widetilde{Q}_j$
of $Q_j$ belongs to $\dd_{\F,Q_0}$, so there is an $I_2\in\W^*_{\widetilde{Q}_j}$ with
$I_2\subset \Omega_{\F,Q_0}$,
$\ell(I_2)\approx \ell(Q_j)$, and $\dist(Q_j,I_2)\lesssim K_0\,\ell(Q_j)$.
The Harnack Chain (in $\Omega$) connecting the centers
of $I_1$ and $I_2$, then passes through $\partial\Omega_{\F,Q_0}$, and maintains a distance to
$\partial\Omega$ on the order of $\ell(Q_j)$.  Consequently, there is an interface between some
pair $I,J\in\W$, with $\interior(I^*)\subset\Omega_{\F,Q_0}$ and
$J\notin \W_Q^*$, for any $Q\in \dd_{\F,Q_0}$
(so that $\tau J\subset \Omega\setminus\Omega_{\F,Q_0}$
for some $\tau\in(1/2,1)$; cf. \eqref{eq2.30*}),
and
$$
\dist(I,Q_j)\approx\dist(J,Q_j)\approx\ell(I)\approx\ell(J)\approx\ell(Q_j)
$$
(here, some of the implicit constants may depend upon $K_0$).  Of course, the interface between
$I$ and $J$ is precisely one face of the smaller of these two cubes.  Therefore,
if $\lambda$ is chosen small enough, then
$\partial I^* \cap J$
contains an $n$-dimensional cube $P_j$ with the stated properties.
\end{proof}

\begin{remark}\label{remark:Pj}
It follows from the proof that if $P_j\cap P_k$ then $\ell(Q_j)\approx \ell(Q_k)$ since two
adjacent Whitney cubes have comparable side length.
Thus, $\dist(Q_j,Q_k)\lesssim \ell(Q_j)$ and therefore we have the bounded overlap property
$$
\sum_j 1_{P_j}(x)
\le C\,,
$$
with $C$ depending on the ADR constants.\end{remark}

For future reference, we note that,
under the assumptions of Proposition \ref{prop:Pj}, if  $x^\star_j$ denotes the center
of $P_j$, then for an appropriate choice of $r_j\approx K_0\,\ell(Q_j)$, we have $P_j\subset \td(x^\star_j,r_j)$ and
\begin{equation}\label{eqn:T-Pj}
\overline{T_{Q_j}} \subset B(x^\star_j,r_j),
\end{equation}
since $\diam(T_{Q_j})\lesssim K_0\,\ell(Q_j)$.
Moreover, given $Q\in\dd_{\F,Q_0}$  and $r_Q\approx K_0\,\ell(Q)$ from
Proposition \ref{prop:cork-both}, by choosing $\widehat{r}_Q\approx r_Q$
(with implicit constants depending on $K_0$) we may suppose that
\begin{equation}\label{eqn:Pj-cork}
Q \cup \left(\cup_{Q_j\in\F:\,Q_j\subset Q}B(x^\star_j,r_j)\right)\subset B(y_Q,\widehat{r}_Q)\,.
\end{equation}
Here, $y_Q$ is the center of $\td(y_Q,r_Q)\subset\partial\Omega_{\F,Q_0}$, appearing in
Proposition \ref{prop:cork-both}.  We omit the routine geometric argument.

We conclude this preamble with the following.
\begin{proposition}\label{prop:surface-ball-sawtooth}
Suppose that $\Omega$ is a 1-sided NTA domain
with ADR boundary.
Fix $Q_0\in \dd$, and let $\F\subset\dd_{Q_0}$ be a disjoint family.
For $Q_j\in\F$, let $B(x_j^\star,r_j)$
be the ball, concentric with $P_j$, satisfying \eqref{eqn:T-Pj}.
Then for each $Q\in \dd_{\F,Q_0}$, there is a surface ball
$$\td^Q:=\td(x_Q^\star, t_Q)\subset \left(Q\cap\partial\Omega_{\F,Q_0}\right)
\cup\left(\cup_{Q_j\in\F:\, Q_j\subset Q} \left(B(x_j^\star,r_j)\cap\partial
\Omega_{\F,Q_0}\right)\right)\,,$$
with $t_Q\approx\ell(Q)$, $x_Q^\star\in \partial\Omega_{\F,Q_0}$, and $\dist(Q,\td^Q)\lesssim \ell(Q)$,
where the implicit constants may depend upon $K_0$.
\end{proposition}

\begin{proof}  Suppose first that there is some $Q_{j_0}\subset Q$, for which
$\ell(Q_{j_0})\geq \ell(Q)/M$, where $M$ is a sufficiently large number to
be chosen.
We  then set $\td^Q=\td(x_{j_0}^\star,\ell(P_{j_0})/2)$,
a surface ball contained in the cube $P_{j_0}$ whose existence was
established in Proposition \ref{prop:Pj}.

Now suppose that $\ell(Q_j)<\ell(Q)/M$, for every $Q_j\subset Q$.  By Lemma
\ref{lemma2.31}, there is a ball $B_Q'=B(x_Q,s)$ with $s\approx\ell(Q)$
and $B_Q'\cap\Omega\subset T_Q\subseteq T_{Q_0}$.   In particular,
$B_Q'$ misses $\partial T_{Q_0}\setminus Q_0$.  Moreover,
$\Delta'_Q:=B_Q'\cap\partial\Omega\subset Q$.  Consider those $Q_j\subset Q$ which meet
$\Delta(x_Q,s/(4\sqrt{M}))$.  If there are no such $Q_j$, then we set $\td^Q=\Delta(x_Q,s/(4\sqrt{M}))$,
which in this case is contained in $Q\cap\partial\Omega_{\F,Q_0}$ by Proposition \ref{prop:sawtooth-contain}.
On the other hand, suppose that there is some $Q_{j_0}\subset Q$ which meets
$\Delta(x_Q,s/(4\sqrt{M}))$.
Then for $M$ large enough, depending on $K_0$, we have
$P_{j_0}\subset B(x_Q, s/(2\sqrt{M})),$ and thus also
\begin{equation}\label{eq5.3}
\td^Q\,:=\,\Delta_\star(x^\star_{j_0},s/(2\sqrt{M}))
\,\subset \,B(x_Q,s/\sqrt{M})\,\subset\, B_Q',
\end{equation}
by the triangle inequality.
Consequently, $\td^Q$ misses $\partial T_{Q_0}\setminus Q_0$,
and $\td^Q\cap\partial\Omega\subset Q$,
by the properties of $B_Q'$.  Moreover, we claim that
\begin{equation}\label{eq6.setcontainment}
\partial\Omega_{\F,Q_0}\subset \left(\partial T_{Q_0}\setminus Q_0\right)
\cup\left(\partial\Omega_{\F,Q_0}\cap Q_0\right) \cup \left( \cup_{Q_j\in\F}
\big(\partial\Omega_{\F,Q_0}\cap\overline{T_{Q_j}}\,\big)\right).
\end{equation}
Let us defer for the moment the proof of this claim.  Given \eqref{eq6.setcontainment},
by \eqref{eqn:T-Pj} and properties of dyadic cubes,
it is enough to verify that if $\td^Q$ meets $\overline{T_{Q_j}}$, for some $Q_j\in\F$,
then $Q_j$ meets $Q$.   Suppose now that $\td^Q$ meets $\overline{T_{Q_j}}$.  By \eqref{eq5.3}
and the definition of $\overline{T_{Q_j}}$, this means that there is a $Q'\subseteq Q_j$,
and an $I\in \W^*_{Q'}$ such that $I^*$ meets $B(x_Q,s/\sqrt{M})$.  It follows that
$$\ell(Q')\approx\ell(I)\approx\dist(I^*,\partial\Omega)\leq\dist(I^*,x_Q)\leq s/\sqrt{M}.$$
Since $\dist(I^*,Q')\lesssim K_0\,\ell(Q')$, by the triangle inequality we have
$$|y-x_Q|\lesssim K_0\, s/\sqrt{M} < s\,,\qquad \forall y\in Q',$$
if $M \gg (K_0)^2$;  i.e., $Q'\subset \Delta_Q' := \Delta(x_Q,s)\subset Q$, whence $Q_j$ meets $Q$,
as desired.

Finally, we establish \eqref{eq6.setcontainment}.
Let $X\in \partial\Omega_{\F,Q_0}$.  There are two cases.

\smallskip

\noindent{\bf Case 1}:  $\delta(X)= 0$.  If $X\in Q_0$ we are done.  Otherwise, since
$$\partial\Omega_{\F,Q_0}\subset\overline{\Omega_{\F,Q_0}}\subset \overline{T_{Q_0}}\,,$$
it suffices to show that $X\notin T_{Q_0}$.  But this is trivial, since $T_{Q_0} \subset \Omega$,
and for $X\in \Omega$, we have that $\delta(X)>0$.

\smallskip

\noindent{\bf  Case 2}: $\delta(X)>0$.  If $X\in \overline{T_{Q_j}}$ for some $j$, we are done, so suppose that this never happens.  As in Case 1, it is enough to show that $X\notin T_{Q_0}$, so suppose by way of contradiction that $X\in T_{Q_0}$.  Since $T_{Q_0}$ is open, this means that there is a small
number $\eps_0\ll\delta(X)$ such that the
ball $B(X,\eps)\subset T_{Q_0}$, whenever $\eps\leq\eps_0$.
By definition of $T_{Q_0}$, and properties
of Whitney cubes,
there exist a uniformly bounded number of Whitney cubes, say, $I_1,\dots,I_M$,
such that $$B(X,\eps_0)\subset \cup_{k=1}^M I_k^*\,,$$
and for each $k\in [1,M]$, there is a $Q^k\in \dd_{Q_0}$ with $I_k\in \W_{Q^k}^*$.
It is possible that for a smaller $\eps$, there may be a smaller collection of $I_k$'s required
to cover $B(X,\eps)$, but these $I_k$'s are of course always chosen from the original collection
(i.e., the one for $\eps_0$.)  Observe that since $B(X,\eps)$ is open, if $B(X,\eps)$ meets
$I^*_k$, then it meets $\interior(I^*_k)$.
For a given $\eps$, we may assume that the covering collection is ``minimal'' in the sense that
$B(X,\eps)$ meets $\interior(I_k^*)$ for each $k$, i.e., we
remove those $I_k^*$ which do not meet  $B(X,\eps)$.

We claim that there must be some $\eps>0$ and a corresponding  ``minimal'' collection,
with the property that
each $Q^k\in\dd_{\F,Q_0}$.   Indeed, if not, then there is a sequence
$\eps_i\to 0$, and for each $i$,  a $Q_j(i)\in\F$, and a $k(i)\in[1,M]$,
such that $Q^{k(i)}\in \dd_{Q_j(i)}$.  Since there were only a bounded number of $I_k$'s
and thus also $Q^k$'s, to start with,
there must be some subsequence, again call it $\eps_i$, such that $k(i) = constant$.
But this means that there is a fixed $Q_j\in\F$, and a sequence $\eps_i$ such that
$B(X,\eps_i)$ meets $T_{Q_j}$, which contradicts our assumption that $X\notin \overline{T_{Q_j}}$
for any $Q_j\in \F$.  This proves the claim.

We now choose $\eps$ as in the claim, and observe that we then have
$$B(X,\eps)\subset \bigcup_{Q\in \dd_{\F,Q_0}}\bigcup_{\W_Q^*}I^*\,,$$
i.e., that $X$ is an {\it interior} point for the set $\bigcup_{Q\in \dd_{\F,Q_0}}\bigcup_{\W_Q^*}I^*$.
But by definition, this means that $X\in \Omega_{\F,Q_0}$, and since the latter set is open,
this contradicts that $X\in \partial \Omega_{\F,Q_0}$.
\end{proof}

\begin{lemma}[Dyadic sawtooth lemma for projections]\label{lemma:DJK-dyadic-proj}
Suppose that $\Omega$ is a 1-sided NTA domain with ADR boundary and that  it also satisfies the qualitative exterior corkscrew condition.
Fix $Q_0\in\dd$,  let $\F=\{Q_j\}\subset \dd_{Q_0}$ be a family of pairwise disjoint dyadic cubes and let $\P_\F$ be the corresponding projection operator.
We write $\hm=\hm^{X_0}$  and $\tom=\tom^{X_0}$ to denote the respective
harmonic measures for the domains $\Omega$ and $\Omega_{\F,Q_0}$,
with fixed pole at the corkscrew point $X_0:=A_{Q_0}$ whose existence was noted
in Proposition \ref{prop:cork-both} and Corollary \ref{cor5.5}.
Let $\nu=\nu^{X_0}$ be the measure defined by
\begin{equation}\label{defi:nu-bar}
\nu(F)
=\tom\left(F\setminus(\cup_{\F} Q_j)\right) +\sum_{Q_j\in\F} \frac{\hm(F\cap Q_j)}{\hm(Q_j)}\,
\tom(P_j),
\qquad F\subset Q_0,
\end{equation}
where $P_j$ is the $n$-dimensional cube produced by Proposition \ref{prop:Pj}.
Then $\P_\F\nu$ depends only on $\tom$ and
not on $\hm$.  More precisely,
\begin{equation}\label{defi-proj-nubar}
\P_\F \nu(F)
=
\tom\left(F\setminus(\cup_{ \F} Q_j)\right)+\sum_{Q_j\in\F} \frac{\sigma(F\cap Q_j)}{\sigma(Q_j)}\,
\tom(P_j),
\qquad F\subset Q_0.
\end{equation}
Moreover, there exists $\theta>0$ such that for all  $Q\in\dd_{Q_0}$ and $F\subset Q$, we have
\begin{equation}\label{dyadic-DJK:proj}
\left(\frac{\P_\F \hm(F)}{\P_\F \hm(Q)}\right)^\theta
\lesssim
\frac{\P_\F \nu(F)}{\P_\F \nu(Q)}
\lesssim
\frac{\P_\F \hm(F)}{\P_\F \hm(Q)}.
\end{equation}
\end{lemma}

\begin{proof}
We observe that \eqref{defi-proj-nubar} follows immediately
from the definitions of $\P_\F$ and $\nu$, as the reader may readily verify.
We omit the details.

Our first main task is to establish
the righthand side inequality in \eqref{dyadic-DJK:proj}. Let us fix $Q\in\dd_{Q_0}$, $F\subset Q$.

\noindent {\bf Case 1}: There exists $Q_j\in \F$ such that $Q\subset Q_j$.
Note that by \eqref{defi-proj-nubar} we have
$$
\frac{\P_\F\nu(F)}{\P_\F\nu(Q)}
=
\frac
{\frac{\sigma(F\cap Q_j)}{\sigma(Q_j)}\,\tom(P_j)}
{\frac{\sigma(Q\cap Q_j)}{\sigma(Q_j)}\,\tom(P_j)}
=
\frac{\sigma(F)}{\sigma(Q)}
=
\frac
{\frac{\sigma(F\cap Q_j)}{\sigma(Q_j)}\,\hm(Q_j)}
{\frac{\sigma(Q\cap Q_j)}{\sigma(Q_j)}\,\hm(Q_j)}
=\frac{\P_\F \hm(F)}{\P_\F \hm(Q)}.
$$

\noindent {\bf Case 2}: $Q$ is not contained in any  $Q_j\in \F$ (i.e.,
$Q\in \dd_{\F,Q_0}$). Notice that if $Q_j\in \F$ with $Q_j\cap Q\neq\emptyset$, then $Q_j$ is {\it strictly} contained in $Q$.    Let us note also that $\tom$ satisfies the doubling property,
by   Lemma \ref{lemma2.30}, Lemma \ref{lemma:inher-quali} and Corollary \ref{cor2.double}.
Set $E_0=Q_0\setminus (\cup_{\F} Q_j)$.
Using \eqref{defi-proj-nubar} we observe that
\begin{align}
\P_\F\nu(Q)
&=
\tom(Q\cap E_0) +\sum_{Q_j\in\F, Q_j\subsetneq Q} \frac{\sigma(Q\cap Q_j)}{\sigma(Q_j)}\,
\tom(P_j)
\nonumber
\\\label{proj-nu-bar:Q_1}
&=
\tom(Q\cap E_0) +\sum_{Q_j\in\F, Q_j\subsetneq Q} \tom(P_j)\\\nonumber
&\gtrsim \tom(Q\cap E_0) + \sum_{Q_j\in\F, Q_j\subsetneq Q}
\tom\left(B(x_j^\star,r_j)\cap \partial \Omega_{\F,Q_0}\right)
\\\nonumber&\geq \tom(\td^Q),
\end{align}
where in the third line we have used the doubling property of $\tom$ (plus a subdivision and Harnack Chain argument if $\ell(Q_j)\approx \ell(Q_0)$), and in the last line we have used
Proposition \ref{prop:surface-ball-sawtooth}, along with
Propositions \ref{prop:sawtooth-contain} and \ref{prop5.0a} and the doubling property to ignore
the difference between $Q\setminus (\cup_\F Q_j)$ and $Q\cap\partial\Omega_{\F,Q_0}$.

Let $A_{Q}$ be as in Proposition \ref{prop:cork-both}.   Then by Corollary \ref{cor2.poles}
plus the doubling property and Harnack Chain condition,
and a differentiation argument, we have that for any
Borel set $H\subset Q$,
\begin{equation}\label{w:AQ-X0}
\hm^{A_{Q}}(H) \approx \frac{\hm^{X_0}(H)}{\hm^{X_0}(Q)}=\frac{\hm(H)}{\hm(Q)}.
\end{equation}
The same occurs for $\tom$ and $\tom^{A_{Q}}$ and for any $H_\star\subset
\td(y_Q, \hat{r}_Q)$, (see \eqref{eqn:Pj-cork} and Proposition \ref{prop:cork-both}). More precisely,
\begin{equation}\label{nu:AQ-X0}
\tom^{A_{Q}}(H_\star) \approx \frac{\tom^{X_0}(H_\star)}{\tom^{X_0}\left(\td(y_Q,\widehat{r}_Q)\right)}
=\frac{\tom(H_\star)}{\tom\left(\td(y_Q,\widehat{r}_Q)\right)}\approx
\frac{\tom(H_\star)}{\tom\left(\td^Q\right)},
\end{equation}
where $\td^Q$ is the surface ball in Proposition \ref{prop:surface-ball-sawtooth}, and where the last step
follows by the
doubling property of $\tom$, since $\dist(\td^Q,\td(y_Q,\widehat{r}_Q)) \lesssim \ell(Q)$, and the radius of
each surface ball is comparable to $\ell(Q)$.

Using \eqref{proj-nu-bar:Q_1} and \eqref{nu:AQ-X0} (and \eqref{eqn:Pj-cork}), we obtain
\begin{align}\label{eq5.19}
\frac{\P_\F\nu(F)}{\P_\F\nu(Q)}
&
\lesssim
\frac{\tom(F\cap E_0)}{\tom(\td^Q)} +
\sum_{Q_j\in\F, Q_j\subsetneq Q} \frac{\sigma(F\cap Q_j)}{\sigma(Q_j)}\,
\frac{\tom(P_j)}{\tom(\td^Q)}
\\ \nonumber
&\approx
\tom^{A_{Q}}(F\cap E_0)
+
\sum_{Q_j\in\F, Q_j\subsetneq Q}
\frac{\sigma(F\cap Q_j)}{\sigma(Q_j)}
\tom^{A_{Q}}(P_j).
\end{align}
We claim that the following estimates hold:
\begin{equation}\label{claim:dyadic-DJK}
\tom^{A_{Q}}(F\cap E_0)
\lesssim
\hm^{A_{Q}}(F\cap E_0),
\qquad\quad
\tom^{A_{Q}}(P_j)
\lesssim
\hm^{A_{Q}}(Q_j).
\end{equation}
Indeed, the first estimate follows immediately from the maximum principle, since $\Omega_{\F,Q_0}\subset\Omega$, and $E_0\subset\partial\Omega\cap\partial\Omega_{\F,Q_0}$, by Proposition
\ref{prop:sawtooth-contain}.
To prove the second estimate, we observe that,
again by the maximum principle, it suffices to show that $\omega^X(Q_j)\gtrsim 1$,
for $X\in P_j$.  But the latter bound follows immediately from \eqref{eq2.Bourgain2} with
$\Delta =\Delta_{Q_j}$ (cf. \eqref{cube-ball}-\eqref{cube-ball2}),
the Harnack Chain condition and \eqref{eq5.1a}.

The bounds in \eqref{eq5.19}, \eqref{claim:dyadic-DJK} and \eqref{w:AQ-X0} imply
\begin{multline*}
\frac{\P_\F\nu(F)}{\P_\F\nu(Q)}
\lesssim
\hm^{A_{Q}}(F\cap E_0)
+
\sum_{Q_j\in\F, Q_j\subsetneq Q} \frac{\sigma(F\cap Q_j)}{\sigma(Q_j)}\,
\hm^{A_{Q}}(Q_j)\\
\approx
\frac{\hm(F\cap E_0)}{\hm(Q)}
+
\sum_{Q_j\in\F, Q_j\subsetneq Q}\frac{\sigma(F\cap Q_j)}{\sigma(Q_j)}\,
\frac{\hm(Q_j)}{\hm(Q)}
=
\frac{\P_\F \hm(F)}{\hm(Q)}
=
\frac{\P_\F \hm(F)}{\P_\F \hm(Q)},
\end{multline*}
where in the last equality we have used that $\P_\F \hm(Q)=\hm(Q)$.
Thus, we have established the righthand  inequality in \eqref{dyadic-DJK:proj}.
We may now obtain the left hand side
of \eqref{dyadic-DJK:proj} by a direct application of Lemma \ref{lemma:CF-dyadic} (see Appendix
\ref{appendixB} below), using the  fact that $\P_\F \hm$ and $\P_\F \nu$ are dyadically doubling by Lemmas  \ref{lemma:dyad-doubling-proj} and \ref{lemma:dyad-doubling-nu}.
\end{proof}

\section{A discrete Corona decomposition}\label{scorona}

In this section we present a discretized version of the stopping time decomposition of a Carleson region
appearing in \cite{CG}, \cite{AHLT}, \cite{AHMTT}, \cite{HM} and \cite{HM-ANU}
(cf.  \cite{Car2}, \cite{LM}, \cite{HL}).  We suppose that
$\{\alpha_Q\}_{Q\in\dd}$ is a sequence of non-negative numbers indexed on the dyadic ``cubes'',
and for any collection $\dd'\subset\dd$, we define
$$\mut(\dd'):= \sum_{Q\in\dd'}\alpha_{Q}.$$
For a fixed $Q_0\in\dd$, we say that $\mut$ is a ``Carleson measure'' on $\dd_{Q_0}$
(with respect to $\sigma$), and we write
$\mut\in \C(Q_0)$, if
\begin{equation*}
\|\mut\|_{\C(Q_0)}:= \sup_{Q\in\dd_{Q_0}} \frac{\mut(\dd_{Q})}{\sigma(Q)} <\infty.
\end{equation*}
We also write
\begin{equation}\label{eq6.0}
\|\mut\|_{\C}:= \sup_{Q\in\dd} \frac{\mut(\dd_{Q})}{\sigma(Q)} <\infty
\end{equation}
to denote the ``global'' Carleson norm on $\dd$.
We furthermore set $\dd_Q^{short}:= \dd_Q\setminus\{Q\}$, and given a family $\F\subset \dd$ of
pairwise disjoint cubes, we define the ``restriction of $\mut$ to the sawtooth $\dd_\F$'' by
\begin{equation*}
\mut_\F(\dd'):=\mut(\dd'\cap\dd_\F)= \sum_{Q\in\dd'\setminus (\cup_\F \dd_{Q_j})}\alpha_{Q}.
\end{equation*}

\begin{figure}
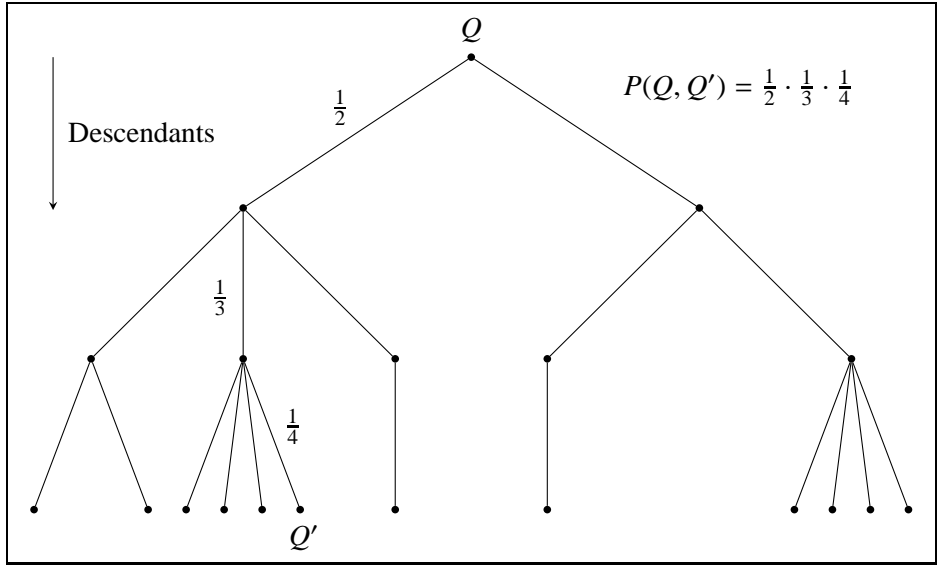

\fbox{\begin{pgfpicture}{-6cm}{-6.6cm}{6cm}{.6cm}
\pgfnodecircle{Node1}[fill]{\pgfxy(0,0)}{0.05cm}

\pgfnodecircle{Node11}[fill]{\pgfxy(-3,-2)}{0.05cm}
    \pgfnodecircle{Node111}[fill]{\pgfxy(-5,-4)}{0.05cm}
        \pgfnodecircle{Node1111}[fill]{\pgfxy(-5.75,-6)}{0.05cm}
        \pgfnodecircle{Node1112}[fill]{\pgfxy(-4.25,-6)}{0.05cm}
    \pgfnodecircle{Node112}[fill]{\pgfxy(-3,-4)}{0.05cm}
        \pgfnodecircle{Node1121}[fill]{\pgfxy(-3.75,-6)}{0.05cm}
        \pgfnodecircle{Node1122}[fill]{\pgfxy(-3.25,-6)}{0.05cm}
        \pgfnodecircle{Node1123}[fill]{\pgfxy(-2.75,-6)}{0.05cm}
        \pgfnodecircle{Node1124}[fill]{\pgfxy(-2.25,-6)}{0.05cm}
    \pgfnodecircle{Node113}[fill]{\pgfxy(-1,-4)}{0.05cm}
        \pgfnodecircle{Node1131}[fill]{\pgfxy(-1,-6)}{0.05cm}

\pgfnodecircle{Node12}[fill]{\pgfxy(3,-2)}{0.05cm}
    \pgfnodecircle{Node121}[fill]{\pgfxy(1,-4)}{0.05cm}
        \pgfnodecircle{Node1211}[fill]{\pgfxy(1,-6)}{0.05cm}
    \pgfnodecircle{Node122}[fill]{\pgfxy(5,-4)}{0.05cm}
        \pgfnodecircle{Node1221}[fill]{\pgfxy(4.25,-6)}{0.05cm}
        \pgfnodecircle{Node1222}[fill]{\pgfxy(4.75,-6)}{0.05cm}
        \pgfnodecircle{Node1223}[fill]{\pgfxy(5.25,-6)}{0.05cm}
        \pgfnodecircle{Node1224}[fill]{\pgfxy(5.75,-6)}{0.05cm}

\pgfnodeconnline{Node1}{Node11}
    \pgfnodeconnline{Node11}{Node111}
        \pgfnodeconnline{Node111}{Node1111}
        \pgfnodeconnline{Node111}{Node1112}
    \pgfnodeconnline{Node11}{Node112}
        \pgfnodeconnline{Node112}{Node1121}
        \pgfnodeconnline{Node112}{Node1122}
        \pgfnodeconnline{Node112}{Node1123}
        \pgfnodeconnline{Node112}{Node1124}
    \pgfnodeconnline{Node11}{Node113}
        \pgfnodeconnline{Node113}{Node1131}

\pgfnodeconnline{Node1}{Node12}
    \pgfnodeconnline{Node12}{Node121}
        \pgfnodeconnline{Node121}{Node1211}
    \pgfnodeconnline{Node12}{Node122}
        \pgfnodeconnline{Node122}{Node1221}
        \pgfnodeconnline{Node122}{Node1222}
        \pgfnodeconnline{Node122}{Node1223}
        \pgfnodeconnline{Node122}{Node1224}

\pgfnodelabel{Node1}{Node11}[0.5][-0.4cm]{\pgfbox[center,center]{$\frac12$}}
\pgfnodelabel{Node11}{Node112}[0.6][-0.3cm]{\pgfbox[center,center]{$\frac13$}}
\pgfnodelabel{Node112}{Node1124}[0.5][0.3cm]{\pgfbox[center,center]{$\frac14$}}

\pgfputat{\pgfxy(0,0.2)}{\pgfbox[center,bottom]{$Q$}}
\pgfputat{\pgfxy(-2.2,-6.2)}{\pgfbox[center,top]{$Q'$}}

\pgfputat{\pgfxy(2,-0.4)}{\pgfbox[left,center]{$P(Q,Q')=\frac12\cdot\frac13\cdot\frac14$}}

\pgfsetendarrow{\pgfarrowsingle}\pgfxyline(-5.5,0)(-5.5,-2)
\pgfputat{\pgfxy(-5.3,-1)}{\pgfbox[left,center]{Descendants}}
\end{pgfpicture}}

\caption{``Tree-graph'' with vertex $Q$ and its random walk}\label{figure:graph}

\end{figure}

We fix $Q_0\in\dd$, and construct a ``tree-graph'' with a vertex for each $Q\in\dd_{Q_0}$, and with
edges connecting a given $Q$ to each of its dyadic ``children'' (these are the subcubes of
$Q$ which lie in the very next dyadic generation $\dd_{k(Q)+1}$).   We consider a random walk
along the graph, in which it is permitted to move only to the descendant generation, but not to
the ancestral generation (nor to any other cube in the same generation),
and we suppose that from a given
$Q\in \dd_{Q_0}$, there is an equal probability of arriving at any of its children.
We set $P(Q,Q)=1$, and in general for $Q'\subseteq Q\in \dd_{Q_0}$,
we denote by $P(Q,Q')$ the probability that such a random walk beginning at $Q$ arrives at $Q'$ (thus also if $Q$ is strictly contained in $Q'$, or if $Q$ and $Q'$ are disjoint,
we have $P(Q,Q')=0$). See Figure \ref{figure:graph}.

\begin{lemma}\label{lemma:Corona}
Suppose that $\partial\Omega$ is ADR.  Fix $Q_0\in \dd$
and $\mut$ as above.  Let $a\geq 0$ and $b>0$, and suppose that
$\mut(\dd_{Q_0})\leq (a+b)\,\sigma(Q_0).$
Then there is a family $\F=\{Q_j\}\subset\dd_{Q_0}$
of pairwise disjoint cubes, and a constant $C$ depending only on dimension
and the ADR constants such that
\begin{equation} \label{Corona-sawtooth}
\|\mut_\F\|_{\C(Q_0)}
\leq C b,
\end{equation}
\begin{equation}
\label{Corona-bad-cubes}
\sigma(B)
\leq \frac{a+b}{a+2b}\, \sigma(Q_0)\,,
\end{equation}
where $B$ is the union of those $Q_j\in\F$ such that $\mut(\dd^{short}_{Q_j})>a\,\sigma(Q_j)$.
\end{lemma}

\begin{remark}\label{remark:general-doubling-sawtooth}
In the proof of this result, the only feature of $\sigma$ that we shall use,
is that it is a non-negative Borel measure satisfying the ``dyadically doubling property on $Q_0$'' (by this we mean that there is a uniform constant $c_\sigma$ such that $\sigma(\widetilde{Q})\leq c_\sigma\,\sigma(Q)$, whenever $\widetilde{Q}\in\dd_{Q_0}$
is the dyadic parent of $Q$). Notice that this property follows at once for our measure $\sigma= H^n\big|_{\partial\Omega}$ by the ADR property. Therefore, Lemma \ref{lemma:Corona} admits an extension in which $\sigma$ can be any non-negative dyadically doubling Borel measure on $Q_0$.
\end{remark}

\begin{proof} We note that $\mut(\dd_{Q_0}) = \mut(\dd^{short}_{Q_0}) + \alpha_{Q_0}$.  Thus,
if $\alpha_{Q_0} > b\sigma(Q_0)$, the result is trivial:  in this case
$\mut(\dd^{short}_{Q_0})\leq a\sigma(Q_0)$, so we may set $\F=\{Q_0\}$, and $B=\emptyset$.

Suppose now that $\alpha_{Q_0} \leq b\,\sigma(Q_0)$.  For $Q'\in \dd_{Q_0}$, we set
$$\beta(Q'):= \sum_{Q:\, Q'\subseteq Q\subseteq Q_0}P(Q,Q')\,\alpha_{Q}.$$
In particular, $\beta(Q_0)=\alpha_{Q_0}\leq b\,\sigma(Q_0).$
We now perform a standard stopping time argument
to select the collection $\F=\{Q_j\}$, comprised of the subcubes of $Q_0$
which are maximal with respect to the property that
\begin{equation}\label{eq6.4a}\beta(Q_j)> 2b\sigma(Q_j).
\end{equation}
If $\F$ is empty, we simply have that $\dd_\F=\dd$, $\mut_\F=\mut$
and $B=\emptyset$.

We now verify that $\F$ satisfies the desired properties.  We start by proving \eqref{Corona-bad-cubes}.
To this end let us record some useful facts.  We first note that, given a fixed $Q\in\dd_{Q_0}$,
\begin{equation}\label{eq6.4} \sum_{Q_j\in\F} P(Q,Q_j)=
\sum_{Q_j\in\F:\,Q_j\subseteq Q} P(Q,Q_j) \,\leq \,1\,,
\end{equation}
since the cubes in $\F$, and therefore also the events in the sum, are disjoint.
Next, we note that since $P(Q_j,Q_j) = 1$,
\begin{equation}\label{eq6.5}
\mut(\dd^{short}_{Q_j}) + \beta(Q_j) = \mut(\dd_{Q_j})
+\sum_{Q:\,Q_j\subsetneq Q\subseteq Q_0}
P(Q,Q_j)\,\alpha_Q\,,
\end{equation}
where the last sum runs over those $Q\in \dd_{Q_0}$ that {\it strictly} contain $Q_j$.
Consequently,
\begin{multline}\label{eq6.6}\sum_{Q_j\in\F}
\Big(\mut(\dd^{short}_{Q_j}) + \beta(Q_j)\Big)\\[4pt]
=\,\sum_{Q_j\in\F}\mut(\dd_{Q_j}) \,\,
+\sum_{Q\in \dd_{\F,Q_0}}\alpha_Q\sum_{Q_j\in\F:\,Q_j\subsetneq Q}
P(Q,Q_j)\,\,\leq\,\,\mut(\dd_{Q_0})\,,
\end{multline}
by \eqref{eq6.4} and the definition of $\dd_{\F,Q_0}$
(cf. \eqref{eq2.discretesawtooth2}).

We now set $\F_{bad}:=\{Q_j\in\F:\, \mut(\dd^{short}_{Q_j})>a\,\sigma(Q_j)\}$.  Then
by definition of $B$ and the stopping time construction,
\begin{multline*}(a+2b)\,\sigma(B)\,=\,(a+2b)\sum_{Q_j\in\F_{bad}}\sigma(Q_j)\\[4pt]
\leq\, \sum_{Q_j\in\F_{bad}}
\Big(\mut(\dd^{short}_{Q_j}) + \beta(Q_j)\Big)
\leq \,\mut(\dd_{Q_0})\,\leq\, (a+b)\,\sigma(Q_0)\,,\end{multline*}
where in the last line we have used \eqref{eq6.6} and our hypothesis.  Estimate \eqref{Corona-bad-cubes}  follows.

We now turn to the proof of $\eqref{Corona-sawtooth} $.  Let us fix $Q\in\dd_{Q_0}$.
We consider
$$\mut_\F(\dd_Q)= \mut(\dd_{\F,Q})=\sum_{Q'\in\dd_{\F,Q}}\alpha_{Q'}=\lim_{N\to\infty}
\sum_{Q'\in \dd_{\F_N,Q}}\alpha_{Q'}=\lim_{N\to\infty}\mut(\dd_{\F_N,Q})\,,$$
where $\F_N:=\F(2^{-N-1})$ is derived from $\F$ as in the discussion at the
beginning of Section \ref {spoincare};  i.e.,
$\F_N=\{Q_i^N\}$ is the collection of maximal cubes of
$$\F\cup\{Q'\in \dd_Q:
\ell(Q')\leq 2^{-N-1}\}\,.$$
Thus,
$$\dd_{\F_N,Q} =\left\{Q'\in \dd_{\F,Q}:\, \ell(Q')\geq 2^{-N}\right\}\,,\qquad N\geq k(Q).$$
It is therefore enough to establish the bound
\begin{equation}\label{eq6.7}
\mut(\dd_{\F_N,Q})\,\leq \,Cb\, \sigma(Q)
\end{equation}
uniformly in $N$.    To this end, we observe that equality holds in \eqref{eq6.4},
for a given cube $Q$ and pairwise disjoint family $\F$, if
$Q$ is covered by a union of cubes in $\F$.  Since this is the case for the family $\F_N$
and for every $Q' \in \dd_{\F_N,Q}$,  we have
\begin{multline*}
\mut(\dd_{\F_N,Q})
=
\sum_{Q'\in \dd_{\F_N,Q}}\alpha_{Q'}\sum_{Q_i^N\in\F_N:\, Q_i^N\subseteq Q'} P(Q',Q^N_i)
\\[4pt]
=
\sum_{Q_i^N\in\F_N \cap \dd_Q}\,\,\sum_{Q':Q_i^N\subset Q'\in \dd_{\F_N,Q}}
P(Q',Q^N_i)\,\alpha_{Q'} \,=\,\Sigma_1 \,+\,\Sigma_2\,,
\end{multline*}
where in $\Sigma_1$ the first sum runs over those $Q_i^N\in \F_N \cap \dd_Q$ which are equal to some $Q_j\in\F$ (i.e., $Q_i^N\in \F_N\cap\F\cap \dd_Q$),
while in $\Sigma_2$ the first sum  runs over the remaining cubes in $\F_N \cap\dd_Q$  (i.e., over $Q_i^N\in (\F_N\setminus\F)\cap \dd_Q$, equivalently those $Q_i^N$ which are not contained
in any $Q_j\in\F$). We then have
\begin{multline*}
\Sigma_2\,=\,\sum_{Q_i^N\in (\F_N\setminus\F)\cap \dd_Q} \Big(\sum_{Q':Q_i^N\subset Q'\in \dd_{\F_N,Q}}
P(Q',Q^N_i)\,\alpha_{Q'}\Big)
\\[4pt]
\leq\,
\sum_{Q_i^N\in (\F_N\setminus\F)\cap \dd_Q} \beta(Q_i^N)\, \leq \,2b\sum_{Q_i^N\in (\F_N\setminus\F)\cap \dd_Q}\sigma(Q^N_i)
\,\leq \,2b\, \sigma(Q),\end{multline*}
by the stopping time construction of $\F$, since $Q_i^N$ is not contained in any
$Q_j\in \F$, and $Q_i^N\in\dd_Q$.

We now consider $\Sigma_1$.   We first note that
no $Q'$ appearing in the sum
can be contained in any $Q_j\in\F$, since
$\dd_{\F_N,Q}\subset \dd_{\F,Q}$.  Therefore, if some $Q_j\in\F$ is contained in
any such $Q'$, then so is its dyadic parent $\widetilde{Q}_j$.  Moreover,
\begin{equation*}
P(Q',Q_j)\leq P(Q',\widetilde{Q}_j)\,,\qquad \forall Q_j\in\F\,. \end{equation*}
We then have that, by definition,
\begin{multline*}
\Sigma_1 \,
\leq
\,\sum_{Q_j\in\F} \,\sum_{Q':\widetilde{Q}_j\subseteq Q'\subseteq Q} P(Q',\widetilde{Q}_j)\,\alpha_{Q'}
\leq \sum_{Q_j\in\F:\,Q_j\subset Q}\beta(\widetilde{Q}_j)\\[4pt]
\leq 2b\!\sum_{Q_j\in\F:\, Q_j\subset Q}\sigma(\widetilde{Q}_j) \leq Cb\,\sigma (Q),\end{multline*}
where the next-to-last inequality holds because the cubes $Q_j$ are maximal with respect to
the property \eqref{eq6.4a}, and the last one holds by the dyadic doubling property of $\sigma$
(see Remark \ref{remark:general-doubling-sawtooth}),
and the pairwise disjointness of the cubes in $\F$.
\end{proof}

\section{Proofs of Theorems \ref{theor-main-I} and \ref{theor-main-I:BP}}
\subsection{Relating geometric and discrete Carleson measures}

We recall that the UR property may be characterized in terms of the Carleson measure estimate
\eqref{eq1.sf}, which we shall invoke with $E=\partial\Omega$.  We also remind the reader
that we may assume that for every Whitney cube $I\in\W$, we have
$\dist(4I,\partial\Omega)\approx\ell(I)$ (cf. \eqref{eq4.1}).   In this case, the ``fattened'' Whitney cubes
$4I$ have bounded overlaps.  From this fact, properties of Whitney cubes,
and the ADR property, it follows that the fattened Whitney
regions $U_Q^{fat}= \cup_{\W^*_Q} 4I$ (cf. \eqref{eq2.whitney1}-\eqref{eq2.whitney2}
and \eqref{eq2.whitney3fat}) also have the bounded overlap property:
\begin{equation}\label{eq7.1}\sum_{Q\in\dd} 1_{U_Q^{fat}}(X) \leq C\,.\end{equation}
We now set
\begin{equation}\label{eq7.2a}
\alpha_Q := \dint_{U_Q^{fat}}|\nabla^2 \mathcal{S}1(X)|^2 \,\delta(X) \,dX \,,\end{equation}
and for any sub-collection $\dd'\subset\dd$, we define
\begin{equation}\label{eq7.3a}\mut(\dd'):= \sum_{Q\in\dd'}\alpha_Q\,,
\end{equation}
as in the Section \ref{scorona}.  By \eqref{eq7.1}, for every
pairwise disjoint family $\F\subset\dd$, and every $Q\in\dd_\F$, we have
\begin{equation}\label{eq7.2}
\mut_\F(\dd_Q)=
\mut(\dd_{\F,Q}) \approx \dint_{\Omega_{\F,Q}^{fat}}|\nabla^2 \mathcal{S}1(X)|^2 \,\delta(X) \,dX
\end{equation}
where $\Omega_{\F,Q}^{fat}$ is defined in \eqref{eq2.sawtooth2fat} (we have used in \eqref{eq7.2}
the rather trivial fact that $\partial\Omega_{\F,Q}^{fat}$ has $(n+1)$-dimensional Lebesgue measure 0).
In particular, taking $\F=\emptyset$, in which case $\dd_{\F,Q}=\dd_Q$,
and $\Omega_{\F,Q}^{fat}=T_{Q}^{fat}$ (cf.
\eqref{eq2.boxfat}), we obtain from \eqref{eq4.13} that $\mut$ inherits the Carleson measure property
\eqref{eq6.0} from \eqref{eq1.sf}, and that the Carleson norm $\|\mut\|_\C$ depends only on dimension
and the various ADR, UR, Corkscrew and Harnack Chain constants for $\Omega$ (including $K_0$).

\subsection{Proof of Theorem \ref{theor-main-I}  with ``qualitative assumptions''}\label{ssqual}
In this subsection, we present the proof of Theorem \ref{theor-main-I}, in the special case that
the  qualitative exterior Corkscrew condition holds in $\Omega$
(and therefore also  in its sawtooths and Carleson boxes).
As we observed in Section \ref{section-fund-est}, this qualitative hypothesis
(along with our standard quantitative assumptions),
were enough to imply the doubling condition for the harmonic measure for $\Omega$ and the sawtooth regions, and
also to allow us to obtain the ``Dyadic Sawtooth'' Lemma \ref{lemma:DJK-dyadic-proj}.
We shall remove the qualitative assumptions, and also give the proof of Theorem
\ref{theor-main-I:BP}, in subsection \ref{sremove}.

We shall use the method of ``extrapolation of Carleson measures'', based on ideas originating
in \cite{CG} and \cite{LM} (cf. \cite{HL}, \cite{AHLT}, \cite{AHMTT}, \cite{HM}).  In more precise detail, we follow our related work in the Euclidean setting \cite{HM}.
In the sequel, we say that a measure $\mu$ is ``dyadically doubling on $Q_0$'' if
there is a uniform constant $c_\mu$ such that $\mu(\widetilde{Q})\leq c_\mu\,\mu(Q)$, whenever $\widetilde{Q}\in\dd_{Q_0}$
is the dyadic parent of $Q$.

\begin{lemma}\label{lemma:extrapol}
We fix $Q_0\in \dd$.
Let $\sigma$ and $\omega$ be a pair of
non-negative,
dyadically doubling Borel measures on $Q_0$,
and let $\mut$ be a discrete Carleson measure with respect to $\sigma$
(cf. Section \ref{scorona}) with
$$\|\mut\|_{\C(Q_0)}\le M_0.$$
Suppose that there is a $\gamma>0$ such that for every $Q\in \dd_{Q_0}$ and every family of pairwise disjoint dyadic subcubes  $\F=\{Q_j\}\subset \dd_{Q}$
verifying
\begin{equation}\label{extrap:Carleson-delta}
\|\mut_{\mathcal{F}}\|_{\mathcal{C}(Q)}
\le\gamma\, ,
\end{equation}
we have that $\P_\F \,\omega$ satisfies the following property:
\begin{equation}\label{extrap:Ainfty:Pw}
\forall\,\varepsilon\in (0,1),\ \exists\, C_\varepsilon>1 \mbox{ such that }
\Big(
F\subset Q,\ \ \frac{\sigma(F)}{\sigma(Q)}\ge
\varepsilon\quad \Longrightarrow \quad
\frac{\P_\F \,\omega(F)}{\P_\F\, \omega(Q)}\ge \frac1{C_\varepsilon}\Big).
\end{equation}
Then, there exist $\eta_0\in(0,1)$  and $C_0<\infty$
such that, for every  $Q\in \dd_{Q_0}$,
\begin{equation}\label{extrap:Ainfty:w}
F\subset Q,\quad \frac{\sigma(F)}{\sigma(Q)}\ge
1-\eta_0\quad \Longrightarrow \quad \frac{\omega(F)}{\omega(Q)}\ge \frac1{C_0}.
\end{equation}
I.e., $\omega \in A^{\rm dyadic}_\infty(Q_0)$.
\end{lemma}

\begin{remark}
Notice that in the statement of the lemma, $\sigma$ and $\omega$
are allowed to be any pair of non-negative, dyadically doubling Borel measures on $Q_0$,  and
that $\sigma$ plays the role of underlying measure. Therefore,  $\sigma$ appears
implicitly in the Carleson conditions, $\P_\F$ and in the definition of the class
$A^{\rm dyadic}_\infty(Q_0)$.
In the present paper,
we shall apply this result in the special case that $\sigma= H^n\big|_{\partial\Omega}$, and
$\omega=\hm^{X_{Q_0}}$, the harmonic measure with pole at the Corkscrew point $X_{Q_0}$.
\end{remark}

\begin{remark} It is known that \eqref{extrap:Ainfty:w}, for every $Q\in\dd_{Q_0}$,
self-improves to \eqref{eq1.ainftydyadic}, but this fact may also be gleaned
from Remark \ref{remark:Ainfty} below. \end{remark}

\begin{remark} The key hypothesis of the lemma,
and the main point that must be verified in applications,
is that (\ref{extrap:Carleson-delta}) implies (\ref{extrap:Ainfty:Pw}), for sufficiently small
$\gamma$.
\end{remark}

In the remainder of this subsection, we shall use Lemma \ref{lemma:extrapol} to prove Theorem
\ref{theor-main-I}, assuming the extra qualitative exterior corkscrew condition.
The qualitative hypothesis  will then be removed in the next subsection.
We defer the proof of the lemma to Section \ref{s8}.

\begin{proof}[Proof of Theorem \ref{theor-main-I} with qualitative  hypothesis]
To begin, we let $\sigma= H^n\big|_{\partial\Omega}$, which is dyadically doubling by the ADR property. Let us fix $Q_0\in\dd$, and set $\hm := \hm^{X_{Q_0}}$, where as usual
$X_{Q_0}$ is a Corkscrew point relative to $Q_0$.  Given the qualitative hypothesis,
it holds in particular that $\hm$ is a doubling measure
(cf. Corollary \ref{cor2.double}), and therefore also dyadically doubling, on
$Q_0$ (to obtain dyadic doubling when $\ell(Q)\approx \ell(Q_0)$,  we may need to invoke
the Harnack Chain condition);
moreover, the doubling constants depend only
upon the constants in the {\it quantitative} hypotheses of Theorem \ref{theor-main-I} (i.e., dimension,
UR, ADR, Harnack Chain and Corkscrew, including the constant $K_0$ which ultimately depended
only upon the other stated parameters).  We define
$\mut$ as in \eqref{eq7.3a}, with $\alpha_Q$ as in \eqref{eq7.2a}.  As observed in the previous subsection, this $\mut$ inherits the discrete Carleson measure property
\eqref{eq6.0} from \eqref{eq1.sf}.    Therefore, once we have verified that
(\ref{extrap:Carleson-delta}), with $\gamma$ small enough, implies (\ref{extrap:Ainfty:Pw}),
we may then conclude from Lemma \ref{lemma:extrapol} that $\omega =\hm^{X_{Q_0}}
\in A_\infty^{\rm dyadic}(Q_0)$, for every $Q_0\in\dd$, and thus by the Harnack Chain condition
that $\hm^{X_{Q_0}} \in A_\infty^{\rm dyadic}(Q_1)$, for every $Q_1$ of the same generation as
$Q_0$ such that $\dist(Q_0,Q_1) \leq 100\diam(Q_0)$.    Since this is true for every
$Q_0\in\dd$, and since $\hm^{X_{Q_0}}$ is {\it concentrically}
doubling, we may conclude that $\hm^{X_\Delta} \in A_\infty(\Delta),$ for every surface ball
$\Delta =\Delta(x,r),\, x\in \partial\Omega$ and $r\leq \diam (\partial\Omega)$, with
$A_\infty$ constants uniformly controlled, and depending only upon
dimension and the UR, ADR, Harnack Chain and Corkscrew constants.
We reached this conclusion by imposing the extra qualitative exterior corkscrew condition,
but as our estimates do not depend quantitatively on that hypothesis,
we shall be able, in
subsection \ref{sremove}, to remove it
by an approximation argument, but at the
the loss of the doubling property of $\hm$.

To complete our task in this subsection, it now remains only to verify that (\ref{extrap:Carleson-delta}),
with $\gamma$ small enough, implies (\ref{extrap:Ainfty:Pw}).
To this end, we fix a $Q\in \dd_{Q_0}$ and a pairwise disjoint family $\F\subset \dd_Q$, and
we suppose that \eqref{extrap:Carleson-delta} holds for some small $\gamma$ to be chosen
momentarily.  By \eqref{eq7.2}, we deduce that
\begin{equation}\label{eq7.11}
\sup_{Q'\in\dd_Q}\frac1{\sigma(Q')}
\dint_{\Omega_{\F,Q'}^{fat}}|\nabla^2 \mathcal{S}1(X)|^2 \,\delta(X) \,dX\, \leq\, C\gamma.
\end{equation}
Consequently, if $\gamma$ is small enough, depending only upon the allowable quantitative parameters, we may apply Corollary \ref{cor4.18}, with $Q$ in place of $Q_0$, to
obtain, for every  surface ball $\td =B\cap\partial \Omega_{\F,Q}$, with
$B=B(x,r),\, x\in \partial\Omega_{\F,Q}$, and $r\leq \diam(Q)$, that
$\tom^{X_{\td}} \in A_\infty(\td)$, where $\tom^X$ denotes harmonic measure for
$\Omega_{\F,Q}$. Moreover, the $A_\infty$ constants are uniformly controlled by the
stated parameters.  By the Harnack Chain condition,
we obtain that $\tom^{A_Q} \in A_\infty(\partial\Omega_{\F,Q})$
(meaning that we view $\partial\Omega_{\F,Q}$ itself as a surface ball $\td^Q$ of radius
$r(\td^Q)\approx K_0\,\ell(Q)$, and that $\tom^{A_Q} \in A_\infty(\td^Q)$)
where $A_Q$ is the simultaneous Corkscrew point
produced in Corollary \ref{cor5.5}, applied with $Q$ in place of $Q_0$.

Let $P_\F \nu$ be defined as
in Lemma \ref{lemma:DJK-dyadic-proj}, but  again with $Q$ in place of $Q_0$.
We shall prove in Lemma \ref{lemmaB.4} (Appendix \ref{appendixB} below) that
$P_\F \nu\in A_\infty^{\rm dyadic} (Q)$.  Thus, by Lemma \ref{lemma:DJK-dyadic-proj},
with $Q$ in place of $Q_0$, we obtain that $P_\F\hm^{A_Q} \in A_\infty^{\rm dyadic} (Q)$.
We may then use Corollary \ref{cor2.poles} (here we only consider the case that $Q$ is not contained in any $Q_j$, otherwise $\P_\F\hm^{X_{Q_0}} \in A_\infty^{\rm dyadic} (Q)$ trivially)
along with the Harnack Chain condition
and a differentiation argument, to replace the pole $A_Q$ by $X_{Q_0}$, the Corkscrew point for
the ambient cube $Q_0$, and to conclude that
$\P_\F\hm =\P_\F\hm^{X_{Q_0}} \in A_\infty^{\rm dyadic} (Q)$.   In particular,
(\ref{extrap:Ainfty:Pw}) holds, by Lemmas \ref{lemma:CF-dyadic} and \ref{lemma:dyad-doubling-proj} in Appendix \ref{appendixB}.
\end{proof}

\subsection{Removing the qualitative  hypothesis,  and conclusion of the proofs
of Theorems \ref{theor-main-I} and \ref{theor-main-I:BP}}\label{sremove}

In this subsection, we first complete the proof of Theorem \ref{theor-main-I}
(modulo the proof of Lemma \ref{lemma:extrapol}
and the technical lemmata that we have deferred
to Appendices), by
removing the  qualitative exterior corkscrew condition.
We then conclude by
giving the proof of Theorem \ref{theor-main-I:BP}.

We define approximating domains as follows.
For each large integer $N$,  set $\F_N := \dd_N$.  We then let
$\Omega_N := \Omega_{\F_N}$ denote the usual (global) sawtooth with respect to the family
$\F_N$ (cf. \eqref{eq2.whitney2}, \eqref{eq2.whitney3} and \eqref{eq2.sawtooth1}.)  Thus,
\begin{equation}\label{eq7.on}
\Omega_N =\interior\left(\bigcup_{Q\in \dd:\,\ell(Q)\geq 2^{-N+1}}U_Q\right),
\end{equation}
so that $\overline{\Omega_N}$ is the union of fattened Whitney cubes $I^*=(1+\lambda)I$, with $\ell(I)\gtrsim 2^{-N}$, and the boundary of $\Omega_N$ consists of portions of faces of $I^*$ with $\ell(I)\approx 2^{-N}$.
By virtue of Lemma \ref{lemma2.30}, each $\Omega_N$ satisfies the ADR,
Corkscrew and Harnack Chain properties.   Moreover, $\partial\Omega_N$ is UR.
We defer the proof of the UR property to Appendix \ref{appendix:approx}.
We note that, for each of these properties, the constants are uniform in $N$, and depend only on dimension and on the corresponding constants for $\Omega$.
In addition, by construction, $\Omega_N$ has exterior corkscrew points at all scales $\lesssim 2^{-N}$. By Lemma \ref{lemma:inher-quali}, the same statement applies to the Carleson boxes $T_Q$ and $T_\Delta$, and to the sawtooth domains $\Omega_\F$ and $\Omega_{\F,Q}$  (all of them relative to $\Omega_N$)  and even to Carleson boxes within sawtooths.

Consequently, by the arguments in the previous subsection, we conclude that for every surface ball $\td=\td^N\subset \partial\Omega_N$, the harmonic measure $\omega_N^{X_{\td}} \in A_\infty(\td)$, uniformly in $N$.  We now consider the limiting case.
Fix a surface ball $\Delta:=\Delta(x,r)\subset\partial\Omega$, and a Borel subset $A\subset\Delta$.
Assuming the hypotheses of Theorem \ref{theor-main-I}, we claim that
\begin{equation}\label{eq7.12}
\sigma(A)\geq \eta \,\sigma(\Delta)\,\,\Longrightarrow\,\,
\hm^{X_\Delta}(A)\geq c_0\,\eta^{\theta}\,,\qquad \forall \eta \in (0,1)\,,
\end{equation}
for some uniform positive constants $c_0$ and $\theta$, where as usual
$X_\Delta$ denotes a Cork\-screw point relative to $\Delta$.  By the outer regularity property of $\hm^X$
(cf. \eqref{eq2.outer}), we may assume that $A$ is (relatively) open.  It then follows that
we may write
$A=\cup_k Q_k$, where $\{Q_k\}\subset \dd$ is a pairwise disjoint collection.
We set $\Delta_k =\Delta(x_k,r_k):=\Delta_{Q_k}$,
where $\Delta_{Q_k}:=B_{Q_k}\cap\pom\subset Q_k$
is the surface ball defined in \eqref{cube-ball}-\eqref{cube-ball2},
so that $r_k\approx \ell(Q_k)$ and $\sigma(Q_k)\approx \sigma(\Delta_k)$.  Then
$$
\eta\,\sigma(\Delta)\leq \sigma(A)=  \sum_k\sigma(Q_k)\approx\sum_k\sigma(\Delta_k).
$$
We now set $A':=\cup_{k=1}^{M} \Delta_k$, where
$M$ is chosen large enough (depending on $A$) so that
\begin{equation}\label{eq7.13}
\sigma(A')=\sum_{k=1}^M\sigma(\Delta_k)\geq \frac1C \,\eta\, \sigma(\Delta).
\end{equation}
By the ADR property and a covering lemma argument, we may further suppose that
the Euclidean balls $B_k:=B_{Q_k}, 1\leq k\leq M,$ are pairwise disjoint.
We now fix $N$ so large that $2^{-N}\ll \min_{1\leq k\leq M} r_k$, and $2^N\gg\diam(\Delta)$.
Fix also a point $\hat{x}\in\pom_N$, with $|x-\hat{x}|\approx 2^{-N}$ (such a point exists,
with implicit constants possibly depending on $K_0$, since $\Omega$
satisfies the Corkscrew condition).
We shall approximate $\Omega$ by a domain $\widehat{\Omega}_N$, which is
defined as follows,
and whose harmonic measure we denote by
$\hat{\hm}^X$.
If $\Omega$ is  bounded, we set $\widehat{\Omega}_N=\Omega_N$.
Otherwise, we define $\widehat{\Omega}_N:=T_{\td^N}$,
where $T_{\td^N}\subset \Omega_N$ denotes the
Carleson box corresponding to $\td^N=B(\hat{x}, 2^N)\cap \pom_N$
for the domain $\Omega_N$.  Then by the arguments in Subsection \ref{ssqual},
for every surface ball $\td=\Delta_{\star,N}\subset \partial\widehat{\Omega}_N$,
the harmonic measure $\hat{\omega}_N^{X_{\td}} \in A_\infty(\td)$, uniformly in $N$,
since $\widehat{\Omega}_N$ is either equal to $\Omega_N$, or else is a subdomain
of $\Omega_N$ which inherits all of the requisite properties as observed above.

We set $B_k':= c B_k$, where $c\in(0,1)$ is the constant in Lemma \ref{Bourgainhm}.
As noted above,
the collection $\{B_k\}_{1\leq k\leq M}$, hence also $\{B_k'\}_{1\leq k\leq M}$, may be taken to be pairwise disjoint.  Let us also note that, since $2^{-N}\ll\min_{1\leq k\leq M} r_k$,
by the ADR properties of $\partial\widehat{\Omega}_N$ and $\pom$, we have
\begin{equation}\label{eq7.14}
H^n\left(\cup_{k=1}^M B_k'\cap\partial\widehat{\Omega}_N\right)\,\gtrsim\, \sum_{k=1}^M r_k^n
\,\gtrsim \,\eta\,\sigma(\Delta)\gtrsim\eta \,H^n\big(\Delta_\star(\hat{x},Cr)\big)\,,
\end{equation}
where in the last pair of inequalities we have used \eqref{eq7.13} and ADR (for both $\pom$ and $\pom_N$).  Moreover, since $2^{-N}\ll \diam(\Delta)\ll 2^N$, we have that
$X_\Delta \in \widehat{\Omega}_N$ is also a Corkscrew point for $\widehat{\Omega}_N$
with respect to the surface ball
$\td(\hat{x},Cr)$, where  $\hat{x}$ is as above,
and where $C$ is chosen large enough that
$\cup_{k=1}^M B_k'\cap\partial\widehat{\Omega}_N \subset \td(\hat{x},Cr)$.

We observe that
$u(X):= \omega^X(A')$ is  harmonic in $\Omega$, and thus also in the bounded subdomain
$\widehat{\Omega}_N$.  Since in bounded domains we have uniqueness by the maximum principle, we obtain
\begin{multline}\label{eq8.17***}
\hm^{X_\Delta}(A')\,=\,\int_{\partial\widehat{\Omega}_N} \hm^Y(A')\, d\hat{\hm}^{X_\Delta}_N(Y)
\,\geq\,  \sum_{k=1}^M
\int_{B_k'\cap\partial\widehat{\Omega}_N} \hm^Y(\Delta_k)\,
d\hat{\hm}^{X_\Delta}_N(Y)\\[4pt]
\gtrsim\, \sum_{k=1}^M
\int_{B_k'\cap\partial\widehat{\Omega}_N}
d\hat{\hm}^{X_\Delta}_N(Y) \,= \,\hat{\hm}^{X_\Delta}_N(\cup_{k=1}^M
B_k'\cap\partial\widehat{\Omega}_N)\, \gtrsim\, \eta^\theta,
\end{multline}
where in the last line we have used Lemma \ref{Bourgainhm} and then
\eqref{eq7.14} and the
$A_\infty$ property of $\hat{\hm}^{X_\Delta}$
(recall that $X_\Delta$ serves as a Corkscrew point for $\td(\hat{x},Cr)$, as we have noted above). Since $A'\subset A$, we then obtain \eqref{eq7.12}.

We now note that \eqref{eq7.12} trivially implies
the following weak version of itself:
for bounded $\Omega$ (the unbounded case is treated below) satisfying the hypotheses of Theorem \ref{theor-main-I},
there exist uniform constants $\eta \in (0,1)$ and $c_0 >0$ such that
\begin{equation}\label{eq7.17}
\sigma(A)\geq \eta \,\sigma(\Delta)\,\,\Longrightarrow\,\,
\hm^{X_\Delta}(A)\geq c_0\,.
\end{equation}
We remark here that to establish \eqref{eq7.17} has really been our main goal.
Indeed, given \eqref{eq7.17}, the remainder of the proof of Theorem \ref{theor-main-I}
will follow
the arguments in \cite{BL}.  We further remark that in \cite{BL}, \eqref{eq7.17}
is essentially taken as a starting point:   by the maximum principle, and the result of \cite{Dah},
an appropriate version of \eqref{eq7.17} (cf. \eqref{eq7.19} below) follows immediately
from the main hypothesis in \cite{BL}, that $\Omega$ has ``interior big pieces''
(in the sense of Definition \ref{def1.bp}) of
Lipschitz sub-domains of $\Omega$, with uniform constants.
Eventually, we shall see that \eqref{eq7.17}, suitably interpreted,
continues to hold under the hypotheses
of Theorem \ref{theor-main-I:BP}.

We now proceed to describe the remaining steps needed to deduce the weak-$A_\infty$ property
of harmonic measure.
By \cite[Lemma 3.1]{BL}, it suffices to show that for each $\eps\in(0,1/1000)$,
there are uniform constants
$\eta_\eps\in(0,1)$ and $C_\eps\in(1,\infty)$,
such that given balls $B,B'$, centered on $\partial\Omega$,
with $2B'\subset B$, and corresponding
surface balls $\Delta := B\cap\pom$ and $\Delta':=B'\cap\pom$,
and a Borel subset $A\subset 2\Delta'$ with
$\sigma(A)\geq \eta_\eps\,\sigma(2\Delta')$, we have
\begin{equation}\label{eq7.15}
\hm^{X_\Delta}(\Delta')\leq \eps\,\hm^{X_\Delta}(2\Delta')+C_\eps \,\hm^{X_\Delta}(A).
\end{equation}
In fact, \cite[Lemma 3.1]{BL} is a purely real variable result which says that any positive Borel measure
$\mu$ on $\partial\Omega$ satisfying \eqref{eq7.15} belongs to weak-$A_\infty(\Delta)$ (equivalently, satisfies the weak
reverse H\"older estimate \eqref{eq1.WRH} for some $q>1$), assuming only that $\partial\Omega$ is ADR.  Under the hypotheses of Theorem \ref{theor-main-I}, we shall establish \eqref{eq7.15} with $\eta_\eps:= \eta$, the constant in \eqref{eq7.17} (independently of $\eps$).

Let us now give the proof of \eqref{eq7.15}.
We prove the desired bound first in the case that $\Omega$ is bounded.
This restriction will be removed at the end of the proof.
We follow the argument in \cite[Lemma 2.2]{BL}
almost verbatim, with some small simplifications permitted by
our hypothesis that the Harnack Chain condition holds in Theorem \ref{theor-main-I}.
Let $B'=B(z,s)$, $\Delta':=B'\cap\pom$,
and suppose $2B'\subseteq B:=B(x,r)$.
We cover $\frac32 \Delta'\setminus \frac54 \Delta'$
by annuli of thickness $\approx \eps s$.  More precisely,
we set
\begin{eqnarray}\label{eq7.16*}&U_k:= \left\{ y\in\pom:\,\left(5/4 +\eps k\right)s\leq|y-z|
< \big(5/4 +\eps (k+1)\big)s\right\},
\\[4pt]\nonumber
&S_k:=\left\{ X\in\Omega:\,|X-z|=\left(5/4 +\eps \big(k+1/2\big)\right)s\right\},
\end{eqnarray}
where $0\leq k\lesssim1/(4\eps)$.  Suppose now $A\subset 2\Delta'$, with
$\sigma(A)\geq \eta\,\sigma(2\Delta')$, for $\eta$ as in \eqref{eq7.17}.
Let $c\in (0,1)$ be the constant
in Lemma \ref{Bourgainhm}.
By the Harnack Chain condition and \eqref{eq7.17}, applied to $2\Delta'$ in place of $\Delta$, we have
\begin{equation}\label{eq7.17hm}
\hm^X(A)\geq c_\eps\, c_0\,,\qquad \forall X \in S_k\cap \{X:\,\delta(X)\geq c\eps s/100\},
\end{equation}
uniformly in $k$.
On the other hand, if $X\in S_k\cap \{X:\,\delta(X)<c\eps s/100\}$, then for a suitable uniform
constant $C$, we have
\begin{equation}\label{eq7.18+}
C\hm^X(U_k)\geq C\omega^X\big(\Delta(\hat{x},\eps s/10)\big) \geq 1\,,
\end{equation}
by Lemma \ref{Bourgainhm}, where $\hat{x}\in\pom$ is chosen so that $|X-\hat{x}|=\delta(X)$.
Thus,
\begin{equation}\label{eq7.16}
\hm^X(\Delta')\leq 1\leq C\hm^X(U_k) +C_\eps\,\hm^X(A)\,,\qquad \forall X\in S_k,
\end{equation}
where $C_\eps= 1/(c_\eps c_0)$.  By the maximum principle, this implies in particular that
\eqref{eq7.16} continues to hold for $X\in \Omega\setminus \frac74 B'$, since $S_k \subset \frac74 B'$,
if $\eps$ is small, for every relevant $k$ (i.e., those for which
$U_k$ meets $\frac32 B'\setminus\frac54B'$).  Since this set of $k$'s has cardinality
$\approx1/\eps$, summing in $k$ we obtain
we obtain
\begin{equation}\label{eq7.20*}
\frac1{\eps} \hm^X(\Delta') \leq C \hm^X(2\Delta') + C_\eps\hm^X(A)\,,\qquad \forall X\in
\Omega\setminus \frac74 B'\,,\end{equation}
since the $U_k$'s are pairwise disjoint and contained in $2\Delta'$.  The desired bound
\eqref{eq7.15} now follows, at least in the case that $\Omega$ is bounded.

Now suppose that $\Omega$ is unbounded.   Given a surface ball $\Delta=\Delta(x,r)$,
we choose $R\gg r$, set $\Delta_R = \Delta(x,R)$, and consider the
domain $\Omega_R:=T_{\Delta_R}$, the Carleson Box associated to $\Delta_R$.
For each such $R$, the argument above may be applied, to obtain
\eqref{eq7.15} for each of the corresponding harmonic measures
$\hm_R^{X_\Delta}$.  For any fixed Borel subset $F\subset \Delta$,
we have that the solutions $u_R(X):= \hm_R^X(F)$ are monotone increasing on any fixed
$\Omega_{R_0}$,
as $R_0\leq R\to\infty$, by the maximum principle.  We then obtain that
$u_R(X)\to u(X):=\hm^X(F)$,
uniformly on compacta, by Harnack's convergence theorem
(as in the discussion at the beginning of Section \ref{section-fund-est}), whence \eqref{eq7.15} follows.
The proof of Theorem \ref{theor-main-I} is now complete, modulo the deferred arguments.

\begin{proof}[Proof of Theorem \ref{theor-main-I:BP}]
Finally, we discuss the modifications needed to prove Theorem \ref{theor-main-I:BP}.
By \cite[Lemma 3.1]{BL} (and a
limiting process to treat the case of an unbounded domain),
it again suffices to establish, for bounded $\Omega$
now satisfying the hypotheses of Theorem \ref{theor-main-I:BP},
an appropriate version of \eqref{eq7.15}.  That is, we seek to show that
for each $\eps\in(0,1/1000)$,
there are uniform constants
$\eta_\eps\in(0,1)$ and $C_\eps\in(1,\infty)$,
such that given balls $B=B(x,r)$ and $B'=B(z,s)$,
with $2B'\subset B$, and corresponding
surface balls $\Delta := B\cap\pom$ and $\Delta':=B'\cap\pom$,
if $A\subset 2\Delta'$ with
$\sigma(A)\geq \eta_\eps\,\sigma(2\Delta')$,  then
\begin{equation}\label{eq7.15*}
\hm^{X}(\Delta')\leq \eps\,\hm^{X}(2\Delta')+C_\eps \,\hm^{X}(A)\,, \qquad \forall X \in \Omega\setminus
B.
\end{equation}

To this end, we first establish a suitable variant of \eqref{eq7.17}.
Given $X\in \Omega$, under the hypotheses of Theorem
\ref{theor-main-I:BP}, there is a point $x\in\pom$, with
$|X-x|=\delta(X)$, and  a subdomain $\Omega'\subset\Omega$ satisfying
the hypotheses of Theorem \ref{theor-main-I},  with the property that
for some constants $K>1$ and $\alpha>0$, we have
$$\sigma\left(\pom' \cap\Delta_X\right)\,\geq\, \alpha\, \sigma(\Delta_X)\,,$$
where
$\Delta_X:= \Delta(x,K\delta(X))$.
We may further suppose that $X$ serves as a Corkscrew point for
$\Omega'$ relative to a surface ball
$\Delta_\star:=B(y,2K\delta(X))\cap\pom'$, with $y\in \Delta_X\cap\pom'$.  That $\Omega'$ exists,
with uniform control
of the various constants involved, is simply a re-statement of
the ``big pieces'' hypothesis of Theorem
\ref{theor-main-I:BP} (cf. Definition \ref{def1.bp}).
We claim that
there exist uniform constants $\eta \in (0,1)$ and $c_0 >0$, such that
for any Borel subset $A\subset \Delta_X$,
\begin{equation}\label{eq7.19}
\sigma(A)\geq \eta \,\sigma(\Delta_X)\,\,\Longrightarrow\,\,
\hm^{X}(A)\geq c_0\,.
\end{equation}
Let us now prove this claim.
Suppose that
$A\subset \Delta_X$, with $\sigma(A)\geq (1-\alpha/2)\sigma(\Delta_X)$. Then
$$\sigma'(\pom'\cap A) \geq \frac\alpha2\,
\sigma (\Delta_X)\approx \alpha\, \sigma'\left(\Delta_\star\right)\,,$$
where $\sigma':=H^n\big|_{\pom'}$ denotes surface measure on $\pom'$
(so $\sigma=\sigma'$ on $\pom\cap\pom'$),
and where we have used that $\pom$ and $\pom'$ are both ADR. Since the hypotheses of Theorem
\ref{theor-main-I} apply in $\Omega'$, we deduce from \eqref{eq7.12} and a formal application of the maximum principle that
\begin{equation}\label{eq8.max_prin} \alpha^\theta\,\lesssim
\omega_{\Omega'}^{X}(\pom'\cap A)\leq
\omega^{X}(A)\,,
\end{equation}
where $\omega_{\Omega'}$ is harmonic measure for $\Omega'$.
Thus, we obtain \eqref{eq7.19}, with $\eta = (1-\alpha/2)$.
We caution the reader
that our use of the maximum principle to obtain the second inequality in
\eqref{eq8.max_prin} is not routine, since we are working in a regime where
the Wiener test may fail, and our solutions $X\to \hm^X(A)$ and $X\to \hm_{\Omega'}^X(\pom'\cap A)$
are not Perron solutions for the same domain, nor are they
continuous on the closures of the respective domains under consideration.
We shall give a rigorous justification of the essential inequality in \eqref{eq8.max_prin}
(namely, that $\alpha^\theta\lesssim \hm^X(A)$), at the end of this section.

It remains to establish \eqref{eq7.15*}.
To this end, we again follow the argument in \cite[Lemma 2.2]{BL}.
Fix $B$ and $B'$ as above, and define $U_k$ and $S_k$ as in
\eqref{eq7.16*}.    In fact, we proceed as we did under the hypotheses
of Theorem \ref{theor-main-I}, except that the proof of \eqref{eq7.17hm} will now be somewhat more delicate,
as we may no longer simply invoke the Harnack Chain condition.  Instead, we return to the original
approach of \cite{BL}.
It is enough to verify \eqref{eq7.16}, as the remainder of the proof is unchanged.
In particular, we obtain \eqref{eq7.20*}, which in turn yields \eqref{eq7.15*}, since
$2B'\subset B$.

As before, \eqref{eq7.16} is a direct consequence of \eqref{eq7.17hm} and \eqref{eq7.18+}.
The latter always holds, by Lemma \ref{Bourgainhm}, so we consider \eqref{eq7.17hm}.
Again we follow \cite{BL} essentially verbatim.
We suppose first that
there exists $Y\in S_k$ with $\delta(Y) =c\eps s/(100K)$,
where $c$ is the constant in Lemma \ref{Bourgainhm}.
For each such $Y$, we fix $y\in \pom$, with $|Y-y|=\delta(Y)$, and set $\Delta_Y:=
\Delta(y,K\delta(Y))$.   If $\eta_\eps\in(0,1)$ is chosen close enough to 1, depending on $\eps$ and the ADR constants of $\pom$,
and if $A\subset 2\Delta'$ with
$\sigma(A)\geq \eta_\eps\,\sigma(2\Delta')$,  then
$$\sigma(A\cap\Delta_Y) \geq \eta\,\sigma(\Delta_Y)\,,$$
for $\eta$ as in \eqref{eq7.19}, so that
$\hm^Y(A)\geq\hm^Y(A\cap\Delta_Y)\geq c_0.$  Thus, \eqref{eq7.17hm} holds in this case
(with $c\eps s/100$ now multiplied by $1/K$), by
Harnack's inequality, because even in the absence of the Harnack Chain condition, there is a Harnack path from any $X\in S_k\cap \{X:\,\delta(X)\geq c\eps s/(100K)\}$ to a point $Y$ in $S_k$ with
$\delta(Y) =c\eps s/(100K)$,
if the latter exists (just follow a geodesic path on $S_k$ from $X$ to the nearest
such Y).

On the other hand, suppose that there is no such $Y$.  Then either
$S_k\subset \{X\in\Omega:\,\delta(X)> c\eps s/(100K)\}$, or
$S_k\subset \{X\in\Omega:\,\delta(X)< c\eps s/(100K)\}$.  In the latter case,
\eqref{eq7.18+} holds now for all $X \in S_k$, so \eqref{eq7.16} follows trivially.
Otherwise, by continuity of $\delta$, there is a number $\rho>0$ such that
\begin{multline}\label{eq7.24}
\left\{ X\in\Omega:\,\rho\leq|X-z|\leq\left(5/4 +\eps \big(k+1/2\big)\right)s\right\}\\[4pt]\subset \,\,
\{X\in\Omega:\,\delta(X)\geq c\eps s/(100K)\}\,,\end{multline}
and $\delta(Y) = c\eps s/(100K)$ for some $Y\in S(\rho):=\{X\in\Omega:
|X-z|=\rho\}$.
In this case, we may repeat the analysis above, in which there was such a $Y$
on $S_k$.  In the present scenario, we have that  \eqref{eq7.17hm} holds for all
$X\in S(\rho)\cap \{X\in\Omega:\,\delta(X)\geq c\eps s/(100K)\}$, which in fact is all
of $S(\rho)$ by \eqref{eq7.24}.  But then by Harnack's inequality  we obtain
\eqref{eq7.17hm} (for all $X\in S_k$), because the containment
in \eqref{eq7.24} allows us to form a radial Harnack path between any $X\in S_k$,
and its projection onto $S(\rho)$.  We conclude that \eqref{eq7.16} holds under all circumstances.

To finish the proof of Theorem \ref{theor-main-I:BP}, it remains only to provide a rigorous justification of
\eqref{eq8.max_prin}.  We shall make up for the lack of continuity of the solutions by
proceeding as in
the removal of the qualitative hypothesis in the proof of Theorem \ref{theor-main-I}, with a few minor modifications.
We fix $\ep >0$ to be chosen momentarily,
and set $F:= A\cap\pom'$.  We recall that $H^n(F)\geq(\alpha/2) H^n(\Delta_X)$.
By outer regularity of Hausdorff measure and $\hm$,  there is a set $\mathcal{O}$,
relatively open in $\pom$, such that
$F\subset \mathcal{O}\subset \Delta_X\subset \pom$, and
$$H^n(\mathcal{O}\setminus F)+\hm^X(\mathcal{O}\setminus F)<\ep\,.$$
We let $\F\subset \dd(\pom)$ be a family of non-overlapping dyadic cubes whose union equals
$\mathcal{O}$, so that $H^n(\mathcal{O}) = \sum_{\F}H^n(Q_k)$, and we set
$$\F':=\left\{Q_k\in \F: H^n(Q_k\cap F)\geq \frac14H^n(Q_k)\right\}.$$
We claim that
\begin{equation}\label{eq8.28+}
\alpha \,H^n(\Delta_X)\lesssim \sum_{\F'}H^n(Q_k).
\end{equation}
Indeed, we have that
$$H^n(F) =\sum_{\F\setminus \F'} H^n(Q_k\cap F) +\sum_{ \F'} H^n(Q_k\cap F)
\leq \frac14 H^n(\mathcal{O}) + \sum_{ \F'} H^n(Q_k) \,,$$
whence \eqref{eq8.28+} follows, if we choose $\ep\ll H^n(F)$ .

Since each $Q_k\in\F'$ has an ample intersection with $F$, by Lemma \ref{lemmaCh} (vi),
we may choose a point $x_k\in Q_k\cap F\subset\pom'\cap\pom$, and a radius $r_k \approx \ell(Q_k)$, such that $Q_k\supset \pom\cap B_k$, where $B_k:= B(x_k,r_k)$.  We emphasize that, in particular, each $B_k$ is centered on $\pom'\cap\pom$.  Set $B_k':=
cB_k$, where $c\in(0,1)$ is the constant in Lemma \ref{Bourgainhm}.    Set $\F_N':= \dd_N(\pom')$, and
let $\Omega_N':= \Omega'_{\F_N}$ be the corresponding approximating domain relative to $\Omega'$.
By the ADR property and a covering lemma argument, and by choice of $N$ sufficiently large,
we can select a finite, pairwise disjoint sub-collection $\{B_k\}_{1\leq k\leq M}$, such that
$$
\alpha \,H^n(\td^N)\approx \alpha \,H^n(\Delta_X) \lesssim \sum_{k=1}^M H^n (B_k' \cap \pom'_N)
$$
where  $\td^N$ is a surface ball on $\pom_N'$ of radius $\approx \delta(X)$, such that $X$ is a Corkscrew point for $\td^N$ in $\Omega_N'$, and
$ \cup_{k=1}^M (B_k' \cap \pom'_N)\subset \td^N$.  We may apply Theorem \ref{theor-main-I} in
$\Omega'_N$ (see the discussion immediately following
\eqref{eq7.on} above), to obtain that $\hm^X_N$, the harmonic measure for the approximating domain
$\Omega_N'$, belongs to $A_\infty(\td^N)$ with bounds that are independent of $N$.

We now set $A':= (\cup_{k=1}^M B_k)\cap\pom$, and observe that $A'\subset\mathcal{O}$.
Since $X\to \omega^X(A')$ is continuous on $\overline{\Omega_N'}$, we may repeat the argument
in \eqref{eq8.17***}, {\it mutatis mutandis}, to obtain that
$$\hm^X(A) +\ep \geq\hm^X(F)+\ep \geq \hm^X(\mathcal{O}) \geq \hm^X(A') \gtrsim \alpha^\theta\,.$$
We choose $\ep\ll\alpha^\theta$, and it follows that $\alpha^\theta\lesssim \hm^X(A)$, as desired.
\end{proof}

\section{Proof of the Extrapolation Lemma}\label{s8}
To finish the proofs of Theorems \ref{theor-main-I} and \ref{theor-main-I:BP},
it remains to prove Lemma \ref{lemma:extrapol}.

\begin{proof}[Proof of Lemma \ref{lemma:extrapol}]
The proof follows the strategy introduced
in \cite{LM}, and developed further in \cite{HL}, \cite{AHLT} and \cite{AHMTT}.
In more precise detail, the argument is based on the systematic treatment given
in \cite{HM} in the Euclidean setting.

The proof uses an induction argument with continuous parameter. The induction hypothesis is the following: given $a\ge 0$,
\\[.3cm]
\null\hskip.5cm \fbox{\rule[8pt]{0pt}{0pt}$H(a)$}\hskip10pt \fbox{\rule[12pt]{0pt}{0pt}\ \parbox[c]{.75\textwidth}{%
There exist $\eta_a\in(0,1)$ and $C_a<\infty$ such that for every $Q\in \dd_{Q_0}$ satisfying  $\mut(\dd_Q)\le a\,\sigma(Q)$, it follows that
    $$
    \null\hskip-4cm
    F\subset Q,\quad \frac{\sigma(F)}{\sigma(Q)}\ge 1-\eta_a\quad \Longrightarrow \quad
    \frac{\omega(F)}{\omega(Q)}\ge \frac1{C_a}.
    $$}\ }

\

The induction argument is split in two steps.

\textbf{Step 1.} Show that $H(0)$ holds.

\textbf{Step 2.} Show that
there exists $b$ depending on $\gamma$,  dimension, and the ADR property such that for all $0\le a\le M_0$, $H(a)$ implies $H(a+b)$.

Once these steps have been carried out, the proof follows easily:  pick $k\ge 1$ such that
$(k-1)\,b<M_0\leq k\,b$ (note that $k$ only depends on $b$
and $M_0$). By \textbf{Step 1} and \textbf{Step 2}, it follows that $H(k\,b)$ holds.
Observe that $\|\mut\|_{\C(Q_0)}\le M_0\le k\,b$ implies $\mut(\dd_Q)\le k\,b \,\sigma(Q)$ for all
$Q\in \dd_{Q_0}$, and by $H(k\,b)$ we conclude \eqref{extrap:Ainfty:w}.

\subsection*{Step 1. $H(0)$ holds}

If $\mut(\dd_Q)=0$ then we take $\F$ to be empty, so that $\dd_Q\cap\dd_{\F} = \dd_Q$,
and $\P_\F \, \omega=\omega$. Then  \eqref{extrap:Carleson-delta} holds (since $0\le \gamma$) and therefore we can use \eqref{extrap:Ainfty:Pw} with $\omega$ in place of $\P_\F \, \omega$, which is the desired property.

\subsection*{Step 2. $H(a)$ implies $H(a+b)$}

Fix $0\le a\le M_0$ and $Q\in\dd_{Q_0}$ such that $\mut(\dd_{Q})\le(a+b)\,\sigma(Q)$, where we
choose $b$ so that $C\,b:= \gamma$ and $C$ is the constant in the righthand side of \eqref{Corona-sawtooth}. We also fix $F\subset Q$ with $\sigma(F)\ge (1-\eta)\,\sigma(Q)$,
where $0<\eta\le\eta_{a,b}$ and $\eta_{a,b}$ is to be chosen.
We may now apply Lemma \ref{lemma:Corona} and Remark \ref{remark:general-doubling-sawtooth} to the cube $Q$, to
construct the non-overlapping family of cubes
$\F=\{Q_j\}\subset\dd_Q$ with the stated properties.  Set
$$
E_0=Q\setminus \bigcup_{Q_j\in\F} Q_j,
\qquad\quad
G
=
\bigcup_{Q_j\in\F_{\rm good}} Q_j,
\qquad\quad
B
=
\bigcup_{Q_j\in\F\setminus \F_{\rm good}} Q_j,
$$
where $\F_{\rm good}=\big\{Q_j\in\F: \mut(\dd^{\rm short}_{Q_j})\le a\,\sigma(Q_j)\big\}$.
We recall that by \eqref{Corona-bad-cubes}, we have
$\sigma(B)/\sigma(Q)\leq (a+b)/(a+2b)\,.$

We shall also require the following ``pigeonhole'' lemma, which says that ``most'' of the cubes $Q_j$
have an ample overlap with $F$.

\begin{lemma}\label{lemma:selection}
Given $0<\tilde{\eta}<1$, we set
$$
\F_1=\{Q_j\in\F_{\rm good}: \sigma(F\cap Q_j)\ge (1-\tilde{\eta})\,\sigma(Q_j)\},
\qquad\qquad
G_1=\bigcup_{Q_j\in \F_1} Q_j.
$$
If $0<\eta\le\eta_1:=\tilde{\eta}\,\frac12\,\big(1-\frac{M_0+b}{M_0+2\,b}\big)$, then
$\sigma(E_0\cup G_1)\ge \eta_1\,\sigma(Q)$.
\end{lemma}

\begin{proof}

Take $\theta$ such that $\sigma(B)=\theta\,\sigma(Q)$, and $\theta_0=(M_0+b)/(M_0+2\,b)$. By \eqref{Corona-bad-cubes} and since $a\le M_0$ we obtain that $\theta \le \theta_0$:
$$
\theta\,\sigma(Q)
=
\sigma(B)
\le
\frac{a+b}{a+2\,b}\,\sigma(Q)
\le
\theta_0\,\sigma(Q).
$$
We set $B_1=\cup_{Q_j\in\F_{\rm good}\setminus \F_1} Q_j$ and
observe that $B_1\subset G\subset Q\setminus B$. Hence,
\begin{align*}
\sigma(F\cap B_1)
&=
\sum_{Q_j\in\F_{\rm good}\setminus \F_1} \sigma(F\cap Q_j)
<
(1-\tilde{\eta})\sum_{Q_j\in\F_{\rm good}\setminus \F_1} \sigma(Q_j)
\\
&=(1-\tilde{\eta})\,\sigma(B_1)
\le(1-\tilde{\eta})\,\sigma(Q\setminus B)
=(1-\tilde{\eta})\,(1-\theta)\,\sigma(Q).
\end{align*}
Thus, using that $\theta\le \theta_0$, we have
\begin{align*}
(1-\eta)\,\sigma(Q)
&\le
\sigma(F)
=
\sigma(F\cap E_0)+\sigma(F\cap B)+\sigma(F\cap G_1)+\sigma(F\cap B_1)
\\
&\le
\sigma(E_0)+\sigma(B)+\sigma(G_1)+(1-\tilde{\eta})\,(1-\theta)\,\sigma(Q)
\\
&
=
\sigma(E_0)+\sigma(G_1)+\big[\theta+(1-\tilde{\eta})\,(1-\theta)\big]\,\sigma(Q)
\\
&
\le
\sigma(E_0)+\sigma(G_1)+\big[1-\tilde{\eta}\,(1-\theta_0)\big]\,\sigma(Q)
\end{align*}
and therefore
$$
\sigma(E_0\cup G_1)
=
\sigma(E_0)+\sigma(G_1)
\ge
\big[\tilde{\eta}\,(1-\theta_0)-\eta\big]\,\sigma(Q)
\ge
\frac12\,\tilde{\eta}\,(1-\theta_0)\,\sigma(Q)
=\eta_1\,\sigma(Q),
$$
where we have used that $\eta\le\tilde{\eta}\,(1-\theta_0)/2=\eta_1$.
\end{proof}

We now return to the proof of Step 2.   We apply Lemma \ref{lemma:selection},
with $\tilde{\eta}\in(0,1)$ to be chosen.
Given $Q_j\in\F_1\subset \F_{\rm good}$ we have
that $\mut(\dd^{\rm short}_{Q_j})\le a\,\sigma(Q_j)$. Moreover,
$$
\dd^{\rm short}_{Q_j}
=
\dd_{Q_j}\setminus\{Q_j\}
=
\bigcup_{i} \dd_{Q^j_i},
$$
where $\{Q^j_i\}_i$ is the family of dyadic ``children'' of $Q_j$ (these are the subcubes of
$Q_j$ which lie in the very next dyadic generation $\dd_{k(Q_j)+1}$).
Then by pigeon-holing, there exists at least one $i_0$ such that
$Q^j_{i_0}=:Q_j'$ satisfies
\begin{equation}\label{pigeon}
\mut(\dd_{Q_j'})\le a\,\sigma(Q_j')
\end{equation}
(there could be more than one $i_0$ with this property, but we just pick one).
We define $\widetilde{\F}_1$ to be the collection of those
selected ``children'' $Q'_{j}$, with $Q_j \in \F_1$.
Let $C_0$ be the dyadically doubling constant of $\sigma$, i.e., $\sigma(Q)\le C_0\sigma (Q')$ for every $Q\in \dd_{Q_0}$, and for every ``child''  $Q'$ of $Q$.
Then, for each such $Q'_j$, using the definition of $\F_1$, and taking $0<\tilde{\eta}=\eta_a/C_0$
(where $0<\eta_a<1$ is provided by $H(a)$), we have
$$
\sigma(Q'_j\setminus F)
\le
\sigma(Q_j\setminus F)
\le
\tilde{\eta}\,\sigma(Q_j)
\leq
\tilde{\eta}\,C_0\sigma(Q_j')=
\eta_a\,\sigma(Q_j'),
$$
which yields $\sigma(Q_j'\cap F)\geq(1-\eta_a)\,\sigma(Q'_j)$. With this estimate and \eqref{pigeon} in hand, we can use the induction hypothesis $H(a)$ to deduce:
\begin{equation}\label{w-Q''}
\omega(Q_j'\cap F) \ge \frac1{C_a}\,\omega(Q_j'),\qquad \forall\, Q_j'\in \widetilde{\F}_1.
\end{equation}

On the other hand, if we set $\widetilde{G}_1=\cup_{Q_j'\in \widetilde{\F}_1}Q_j'$, then
$$
\sigma(\widetilde{G}_1)
=
\sum_{Q_j'\in\widetilde{\F}_1}\sigma(Q_j')
\ge
C_0^{-1}\,\sum_{Q_j\in\F_1}\sigma(Q_j)
=
C_0^{-1}\sigma(G_1)
$$
Thus, by  Lemma \ref{lemma:selection}, having now fixed $\tilde{\eta}$ above, we have that
\begin{equation*}
\sigma(E_0\cup \widetilde{G}_1)
=
\sigma(E_0)+\sigma(\widetilde{G}_1)
\geq
C_0^{-1}\,\sigma(E_0\cup G_1)
\geq
C_0^{-1}\, \eta_1\sigma(Q)=:\eta_2\,\sigma(Q),
\end{equation*}
if $\eta \leq \eta_1$, from which it follows that
\begin{equation*}
\sigma(F\cap(E_0\cup \widetilde{G}_1)) \geq\frac12\eta_2 \sigma(Q)=:\eta_{a,b}\,\sigma(Q),
\end{equation*}
if $\eta\leq \eta_2/2,$ since $\sigma(Q\setminus F)\leq \eta\, \sigma(Q)$.

We recall that the family $\F$ was constructed using Lemma \ref{lemma:Corona}
with $C\, b := \gamma$.
Consequently, by \eqref{Corona-sawtooth}, we may deduce that \eqref{extrap:Carleson-delta} holds,
so in turn, by hypothesis, we can apply \eqref{extrap:Ainfty:Pw} to the set
$F\cap (E_0\cup \widetilde{G}_1)$, obtaining
$$
\frac{\P_\F\, \omega(F\cap (E_0\cup \widetilde{G}_1))}{\P_\F\, \omega(Q)}
\ge \frac1{C_{\eta_{a,b}}}.
$$
As observed before, $\P_\F \,\omega(Q)=\omega(Q)$.
Thus, in order to establish the conclusion of $H(a+b)$, and consequently to complete the proof of
Lemma  \ref{lemma:extrapol}, it remains only to show that
\begin{equation*}
\P_\F\, \omega(F\cap (E_0\cup \widetilde{G}_1))\leq C\,\omega(F).
\end{equation*}
To this end, we use first the definition of $\P_\F$, then that $\omega$ is dyadically doubling and finally
\eqref{w-Q''} to obtain
\begin{align*}
\P_\F\, \omega(F\cap (E_0\cup \widetilde{G}_1))
&
=
\P_\F \,\omega(F\cap E_0) + \P_\F \,\omega(F\cap  \widetilde{G}_1)
\\
&=
\omega(F\cap E_0)+
\sum_{Q_j\in \F_1}
\frac{\sigma(Q_j'\cap F)}{\sigma(Q_j)}\,\omega(Q_j)\,\,
\\
&
\le
\omega(F\cap E_0)+C_\omega
\sum_{Q_j'\in \widetilde{\F}_1}
\omega(Q_j')\,\,
\\
&
\le
\omega(F\cap E_0)+C_\omega C_a\sum_{Q_j'\in \widetilde{\F}_1} \omega(Q_j'\cap F)
\\
&
\le C\,\omega(F).
\end{align*}
This concludes the proof of Lemma \ref{lemma:extrapol}.
\end{proof}


\appendix

\section{Inheritance of properties by  Carleson and Sawtooth regions}\label{ainherit}

This section is devoted to the proof of Lemma \ref{lemma2.30}, which states that
Carleson and Sawtooth regions inherit the Corkscrew, Harnack Chain and ADR properties
from the original domain $\Omega$.  Moreover, in the presence of the Corkscrew, Harnack Chain
and ADR properties, the UR property is transmitted
to the Carleson boxes $T_Q$ and $T_\Delta$.
We discuss these properties one at a time.
We shall find it convenient for our purposes in this section to continue to let
$\Delta$ denote a surface ball on $\partial\Omega$, while $\td$ will denote a surface ball
on the boundary of the sub-domain under consideration.  Similarly
$\delta(X)$ will continue to denote the distance from $X$ to $\partial\Omega$, while
$\tdelta(X)$ will denote the distance from $X$ to the boundary of the sub-domain under consideration.

In order to avoid possible confusion, let us emphasize that
the construction of our sawtooth and Carleson sub-domains is always based
on the Whitney decomposition of the domain under consideration at that moment,
even if that domain happens to be, say,
an approximating domain
$\Omega_N$ which had been constructed in the first place from Whitney cubes of
the original domain $\Omega$.

\subsection{Corkscrew}  For the sake of specificity, we treat only the case of a local
sawtooth region $\Omega_{\F,Q}$.  The proof for the global sawtooth $\Omega_\F$ is almost identical.
Moreover, specializing to the case that $\F =\emptyset$, we see that
the result for a sawtooth $\Omega_{\F,Q}$ applies immediately to the Carleson box $T_Q$,
and therefore also almost immediately
to any box $T_\Delta$, since the latter is a union of a bounded number of $T_Q$'s.

We fix $Q_0\in \dd$, and a pairwise disjoint family $\{Q_j\}=\F\subset \dd_{Q_0}$,
and let $\Omega_{\F,Q_0}$ denote the associated local sawtooth region
(cf. \eqref{eq2.discretecarl}-\eqref{eq2.sawtooth2}).  Set
$$\td:=\td(x,r):=B(x,r)\cap\partial\Omega_{\F,Q_0},$$
with $r\lesssim \ell(Q_0)$ and $x\in \partial\Omega_{\F,Q_0}$.
We suppose first that $x\in \partial\Omega_{\F,Q_0}\cap\partial\Omega$.
Then by construction of $\Omega_{\F,Q_0}$, there is a $Q\in \dd_{\F,Q_0}$, with $x\in \overline{Q},$
and $r\approx 100K_0\,\ell(Q)$ (see Proposition \ref{prop:sawtooth-contain}). Consequently, by \eqref{eq2.whitney3}-\eqref{eq2.whitney2}, we have
$$X_Q\in U_Q\subset B(x,r)\cap \Omega_{\F,Q_0} \,,$$
where $X_Q$ is a Corkscrew point for $\Omega$, relative to $Q$, and which we have assumed
(without loss of generality) to be the center of some $I\in \W_Q^*$. This same $X_Q$ then serves
as a Corkscrew point for $\Omega_{\F,Q_0}$, relative to $\td(x,r)$, with Corkscrew constant $c\approx
1/(100K_0)$.

Next, we suppose that $x\in \partial\Omega_{\F,Q_0}\setminus\partial\Omega,$
where as above $\td:=\td(x,r).$  Then by definition of the sawtooth region,
$x$ lies on a face of a fattened Whitney cube $I^*=(1+\lambda)I$, with
$I\in \W_Q^*$, for some
$Q\in \dd_{\F,Q_0}.$  If $r\lesssim \ell(I)$, then trivially there
is a point $X^{\!\star}\in I^*$ such that $B(X^{\!\star},cr)\subset
B(x,r)\cap\interior(I^*)\subset B(x,r)\cap\Omega_{\F,Q_0}$.  This $X^{\!\star}$ is then a Corkscrew point
for $\td$.  On the other hand, if $\ell(I)<r/(MK_0),$ with $M$ sufficiently large to be chosen momentarily, then there is a $Q'\in \dd_{\F,Q_0}$,
with $\ell(Q')\approx r/(MK_0),$ and $Q\subseteq Q'$.  Now fix $I'\in \W_{Q'}\subset\W^*_{Q'}$, and observe that
$$|x-X(I')|\lesssim \dist(I,Q) + \dist(Q',I')\lesssim K_0\,\ell(I) +K_0\,\ell(I')\lesssim r/M.$$
Note that $B(X(I'),cr)\subset \interior(I')$, for $c\approx (MK_0)^{-1}$.
Moreover, for $M$ large enough we have that $B(X(I'),cr)\subset B(x,r)\cap\interior(I')\subset
B(x,r)\cap\Omega_{\F,Q_0},$
so that $X(I')$ is a Corkscrew point for $\td$.

\subsection{Harnack Chain}  We establish the Harnack Chain condition for a local sawtooth
$\Omega_{\F,Q}$, of which, as noted above, the Carleson box $T_Q$ is a special case
(with $\F=\emptyset$).  The proof for a global sawtooth  is almost the same, and we omit it.
We shall discuss the Carleson boxes $T_\Delta$ at the end of this subsection.

Fix $Q\in\dd$, and a pairwise disjoint family $\F\subset \dd_Q$,
and let $\Omega_{\F,Q}$ be the corresponding local sawtooth region.
Let $X_1,\, X_2 \in \Omega_{\F,Q}$.  By definition of the sawtooth regions,
there exist $Q_1,\, Q_2\in \dd_{\F,Q}$,
with $X_i\in (1+\lambda)I_i=I^*_i$ where $I_i\in \W^*_{Q_i},\, i=1,2$.  Without loss of generality we may suppose that
$\ell(Q_1)\leq\ell(Q_2).$  We first observe that the desired result is clear if $I_1=I_2$, or more
generally, if
$I_1^*$ and $I_2^*$ overlap.
Therefore, we may suppose that
\begin{equation}\label{eqA.0*}
\dist(I_1^*,I_2^*) \gtrsim \ell(I_2) \gtrsim\ell(I_1)
\end{equation}(cf. \eqref{eq2.29*}.)
In order to construct a Harnack Chain under these circumstances, relative to
$\Omega_{\F,Q}$, from $X_1$ to $X_2$, it is convenient to make a few
simple reductions and observations, as follows.

\begin{enumerate}

 \item  It is enough to treat the case that $X_i$ is the center of $I_i$.  If $X_i$ is near the boundary of the sawtooth (and hence also near the boundary of $I_i^*$), then
$\dist(X_i ,\partial I_i^*) \approx \dist(X_i, \partial\Omega_{\mathcal{F},Q})$, so that the Harnack Chain within $I_i^*$, that connects $X_i$ to the center $X(I_i)$, is also a Harnack chain for the sawtooth. On the other hand,  if $X_i$ is not near the boundary of the sawtooth, then we can easily join $X_i$ with $X(I_i)$  by a bounded number of balls of radius $\approx \ell(I_i)$ with distance to the boundary of $\partial\Omega_{\mathcal{F},Q}$ comparable to $\tdelta(X_i)$.

\item By construction (cf. \eqref{eq2.whitney3}-\eqref{eq2.whitney2}), we may then further suppose that $X_i=X_{Q_i},$ the designated Corkscrew point (for the ambient domain $\Omega$), relative to $Q_i$.

\item Recall that by construction,
if $Q'\subset Q''$ belong to consecutive generations
in $\dd$ (i.e., $k(Q'') = k(Q')-1$), then $U_{Q'}\cap U_{Q''}$
contains the Corkscrew point $X_{Q'}$ (cf. \eqref{eq2.whitney2*})
and is therefore non-empty.
Thus, by
\eqref{eq2.whitney3}-\eqref{eq2.whitney2} there is a Harnack Chain joining the respective Corkscrew points $X_{Q'}$ and $X_{Q''}$.
\item We note that by definition, if $Q_i\in \dd_{\F,Q},$
then also $Q'\in\dd_{\F,Q}$ for every $Q'$ such that $Q_i\subseteq Q'\subseteq Q.$
\item If $X(I)$ denotes the center of a Whitney cube $I$, then $\delta(X(I)) \approx
\tdelta(X(I)) \approx \ell(I).$
\end{enumerate}

With these observations in mind, we consider three cases.  Set
$R:=|X_1-X_2|$.

\noindent {\bf Case 1}:  $Q_1\subseteq Q_2$.  In this case, $R\lesssim \ell(Q_2)$
(with $R\approx \ell(Q_2)$ if $\ell(Q_2)\gg\ell(Q_1)$), and
$\min(\tdelta(X_1),\tdelta(X_2))\gtrsim \ell(Q_1)$.  Consequently,
we may form a Harnack Chain of cardinality $\approx k(Q_1)-k(Q_2) +1$
that connects the Corkscrew points of every $Q'$, with $Q_1\subseteq Q'\subseteq Q_2$.

Before proceeding to the remaining cases, we observe that if Case 1 does not hold, then
$Q_1$ and $Q_2$ are disjoint, whence it follows from \eqref{eqA.0*}
and observations (1) and (5) above that
\begin{equation}\label{eqA.one}R\gtrsim \tdelta(X_2)\approx \ell(Q_2)\geq\ell(Q_1)\approx \tdelta(X_1).
\end{equation}
Of course, we also have $R\lesssim \ell(Q).$

\noindent {\bf Case 2}:  $Q_1\cap Q_2=\emptyset$, but have a common ancestor
$Q^*\subseteq Q$, with $\ell(Q^*)\approx R$.  We may then proceed as in Case 1, to construct
respective Harnack Chains, connecting each of
$X_1$ and $X_2$, to $X_{Q^*}$.   The union of these two chains connects $X_1$ to $X_2$.

\noindent {\bf Case 3}:  $Q_1$ and $Q_2$ have no common ancestor
of length $\approx R$.  In this case, we may suppose that $R<\ell(Q)/(MK_0)$,
where $M$ is a sufficiently large number to be chosen momentarily.
Indeed, if not, then $\ell(Q)/(MK_0)\leq R\lesssim \ell(Q)$, in which case $Q$
would be a common ancestor with $\ell(Q)\approx R$.

Thus, since $R<\ell(Q)/(MK_0)$, there exist $Q_1^*,Q_2^*\in \dd_Q$
such that, for $i=1,2$,
$Q_i^*$ is an ancestor of $Q_i$, with $\ell(Q^*_1)=\ell(Q^*_2)\approx MK_0R$.
Since $\dist(X_i,Q_i)\lesssim K_0\, \ell(Q_i)$ by construction (cf. \eqref{eq2.whitney2}),
we then have that $\dist(Q_1,Q_2)\lesssim K_0 R$ (by \eqref{eqA.one} and the triangle inequality),
and therefore also that
$$\dist(Q^*_1,Q^*_2)\leq C \ell(Q^*_1)/M\leq \ell(Q^*_1)= \ell(Q^*_2)\,,$$
by choice of $M$ large enough.  Consequently, by \eqref{eq2.whitney2**},
$\W^*_{Q_1^*}\cap \W^*_{Q_2^*}$ is non-empty,
whence there is a Harnack Chain connecting the respective Corkscrew points
$X_{Q_1^*}$ and $X_{Q_2^*}$.  We may then proceed as above to construct a Harnack
Chain from $X_i$ to $X_{Q_i^*},\, i=1,2$, and the proof of the Harnack Chain condition
for the sawtooth $\Omega_{\F,Q}$ is now complete.

We finish this subsection by verifying the Harnack Chain property
for a Carleson box $T_\Delta$.   Let $X_1,X_2\in T_\Delta$, with
$a:=\tdelta(X_1)\leq\tdelta(X_2)=:b$.
As above, we may suppose that $I_1^*$ and $I_2^*$ are separated, and thus as in observation (1),
that each $X_i,\, i=1,2$ is the respective center of the Whitney cube $I_i$ whose dilate contains it.

By definition of $T_\Delta$ (cf.  \eqref{eq2.box3}-\eqref{eq2.box2}),
and since $X_i$ is the center of $I_i$,
we have $X_i\in T_{Q^i},$
where $Q^i\in \dd^\Delta, \,i=1,2$.
By \eqref{eq2.whitney2**}, $\W_{Q^1}\cap\W_{Q^2}$ is non-empty.
Consequently, there is a
Harnack Chain connecting the respective Corkscrew points $X_{Q^1}$ and
$X_{Q^2}$, so in the case that  $|X_1-X_2|=:\Lambda a \approx r_\Delta$,
we may connect $X_1$ to $X_{Q^1}$ to $X_{Q^2}$ to $X_2$.

Therefore, we may now suppose that $|X_1-X_2|=\Lambda a\leq r_\Delta/(MK_0)$,
for some sufficiently large $M$ to be chosen momentarily.
We note that there is a uniform constant $c>0$ such that
$\Lambda\geq c$, since $X_1$ and $X_2$ are the respective centers
of non-overlapping Whitney cubes (cf. observation (5) above).
We now claim that we also
have $b\lesssim \Lambda a.$  Indeed, if $b\gg \Lambda a$, then by the triangle inequality,
$a\geq b-\Lambda a \gg \Lambda a,$ which contradicts the uniform lower bound for $\Lambda$.
Therefore, by observation (5) above, and by construction of each $T_{Q^i}$,
there exist
$\widetilde{Q}^1\subset Q^1,$
$\widetilde{Q}^2\subset Q^2$ such that $X_i\in T_{\widetilde{Q}^i}\,$,
$\ell(\widetilde{Q}^1)=\ell(\widetilde{Q}^2)
\approx MK_0\, \Lambda a$,  and
$$
\dist(X_1,\widetilde{Q}^1)\lesssim K_0\, a,
\qquad \dist(X_2,\widetilde{Q}^2)\lesssim K_0\, b\lesssim K_0\,\Lambda a.
$$
By the triangle inequality, we then have
$$\dist(\widetilde{Q}^1,\widetilde{Q}^2)\leq C\ell(\widetilde{Q}^1)/M\leq
\ell(\widetilde{Q}^1)=\ell(\widetilde{Q}^2)\,,$$
by choice of $M$ large enough.
By \eqref{eq2.whitney2**}, $\W_{\widetilde{Q}^1}\cap\W_{\widetilde{Q}^2}$
is non-empty, so that we may construct a Harnack Chain from $X_1$ to $X_2$,
by a now familiar argument,
via the Corkscrew points $X_{\widetilde{Q}^1}$ and $X_{\widetilde{Q}^2}$.

\subsection{ADR}  Suppose that $\partial\Omega$ is ADR, and we show first that for each
$Q\in \mathbb{D}(\partial\Omega)$, the boundary of the ``Carleson box'' $T_Q$ is
also ADR.    We begin with the upper bound.  Let $x\in \partial T_Q$, and let
$\td:=\td(x,r) := B(x,r)\cap\partial T_Q$, with $r\lesssim \diam Q$.
If $B(x,r)$ meets $\partial\Omega$, then there is a point
$x'\in \partial\Omega$ such that
$B(x,r)\subset B(x',2r).$  Consequently,
$$H^n\left(B(x,r)\cap \partial\Omega \cap\partial T_Q\right)\, \leq\,
H^n\left(\Delta(x',2r)\right)\,\lesssim \,r^n,$$
since $\partial \Omega$ is ADR.

Now consider $\td \setminus \partial\Omega$.  This portion of $\td$ is a contained in
a union of faces (or partial faces)
of fattened Whitney cubes $I^*=(1+\lambda)I$.    Let $\I_Q$ denote
the collection of Whitney cubes $I$ for which
$\partial I^*$ meets $\partial T_Q$, and $\interior(I^*)\subset T_Q$.
Suppose that $I\in \I_Q$ is a Whitney cube such that $\partial I^*$ meets
$\td$.  Then
\begin{equation}\label{eqA.1}H^n\left(\td\cap\partial I^*\right)\leq
H^n\Big(B(x,r)\cap \partial I^*\Big)\lesssim \min(\ell(I)^n,r^n).
\end{equation}
Therefore,
$$\sum_{I\in\I_Q: \,\ell(I)\geq r/(MK_0)}H^n\left(\td\cap \partial I^*\right)\lesssim r^n ,$$
because only a bounded number of terms can appear in this sum.  Here, $M$
is a sufficiently large number to be chosen, and $K_0$ is the same constant appearing in
\eqref{eq2.whitney2}.
It remains to consider
$$\sum_{I\in\I_Q: \,\ell(I)< r/(MK_0)}H^n\left(\td\cap \partial I^*\right)\,=\,\sum_{k:2^{-k}< r/(MK_0)}\,
\sum_{I\in\I^k_Q}H^n\left(\td\cap \partial I^*\right),$$
where $\I^k_Q := \{I\in \I_Q: \ell(I) = 2^{-k}\}.$    It is then enough to
show that there is an $\epsilon >0$ such that for each $k$ with $2^{-k}< r/(MK_0)$, we have
\begin{equation}\label{eqA.2}
\sum_{I\in\I^k_Q}H^n\left(\td\cap \partial I^*\right) \lesssim 2^{-k\epsilon}r^{n-\epsilon}.
\end{equation}
It follows from \eqref{eqA.1} that the latter bound will hold if the cardinality of the set of $I$
which make a non-trivial contribution to the sum is no larger than
\begin{equation}\label{eqA.3}C (2^kr)^{n-\epsilon}.\end{equation}
We recall that by the definition of $T_Q$ (cf. \eqref{eq2.whitney1}-\eqref{eq2.box}),
for each $I \in \I_Q^k$, there is a $Q_I\in\dd_Q$ such that $\ell(Q_I) \approx \ell(I)= 2^{-k},$
and $\dist(I,Q_I)\lesssim K_0\,\ell(I).$    Since $2^{-k}< r,$
there is a uniform constant $C$ such that
$B(x,Cr)$ contains each such $Q_I$, for every $I$ such that $\partial I^*$
meets $\td$.  We may then cover
$B(x,Cr)\cap Q$ by a bounded number of subcubes $Q'\in\dd_Q$, with $\ell(Q')\approx r$,
so that each relevant $Q_I$ is contained in some $Q'$.   It is enough to
consider those $Q_I$ contained in one such $Q'$.
We therefore now fix $Q'$ and $k$, and distinguish two types of $Q_I\subset Q'$:
\begin{eqnarray*}& {\rm Type\,\, 1}:\qquad \dist(Q_I,(Q')^c)>2^{-\gamma k}\, r^{1-\gamma}\\
&{\rm Type\, \,2}:\qquad \dist(Q_I,(Q')^c)\leq 2^{-\gamma k} \,r^{1-\gamma}
\end{eqnarray*}
where we have fixed $\gamma\in(0,1)$.  We note that there are at most a bounded
number of $I$'s corresponding to each $Q_I$.   Thus, since $2^{-k}< r/(MK_0)\ll r,$
by Lemma \ref{lemmaCh} $(vi)$ we have that the cardinality of the set of $I$'s for which
$Q_I$ is of Type 2 is no larger than $C(2^kr)^{-\gamma\eta} r^n/(2^{-kn}) \approx
(2^kr)^{n-\gamma \eta} $, which is \eqref{eqA.3}, with $\epsilon = \gamma\eta.$

We now claim that for $M$ chosen large enough, depending on $\gamma$ and $K_0$, the collection
of $I$ such that $\partial I^*$ meets $\td$, and for which $Q_I$ is of Type 1, is empty.
Indeed, if $Q_I$ is of Type 1, and if $M$ is sufficiently large, we then have
$$\dist(Q_I,(Q')^c) \,>\, 2^{-\gamma k}\, r^{1-\gamma} \,
>\,(MK_0)^{(1-\gamma)}2^{-k}\,\gg\,K_0\, \ell(I)\,\gtrsim\,\dist(I,Q_I).$$
Consequently, if $y\in \partial\Omega$ satisfies $\dist(I,y)\lesssim K_0\,\ell(I)$, then
$$
\dist(y,Q_I)\lesssim K_0\,\ell(I)\ll\dist(Q_I,(Q')^c),
$$ so that $y\in Q'$, and
$\dist(y,(Q')^c)\gg K_0\,\ell(I).$   Now consider any Whitney cube $J\in \W$ that touches
$I$.  Then $\dist(J,\partial\Omega)\approx \ell(J)
\approx\ell(I)\approx \dist(I,\partial\Omega)$,
so that for some $y_J\in \partial\Omega,$ we have
$\dist(y_J,J) \approx \dist(y_J,I) \ll C_0\,\ell(I)\leq K_0\,\ell(I)$
(cf. \eqref{eq2.whitney1} and \eqref{eq2.whitney2}.)  Thus, $y_J\in Q'$, and
$\dist(y_J,(Q')^c)\gg K_0\,\ell(I) \approx K_0\,\ell(J).$   It follows that there is a $Q_J\in \dd(\partial\Omega)$, with $Q_J\subset Q'\subset Q$,  $y_J\in \overline{Q}_J$,  $\ell(Q_J) = \ell(J)$, and
$\dist(Q_J,J)\leq\dist(y_J,J)\leq C_0\,\ell(J).$ Therefore, $J\in \W_{Q_J}$.  Since this is true for all
Whitney cubes $J$ that touch $I$, we have in particular that every point on $\partial I^*$ is an
interior point of $T_Q$, hence $\td \cap\partial I^* =\emptyset.$
We have now established the upper bound $H^n(\td(x,r))\lesssim r^n.$

The lower bound is easy.
Consider $B:=B(x,r)$, $r\lesssim \diam Q,$ with $x\in \partial T_Q.$  If $B\cap Q$ contains a surface ball
$\Delta \subset \partial\Omega$, with radius $r_\Delta \gtrsim r$, then we are done, by the ADR property of $\partial\Omega$.  Otherwise, if $B\cap Q$ contains no such surface ball, then $\dist(x,Q) \gtrsim r$,
whence it follows that $x\in \partial I^*$, where $I$ is a Whitney cube with $\ell(I)\gtrsim r/K_0$
(cf. \eqref{eq2.whitney2}), and where $x$ lies in a subset $F$ of a (closed)
face of $I^*$,  with $F\subset B\cap \partial T_Q$,
and $H^n(F\cap B)=H^n(F)\gtrsim (r/K_0)^n$,
as desired.

Next, we discuss the ADR property of a Carleson region $T_\Delta$.
By definition (cf. \eqref{eq2.box2}), $T_\Delta$ is a union of a bounded number
of regions $T_Q$.  The upper bound in the ADR condition is then an immediate consequence of
the corresponding bound for $T_Q$.  The lower bound is proved in the same way as it was
for $T_Q$ depending on whether or not the ball $B$ has an ample intersection with
some $Q\in \dd^\Delta$.  We omit the routine details.

Finally, we establish the ADR property for the global \eqref{eq2.sawtooth1}
and local \eqref{eq2.sawtooth2} sawtooth regions.  The proofs are similar,
so for the sake of specificity, we treat the global sawtooth $\Omega_{\F}.$  We first prove the
upper bound in the ADR condition.
Fix $B:= B(x,r)$.  The desired
bound for $H^n(B\cap\partial \Omega\cap\partial\Omega_{\F})$ is an immediate consequence
of the fact that $\partial \Omega$ is ADR.

Now consider
$\Sigma:=\partial\Omega_{\F}\setminus\partial\Omega.$  We observe that this portion of the boundary consists of (portions of) faces of certain fattened Whitney cubes $J^* =(1+\lambda)J$,
with $\interior(J^*)\subset\Omega_\F$,
which meet some $I\in\W$ for which $I\notin \W_Q^*$, for any $Q\in \dd_{\F}$
(so that $\tau I\subset \Omega\setminus\Omega_{\F}$
for some $\tau\in(1/2,1)$; cf. \eqref{eq2.30*}.)
Necessarily, $I\in \W_{Q'}^*$, where $Q'\in \dd_{Q_j}$
for some $Q_j \in \F$.
For each $Q_j\in \F$, we set
$$\R_{Q_j}:=\cup_{Q'\in \dd_{Q_j}}\W_{Q'}^*,$$
and denote by $\F_B$
the sub-collection of those $Q_j\in \F$ such that
there is an $I\in \R_{Q_j}$ for which $B\cap\Sigma$ meets
$I $.
We then split
the latter collection into $\F_B = \F_1\cup\F_2$, where $Q_j \in \F_1$ if $\ell(Q_j) < r$,
and $Q_j \in \F_2$ if $\ell (Q_j)\geq r.$  We consider the contribution of the latter first.
Suppose that $Q_j$ and $Q_k$ are both in $\F_2$, and without loss of generality
that $r\leq \ell(Q_j)\leq \ell(Q_k).$  Since $B$ meets some $I\in \R_{Q_j}$, we obtain
in particular that
$\dist(y,\partial\Omega)\lesssim \ell(Q_j),\,\forall y \in B.$   Thus,
$B\cap \Sigma$
lies within $C \ell(Q_j)$ of  $\partial\Omega$, and therefore meets no
Whitney cubes of side length greater than $C\ell(Q_j)$.  Consequently, any such
Whitney cube $I'\in \R_{Q_k}$, which meets $B$, must lie within
$CK_0\,\ell(Q_j)$ of $Q_k$.  Therefore, for any pair $Q_j,Q_k \in \F_2$, we have that
$\dist(Q_j,Q_k)\lesssim \min(\ell(Q_j),\ell(Q_k))$
(with implicit constants depending on $K_0$.)
Since the cubes in $\F$ are pairwise disjoint, it follows that
the cardinality of $\F_2$ is uniformly bounded, hence
$$ H^n\left(B\cap\Sigma\cap\left(\cup_{Q_j\in\F_2}\cup_{I\in\R_{Q_j}} I\right)\right)
\lesssim \sup_{Q_j\in \F}H^n\left(B\cap\Sigma_j\right)\,,$$
where $\Sigma_j:=\Sigma\cap(\cup_{I\in\R_{Q_j}} I).$
The desired bound for the contribution of $\F_2$ is an immediate consequence of
following estimate, which holds for every $Q_j\in\F$:
\begin{equation}\label{eqA.4}
H^n\left(B\cap\Sigma_j\right) \lesssim \left(\min\left(r,\ell(Q_j)\right)\right)^n.
\end{equation}

Let us take the latter bound for granted momentarily,
and consider the contribution of $\F_1$.  If $Q_j\in \F_1$, then
$Q_j\subset B^*:=CK_0\,B$ for some uniform constant $C$.  By the case $r> \ell(Q_j)$
of \eqref{eqA.4},
we have that $H^n(\Sigma_j\cap B)
\lesssim H^n(Q_j)$.  Therefore,
$$H^n\left(B\cap\left(\cup_{Q_j\in \F_1}\Sigma_j\right)\right)\,\lesssim\,
\sum_{\F_1}H^n(Q_j)\,\leq H^n\left(B^*\cap\partial \Omega\right)\approx(K_0\,r)^n,$$
since the $Q_j$'s are pairwise disjoint.

Thus, to finish proving the upper ADR bound for the sawtooth regions, it remains only to establish
\eqref{eqA.4}.    Suppose first that $\ell(Q_j)\lesssim r.$  We write
$$\Sigma_j := \bigcup_{k:2^{-k}\lesssim \ell(Q_j)}\, \Sigma_j^k,$$
where $\Sigma_j^k=\Sigma\cap(\cup_{\{I\in \R_{Q_j}:\,\ell(I)=2^{-k}\}}\, I)= \Sigma_j\cap(\cup_{\{I\in \R_{Q_j}:\,\ell(I)=2^{-k}\}}\, I).$
We observe that for any $I\in\W$,
\begin{equation}\label{eqA.6***}
H^n(\Sigma\cap I)\lesssim \ell(I)^n\,.
\end{equation}
Moreover,
there are at most a bounded number of $I\in \R_{Q_j}$ for which $\ell(I) \approx \ell(Q_j)$, so that
$$H^n\left(\sum_{k:2^{-k}\approx \ell(Q_j)}\Sigma_j^k\right) \lesssim \ell(Q_j)^n\,,$$
as desired.    On the other hand, suppose $I\in \R_{Q_j}$, with
$\ell(I)=2^{-k}\ll\ell(Q_j)$.  Then there is a $Q_I\in\dd_{Q_j}$, with $I\in \W^*_{Q_I}$.
In addition, if $I$ meets $\Sigma_j$, then $I$ meets $J^*$,
for some $J\in \W^*_{Q'}$, with $Q' \in \dd_\F$,
and $\ell(Q')\approx\ell(J)\approx\ell(I)\ll\ell(Q_j).$  We note that
$Q'\cap Q_j =\emptyset,$ by definition
of $\dd_\F$, and the fact that $\ell(Q')<\ell(Q_j)$.  Consequently,
$$\dist(Q_I,(Q_j)^c) \leq \dist(Q_I,Q')\lesssim K_0 \, 2^{-k}.$$
Notice that
for each such $Q_I$, there are at most a bounded number of $I'\in \W^*_{Q_I}$
(indeed, by definition  of $\W^*_{Q_I}$, all such $I'$ satisfy
$\ell(I') \approx \ell(Q_I) \approx \dist(I',Q_I)$).
By Lemma \ref{lemmaCh} $(vi)$ we therefore have that
$$\#\left\{I\in \R_{Q_j}:  \ell(I)=2^{-k},\ I\cap\Sigma_j^k\neq\emptyset\right\}\lesssim
\left(2^k\ell(Q_j)\right)^{n-\eta},$$
(where the implicit constant depends upon $K_0$), whence it follows that
$$H^n\left(\sum_{k:2^{-k}\ll \ell(Q_j)}\Sigma_j^k\right) \lesssim \ell(Q_j)^n\,.$$
Thus \eqref{eqA.4} holds in the case $r\gtrsim \ell(Q_j).$


Now suppose that $r\ll \ell(Q_j)$.  If $x\notin\pom$, then $B=B(x,r)$ is centered on a face
of some $J^*_x$, with $\interior(J^*_x)\subset \Omega_{\F}$.
If $\ell(J_x)\gg r$, we are done, by the nature of Whitney cubes.
On the other hand, if
$\ell(J_x) \lesssim r$, or if $x\in \pom$, then for each $I\in \R_{Q_j}$ which
meets $B$, we have that $\ell(I)\lesssim r$, and also that
$B(x,Cr)$ meets $Q_I$, for some uniform constant $C$,
where $Q_I\in\dd_{Q_j}$ is defined as in the previous paragraph.
We may then cover
$B(x,Cr)\cap Q_j$ by a bounded number of subcubes $Q^i\subset Q_j$,
with $\ell(Q^i)\approx Mr$,
so that each relevant $Q_I$ is contained in some $Q^i$.  Here, $M$ is a sufficiently large
number, to be fixed momentarily.
Now suppose that $I$ meets
$\Sigma_j$.  We may then proceed as in the previous paragraph, except that in this case
we consider $\dist(Q_I, (Q^i)^c),$ and $\ell(Q_j)$ is replaced by $\ell(Q^i)\approx M r.$
As above, we find that $I$ meets some $J^*$, with $J\in\W_{Q'}^*$
and $Q'\in\dd_\F$, so that $\ell(Q')\approx\ell(J)\approx\ell(I)$.  In the present scenario,
we have $\ell(I)\lesssim r$, therefore $\ell(Q')<\ell(Q^i)$, for $M$ chosen large enough
and consequently $Q'\cap Q^i=\emptyset$.
The rest of the argument follows as before.
We omit the details.

Finally, to complete our discussion of the ADR property, it remains only to
prove the lower ADR bound for the sawtooth regions.  For the sake of specificity,
we treat only the case of a local sawtooth, as the proof in the global case is similar.

Fix now $Q_0\in\dd$, $r\lesssim\diam Q_0$ and  $x\in \pomfqo$, where $\F\subset\dd$ is a disjoint family, and set
$B:=B(x,r)$ and $\td=\td(x,r):=B\cap\pomfqo$.  We consider two main cases.  As usual,
$M$ denotes a sufficiently large number to be chosen.

\noindent{\bf Case 1}:  $\delta(x)\geq r/(MK_0)$.
In this case, for some $J$ with
$\interior(J^*)\subset \Omega_{\F,Q_0}$, we have that
$x$ lies on a subset $F$ of a (closed)
face of $J^*$, satisfying $H^n(F)\gtrsim (r/(MK_0))^n$,  and $F\subset \pomfqo$.
Thus, $H^n(B\cap \pomfqo)\geq H^n(B\cap F)\gtrsim (r/(MK_0))^n,$ as desired.

\noindent{\bf Case 2}:  $\delta(x)< r/(MK_0)$.   In this case, we have that
$\dist(x,Q_0) \lesssim r/M.$
Indeed, if $x\in\pom\cap\pomfqo$, then by Proposition
\ref{prop:sawtooth-contain}, $x\in \overline{Q_0}$, so that $\dist(x,Q_0) =0$.
Otherwise, there is some cube
$Q\in \dd_{\F,Q_0}$ such that $x$ lies on the face
of a fattened Whitney cube $I^*$, with $I\in \W_{Q}^*$, and
$\ell(Q)\approx\ell(I)\approx \delta(x)<r/(MK_0)$.   Thus,
$$\dist(x,Q_0)\lesssim\dist(I,Q)\lesssim K_0\,\ell(Q)\lesssim r/M.$$
Consequently,  we may choose $\hat{x}\in Q_0$
such that $|x-\hat{x}|\lesssim r/M$.
Fix now $\widehat{Q}\in \dd_{Q_0}$
with $\hat{x}\in \widehat{Q}$ and $\ell(\widehat{Q})\approx r/M$.  Then
for $M$ chosen large enough we have that
$\widehat{Q}\subset B(\hat{x},r/\sqrt{M})
\subset B(x,r)$.  We now consider two sub-cases.

\noindent{\it Sub-case 2a}: $B(\hat{x},r/\sqrt{M})$  meets a $Q_j\in\F$ with
$\ell(Q_j)\geq r/M$.  Then in particular, there is a $Q\subseteq Q_j$,
with $\ell(Q)\approx r/M$, and $Q\subset B(\hat{x},2r/\sqrt{M})$.  By Lemma
\ref{lemma4.9}, there is a ball $B'\subset \ree\setminus \overline{\Omega_{\F,Q_0}}$,
with radius $r' \approx \ell(Q)/K_0\approx r/(K_0M)$, such that $B'\cap\pom\subset Q$, and thus also
$B'\subset B$ (for $M$ large enough).
On the other hand, we have already established above that $\Omega_{\F,Q_0}$
satisfies the (interior) Corkscrew condition, so there is another ball $B''\subset B\cap\Omega_{\F,Q_0}$,
with radius $r''\approx r$.  Therefore, by the isoperimetric inequality and the structure theorem
for sets of locally finite perimeter (cf. \cite{EvG}, pp. 190 and 205, resp.)
we have $H^n(\td)\gtrsim c_{K_0} r^n$.

\noindent{\it Sub-case 2b}:  there is no $Q_j$ as in sub-case 2a.   Thus,
if $Q_j\in\F$ meets $B(\hat{x},r/\sqrt{M})$, then $\ell(Q_j)\leq r/M$.  Since $\hat{x}\in Q_0$,
there is a surface ball
$$\Delta_1:= \Delta(x_1,cr/\sqrt{M})\subset Q_0\cap B(\hat{x},r/\sqrt{M})\subset Q_0\cap B.$$
Let $\F_1$ denote the collection of those
$Q_j\in\F$ which meet $\Delta_1$.  We then have the covering
$$\Delta_1\subset \left(\cup_{\F_1} Q_j \right)\cup \left(\Delta_1\setminus(\cup_{\F_1} Q_j)\right).$$
If \begin{equation}\label{eqA.6*}
\sigma\left(\frac12 \Delta_1\setminus (\cup_{\F_1} Q_j)\right)
\geq\frac12 \sigma\left(\frac12 \Delta_1\right) \approx  r^n,\end{equation}
then we are done, since $\Delta_1\setminus (\cup_{\F_1} Q_j)\subset
(Q_0\setminus(\cup_{\F} Q_j))\cap B\subset \td,$
by Proposition \ref{prop:sawtooth-contain}.

Otherwise, if \eqref{eqA.6*} fails, then
\begin{equation}\label{eqA.7**}
\sum_{Q_j\in\F_1'}\sigma(Q_j) \gtrsim  r^n,\end{equation}
where $\F_1'$ denotes those $Q_j \in\F_1$ which meet
$\frac12 \Delta_1$.
Let us remind the reader that $x^\star_j$ is the center of the $n$-dimensional cube $P_j$ constructed in
Proposition \ref{prop:Pj}, and we recall  \eqref{eqn:T-Pj} and the related discussion.
We claim that there is a uniform constant $C$ such that for each such $Q_j$,
the ball $B^*_{Q_j}:= B(x_j^\star, CK_0\, \ell(Q_j))$
contains both an interior and an exterior
Corkscrew point for $\Omega_{\F,Q_0}$, with respect to the surface ball
$B^*_{Q_j}\cap\pomfqo$ (with Corkscrew constants that may depend upon
$K_0$).

Indeed, the exterior point exists by virtue of Lemma \ref{lemma4.9},
while the interior point may be taken to be the center
of some $I\in \W^*_{\widetilde{Q}_j}$ with $\ell(I)\approx \ell(\widetilde{Q}_j)$, where $\widetilde{Q}_j$ is the dyadic parent of $Q_j,$ so that $\widetilde{Q}_j\in \dd_{\F,Q_0}$
and therefore $I\subset\interior(I^*)\subset \Omega_{\F,Q_0}$.  Consequently,
by the isoperimetric inequality and the structure theorem for sets of locally finite perimeter,
we have
\begin{equation}\label{eqA.7*}H^n(B^*_{Q_j}\cap\pomfqo)\gtrsim \ell(Q_j)^n\approx\sigma(Q_j).
\end{equation}
Now, by the ADR property and a covering lemma argument, and \eqref{eqA.7**},
there is a sub-collection
$\F_1''\subset \F_1'$ such that the balls in $\{B^*_{Q_j}\}_{Q_j\in\F_1''}$ are pairwise disjoint and
\begin{equation}\label{eqA.8*}\sum_{Q_j\in\F_1''}\sigma(Q_j)\gtrsim r^n.
\end{equation}
Combining \eqref{eqA.7*} and \eqref{eqA.8*}, we obtain that $H^n(\td)\gtrsim  r^n$,
since for $M$ large enough, each $B_{Q_j}^* \subset B$, by construction.

\subsection{UR}  In this subsection, we show that the Carleson box $T_Q$ inherits
the UR property from $\Omega$.  This fact extends routinely to any $T_\Delta$,
and we omit the details.

Let us note that, since $\partial\Omega$ is UR, we have the global $L^2$ bound
\begin{equation}\label{eqA.6a}
\dint_{\ree}|\nabla^2\mathcal{S}f (X)|^2\,\delta(X)\, dX
\,\leq\, C\,\|f\|^2_{L^2(\partial\Omega)},
\end{equation}
which is equivalent to the Carleson measure condition \eqref{eq1.sf} by ``$T1$ reasoning''.

Fix now $Q\in \dd(\partial\Omega)$, and as usual let $\tdelta(X):=\dist(X,\partial T_Q)$
(in the present context, $X$ need not belong to $T_Q$, but of course $\tdelta(X)$
is still well-defined).  By ``local $Tb$'' theory (see \cite{GM} in the present context), it is enough to verify that for every
$\td=\td(x,r) :=B(x,r)\cap\partial T_Q,$ with $x\in \partial T_Q$ and $r\lesssim \diam (Q)$,
there is a function $b_{\td}$, supported in $\td$, and satisfying
\begin{equation}
\label{eqA.6}
\left| \fint_{\td} b_{\td} \,dH^n\right| \,\geq \,\frac1C,
\end{equation}
\begin{equation}
\label{eqA.7}
\fint_{\td} |b_{\td}|^2 \,dH^n \,\leq \,C,
\end{equation}
\begin{equation}\label{eqA.8}
\dint_{B(x,2r)}|\nabla^2\mathcal{S} b_{\td} (X)|^2\,\tdelta(X)\, dX \leq C r^n,
\end{equation}
where $C$ is a uniform constant, independent of $Q$.  We fix a large constant $M$ to be chosen.
There are two cases:

\noindent{\bf Case 1}:  $\dist(x,\partial\Omega)\leq r/(MK_0).$

In this case, either $x\in \overline{Q}$, or $x$ lies on a face of some $I^*=(1+\lambda)I$,
with $I\in \W^*_{Q_I}$,
for some $Q_I\in\dd_Q$, where $\ell(Q_I)\approx \ell(I)\lesssim r/(MK_0)$,
and $\dist(Q_I,I)\lesssim K_0\,\ell(I) \lesssim r/M.$

We claim that there is a surface ball
$\Delta'=B'\cap\partial\Omega\subset \td\cap\partial\Omega$, with $r_{\Delta'}\approx r/M$,
and with $B'\cap\Omega\subset T_Q$.  Indeed,
if $x\in \overline{Q}$, then there is a $Q'\in \dd_Q$ such that $x\in \overline{Q'}$
and $\ell(Q')\approx r/M,$
and we may then set $\Delta':= B'\cap\partial\Omega$, where $B':=B'_{Q'}$ is
the ball promised by Lemma \ref{lemma2.31}, applied with $Q'$ in place of $Q$, so that
$B' \cap \Omega \subset T_{Q'}\subset T_Q$.
On the other hand,
if $x\in \partial I^*$, with $I\in \W^*_{Q_I}$ as above, then there is a
$Q'\in \dd_Q$ with $Q_I\subseteq Q'$, and $\ell(Q')\approx r/M.$
Moreover, by the triangle inequality, $|x-y|\lesssim r/M$, for every $y\in Q'$,
so that for $M$ large enough we have
$Q'\subset B(x,r)\cap \partial T_Q = \td$ by Proposition \ref{prop:sawtooth-contain}. Thus, we may again set $B':=B'_{Q'},$
as in Lemma \ref{lemma2.31},
and the claim is established.

We fix $\Delta'=\Delta(x_{Q'},r_{\Delta'})$
as in the previous paragraph, and  then set $b_{\td}:= 1_{\Delta''},$ where
$\Delta'':= \Delta(x_{Q'},r_{\Delta'}/4)$. Then  \eqref{eqA.7} is trivial, and
\eqref{eqA.6} holds  by the ADR properties of $\Omega$ and $T_Q$.
It remains to establish
\eqref{eqA.8}.
To this end, we claim that
$\tdelta(X)=\delta(X)$, for $X\in 2B''=\frac12 B'$.
Momentarily taking this claim for granted, we obtain that
$$
\dint_{2B''}|\nabla^2\mathcal{S} b_{\td} (X)|^2\,\tdelta(X)\, dX \leq C r^n,
$$
by \eqref{eqA.6a}, since $b_{\td}=1_{\Delta''}$ is supported in $\partial\Omega$.
Otherwise, for $X\in B(x,2r)\setminus 2B''$, we have
\begin{equation}\label{eqA.10}
|\nabla^2\mathcal{S} b_{\td} (X)| \lesssim \left|\int_{|X-y|\gtrsim r_{\Delta'}}|X-y|^{-n-1}
\,1_{\Delta''}(y)\,d\sigma(y)\right| \lesssim 1/r_{\Delta'}\approx M/r\,,
\end{equation}
by the ADR property, and \eqref{eqA.8} follows.

Let us now verify the claim.  Fix $X\in \frac12B'$.
We note that
$$\dist(X,\pom\cap\partial T_Q) \leq \frac12r_{\Delta'}\,,$$
since $B'$ is centered on $\pom \cap \partial T_Q$.
On the other hand, $\partial T_Q\setminus\pom \subset \Omega$, and therefore lies
outside of $B'$, since, by construction, $B'\cap\Omega\subset T_Q$.  Thus,
$$\dist(X, \partial T_Q\setminus\pom) \geq \frac12r_{\Delta'}\,.$$
Consequently, $\tdelta(X)=\dist(X,\pom\cap\partial T_Q)$.
Similarly, we shall have that $\delta(X)=\dist(X,\pom\cap\partial T_Q),$ and thus $\delta(X)=\tdelta(X)$
as claimed,
once we show that
$\pom\setminus \partial T_Q$
lies outside $B'$, or equivalently, that $\pom\cap B'\subset \partial T_Q$.
So, fix $y\in \pom\cap B'$.  Since $B'$ is open, we have that
$B(y,\eps_0)\subset B'$ for $\eps_0$ small enough.  Note that  $B(y,\eps_k)$ meets
$\Omega$ for a sequence $\eps_k\to 0$, with $\eps_k<\eps_0$.  Thus, there exists a sequence
$\{y_k\}_k\subset B'\cap\Omega\subset T_Q$, with $y_k\to y$, whence $y\in\partial T_Q$.

\noindent{\bf Case 2}:  $\dist(x,\partial\Omega)> r/(MK_0).$

 In this case, we can find a Whitney cube $I$ with $x\in \partial I^*$ and
 $\interior(I^*)\subset T_Q$, and a ball $B'=B(x',r')$,
 with $r'\approx r/(MK_0)$, such that some face $F$ of $I^*$
 contains the surface ball $\td':=B'\cap\partial T_Q.$
 We define $b_{\td}:= 1_{\td''},$ where
 $\td''=B''\cap\partial T_Q$ and $B'' :=\frac14 B'$.  We may now proceed as in Case 1, using that
 of course \eqref{eqA.6a} holds when $\partial\Omega$ is replaced by the hyper-plane $\mathcal{H}$
 that contains $F$,
 and $\delta(X)=\dist(X,\mathcal{H})$.
 We omit the routine details.

\section{Dyadically doubling and Muckenhoupt weights}\label{appendixB}

Recall that, for a fixed  cube $Q_0\in \dd$, we say that $\hm$ is dyadically doubling on $Q_0$
if there exists $C_\hm$ such that $\hm(Q)\le C_\hm\, \hm(Q')<\infty$ for every $Q\in \dd_{Q_0}$,  and for every dyadic ``child'' $Q'$ of $Q$. We write $C_\sigma$ for the dyadic doubling constant of $\sigma$ (which depends on the ADR property).  Throughout Appendix \ref{appendixB},
$Q_0$ will denote a fixed cube in $\dd$.  Let us also recall that the projection operators  $\P_\F$ have been introduced in Section \ref{s6*}.

\begin{lemma}\label{lemma:dyad-doubling-proj}
Fix $Q_0$. Let $\hm$  be a dyadically doubling measure on $Q_0$ with constant $C_\hm$. Then for every family $\F\subset \dd_{Q_0}$ of pairwise disjoint dyadic cubes, $\P_\F \hm$ is dyadically doubling on $Q_0$, indeed $\P_\F \hm(Q)\le \max(C_\hm,C_\sigma)\, \P_\F \hm(Q')$ for every $Q\in \dd_{Q_0}$, and for every dyadic ``child'' $Q'$ of $Q$.
\end{lemma}

\begin{proof}
We follow
the proof of \cite[Lemma B.1]{HM}.
Let us fix $Q\in \dd_{Q_0}$ and one of its dyadic ``children'' $Q'$. We consider several cases.

\noindent {\bf Case 1}: There exists $Q_k\in \F$ with $Q\subset Q_k$. The estimate is trivial in this case:
$$
\P_\F \hm(Q)
=
\frac{\sigma(Q)}{\sigma(Q_k)}\,\hm(Q_k)
\le C_\sigma
\frac{\sigma(Q')}{\sigma(Q_k)}\,\hm(Q_k)
= C_\sigma\,\P_\F \hm(Q')
<\infty.
$$

\noindent {\bf Case 2}: $Q'\in \F$. Notice that $\P_\F \hm(Q')=\hm(Q')$. Let $\F_1$ be the family of cubes $Q_k\in\F$ with $Q_k\cap Q\neq\emptyset$ and observe that if $Q_k\in \F_1$ then $Q_k\subsetneq Q$. Thus,
\begin{align*}
\P_\F \hm(Q)
&=
\hm(Q\setminus (\cup_{Q_k\in\F} Q_k)) + \sum_{Q_k\in \F_1} \frac{\sigma(Q_k\cap Q)}{\sigma(Q_k)} \hm(Q_k)
\\
&=
\hm(Q\setminus (\cup_{Q_k\in\F} Q_k)) + \sum_{Q_k\in \F_1} \hm(Q_k)
\\
&
=
\hm(Q)
\le
C_\hm\,\hm(Q')
=
C_\hm\,\P_\F \hm(Q')<\infty.
\end{align*}

\noindent {\bf Case 3}: None of the conditions in the previous cases occur. We take the same set $\F_1$ and observe that if $Q_k\in \F_1$ then $Q_k\subsetneq Q$ (otherwise we are driven to Case 1). Let $\F_2$ be the family of cubes $Q_k\in\F$ with $Q_k\cap Q' \neq\emptyset$. Notice that if $Q_k\in \F_2$ then $Q_k\subsetneq Q'$: otherwise, either $Q_k=Q'$ which leads us to Case 2, or $Q'\subsetneq Q_k$ which implies $Q\subset Q_k$ and this is Case 1. Then proceeding as in the previous case one obtains that $\P_\F \hm(Q)=\hm(Q)$ and $\P_\F \hm(Q')=\hm(Q')$ which in turn imply
$$
\P_\F \hm(Q)
=
\hm(Q)
\le
C_\hm\,\hm(Q')
=
C_\hm\,\P_\F \hm(Q')<\infty.
$$
\end{proof}

\begin{lemma}\label{lemma:dyad-doubling-nu}
Under the hypotheses of Lemma \ref{lemma:DJK-dyadic-proj}, $\nu$ and $\P_\F\nu$ are dyadically doubling on $Q_0$.
\end{lemma}

\begin{proof}
We proceed as in \cite[Lemma B.2]{HM}. Let us first consider $\nu$. Fix $Q\in \dd_{Q_0}$, and one of its dyadic ``children'' $Q'$.  We recall Proposition \ref{prop:Pj}, which for each
$Q_k\in\F$ promises the existence of an $n$-dimensional cube $P_k\subset \pom_{\F,Q_0}$,
with $\ell(P_k) \approx \ell(Q_k) \approx \dist(P_k,Q_k)\approx \dist(P_k,\pom)$.

\noindent {\bf Case 1}: There exists $Q_k\in \F$ with $Q\subset Q_k$. The estimate is trivial in this case since $\hm$ is dyadically doubling:
\begin{align*}
\nu(Q)
=
\frac{\hm(Q)}{\hm(Q_k)}\,\tom(P_k)
\le
C_\hm\,
\frac{\hm(Q')}{\hm(Q_k)}\,\tom(P_k)
=
C_\hm\, \nu(Q')<\infty.
\end{align*}

\noindent {\bf Case 2}: $Q'\in \F$. Write $Q'=Q_1\in\F$. Notice that $\nu(Q')=\tom(P_1) $. Let $\F_1$ be the family of cubes $Q_k\in\F$ with $Q_k\cap Q\neq\emptyset$ and observe that if $Q_k\in \F_1$ then $Q_k\subsetneq Q$. Thus, Remark \ref{remark:Pj} implies
\begin{align*}
\nu(Q)
&=
\tom\left(Q\setminus(\cup_{Q_k\in\F} Q_k)\right) + \sum_{Q_k\in \F_1} \frac{\hm(Q_k\cap Q)}{\hm(Q_k)} \tom(P_k)
\\
&=
\tom\left(Q\setminus (\cup_{Q_k\in\F} Q_k)\right) + \sum_{Q_k\in \F_1} \tom(P_k)
\\
&
\lesssim
\tom\left(
\big(Q\setminus (\cup_{Q_k\in\F} Q_k)\big)\cup \big(\cup_{Q_k\in\F_1} P_k\big)
\right).
\end{align*}
We note that there are uniform positive constants $c$ and $C$ such that
\begin{equation}
\label{eqn:dyad-doubling-nu:Claim1}
\big(Q\setminus (\cup_{Q_k\in\F} Q_k)\big)\cup \big(\cup_{Q_k\in\F_1} P_k\big)\subset \td(x^\star_1, C\ell(P_1))
\end{equation}
and
\begin{equation}
\label{eqn:dyad-doubling-nu:Claim2}
\td(x^\star_1, c\ell(P_1))
\subset P_1\,,
\end{equation}
where as usual $x_1^\star$ denotes the center of the $n$-dimensional cube $P_1$.  Indeed,
\eqref{eqn:dyad-doubling-nu:Claim2} is trivial, since by construction
(cf. Proposition \ref{prop:Pj}), $P_1\subset \pom_{\F.Q_0}$.
To verify \eqref{eqn:dyad-doubling-nu:Claim1},
it is enough to observe that, by Proposition
\ref{prop:Pj}, and the fact that $Q$ is the dyadic parent of $Q_1$, for $Q_k\in\F_1$ we have
$$\ell(P_k)\approx \dist(P_k,Q_k) \approx \dist(P_k,Q)\lesssim
\ell(P_1) \approx \ell(Q) \approx \dist(Q,P_1)\,.$$

Consequently, since $\tom$ is doubling,
we have
$$
\nu(Q)
\lesssim
\tom(\td(x^\star_1, C\ell(P_1)))
\lesssim
\tom(\td(x^\star_1, c\ell(P_1)))
\le
\tom(P_1)
=\nu(Q')
$$

\noindent {\bf Case 3}: None of the conditions in the previous cases occur. We take the same set $\F_1$ and observe that if $Q_k\in \F_1$ then $Q_k\subsetneq Q$ (otherwise we are driven to Case 1). Let $\F_2$ be the family of cubes $Q_k\in\F$ with $Q_k\cap Q' \neq\emptyset$. Notice that if $Q_k\in \F_2$ then $Q_k\subsetneq Q'$: otherwise, either $Q_k=Q'$ which leads us to Case 2, or $Q'\subsetneq Q_k$ which implies $Q\subset Q_k$ and this is Case 1. We claim that for some uniform constant $C$, we have
\begin{equation}
\label{eqn:dyad-doubling-nu:Claim3}
\big(Q\setminus (\cup_{Q_k\in\F} Q_k)\big)\cup \big(\cup_{Q_k\in\F_1} P_k\big)\subset \td(x^\star_{Q'}, Ct_{Q'})
\end{equation}
(see Proposition \ref{prop:surface-ball-sawtooth} for the notation).  Indeed,
since $Q$ is the dyadic parent of $Q'$, by the construction in
Proposition \ref{prop:surface-ball-sawtooth} (applied to $Q'$), we have that
$$\dist(x^\star_{Q'},Q)\leq\dist(x^\star_{Q'},Q') \lesssim \ell(Q') \approx \ell(Q) \approx t_{Q'}\,,$$
whence \eqref{eqn:dyad-doubling-nu:Claim3} follows immediately.

Then we proceed as in the previous case and  obtain that
\begin{multline*}
\nu(Q)
\lesssim
\tom\left(
\big(Q\setminus (\cup_{Q_k\in\F} Q_k)\big)\cup \big(\cup_{Q_k\in\F_1} P_k\big)
\right)
\\
\le
\tom(\td(x^\star_{Q'}, Ct_{Q'}))
\lesssim
\tom(\td^{Q'})
\lesssim
\nu(Q')
\end{multline*}
where we have used that $\tom$ is doubling and where the last inequality follows as in \eqref{proj-nu-bar:Q_1}:
$$
\nu(Q')
=
\tom(Q'\cap E_0)+ \sum_{Q_k\in \F_2} \tom(P_k)
\gtrsim
\tom(\td^{Q'}).
$$

One might show that $\P_\F\nu$ is dyadically doubling
by invoking Lemma \ref{lemma:dyad-doubling-proj},
but then the doubling constant would depend on $\hm$ and $\tom$. This is not the right approach as we have already observed that $\P_\F\nu$ does not depend on $\hm$. On the other hand,
following the previous argument for $\nu$
we can see that the doubling constant does not depend on $\hm$.
In Cases 2 and 3, we have that $\P_\F \nu(Q)=\nu(Q)$ and $\P_\F \nu(Q')=\nu(Q')$ so the
doubling condition follows at once from the previous computations, and  depends
quantitatively only upon
the doubling constant for $\tom$, but not on $\hm$.
In Case 1 we obtain
$$
\P_\F \nu(Q)
=
\frac{\sigma(Q)}{\sigma(Q_k)}\,\tom(P_k)
\le
C_\sigma
\frac{\sigma(Q')}{\sigma(Q_k)}\,\tom(P_k)
=
C_\sigma
\P_\F \nu(Q')
$$
\end{proof}

Let us remind the reader that, as explained above, we may view $\partial\Omega_{\F,Q_0}$ itself as a surface ball $\td^{Q_0}$ of radius
$r(\td^{Q_0})\approx K_0\,\ell(Q_0)$, and then $A_\infty(\partial\Omega_{\F,Q_0})$ is identified with $A_\infty(\td^{Q_0})$.

\begin{lemma}\label{lemmaB.4}
Under the hypotheses of Lemma \ref{lemma:DJK-dyadic-proj}, if $\tom
\in A_\infty(\partial\Omega_{\F,Q_0})$, then $\P_\F\nu\in A_\infty^{\rm dyadic}(Q_0)$.
\end{lemma}

\begin{proof}
Fix $0<\eta<1/2$ and $F\subset Q\in\dd_{Q_0}$ with $\sigma(F)\ge (1-\eta)\,\sigma(Q)$.

\noindent {\bf Case 1}: There exists $Q_k\in \F$ with $Q\subset Q_k$. The estimate is trivial in this case:
$$
\frac{\P_\F \nu(F)}{\P_\F \nu(Q)}
=
\frac{\frac{\sigma(F)}{\sigma(Q_k)}\tom(P_k)}{\frac{\sigma(Q)}{\sigma(Q_k)}\tom(P_k)}
=
\frac{\sigma(F)}{\sigma(Q)}
\ge 1-\eta.
$$

\noindent {\bf Case 2}:  $Q$ is not contained in any  $Q_k\in \F$ (i.e., $Q\in \dd_{\F,Q_0}$). Let $\F_1$ be the family of cubes $Q_k\in\F$ with $Q_k\cap Q\neq\emptyset$ and observe that if $Q_k\in \F_1$ then $Q_k\subsetneq Q$. We set
$$
\widetilde{\F}=\{Q_k\in\F_{1}: \sigma(F\cap Q_k)\ge (1-2\eta)\,\sigma(Q_k)\},
$$
and
$$
E_0=Q\setminus \bigcup_{Q_k\in\F} Q_k,
\qquad\quad
G
=
\bigcup_{Q_k\in\widetilde{\F}} Q_k,
\qquad\quad
B
=
\bigcup_{Q_k\in\F_1\setminus \widetilde{\F}} Q_k,
$$
Note that
$$
\sigma(F\cap B)
=
\sum_{Q_k\in \F_1\setminus \widetilde{\F}} \sigma(F\cap Q_k)
\le
(1-2\eta)\,\sum_{Q_k\in \F_1\setminus \widetilde{\F}} \sigma(Q_k)
\le
(1-2\eta)\,\sigma(Q).
$$
Thus,
\begin{multline*}
(1-\eta)\,\sigma(Q)
\le
\sigma(F)
\le
\sigma(F\cap E_0)+\sigma(F\cap B)+\sigma(F\cap G)
\\
\le
\sigma\big((F\cap E_0)\cup G\big)+(1-2\eta)\,\sigma(Q),
\end{multline*}
and therefore $\sigma\big((F\cap E_0)\cup G\big)\ge \eta\,\sigma(Q)$.

Note that the ADR property of $\pom$ and $\pomfqo$ imply
$$
\sigma(Q)\approx \ell(Q)^n\approx (\widehat{r}_Q)^n\approx \sigma_\star(\Delta_\star(y_Q,\widehat{r}_Q))
$$
with $\Delta_\star(y_Q,\widehat{r}_Q)$ given in \eqref{eqn:Pj-cork} (see also Proposition \ref{prop:cork-both}),  where as usual
we write $\sigma_\star$ to denote the ``surface measure'' on $\pomfqo$, i.e., $\sigma_\star=H^n\big|_{\pomfqo}$. If we set $G_\star=\cup_{Q_k\in\widetilde{\F}} P_k$ we have that $\sigma(G)\approx \sigma_\star(G_\star)$.  Indeed, since $\Omega_{\F_,Q_0}$ is ADR we have that
$$\sigma_\star(P_k)\approx\ell(P_k)^n\approx\ell(Q_k)^n\approx \sigma(Q_k)\,,$$
by Proposition \ref{prop:Pj}; thus Remark \ref{remark:Pj} yields
$$
\sigma(G)
=
\sum_{Q_k\in\widetilde{\F}} \sigma(Q_k)
\approx
\sum_{Q_k\in\widetilde{\F}} \sigma_\star(P_k)
\approx
\sigma_\star(G_\star).
$$
On the other hand, Proposition \ref{prop:sawtooth-contain} gives $\sigma(F\cap E_0)=\sigma_\star(F\cap E_0)$ and therefore
$$
\sigma_\star\big((F\cap E_0)\cup G_\star\big)\approx\sigma\big((F\cap E_0)\cup G\big)\ge \eta\, \sigma(Q)
\approx
\eta\, \sigma_\star(\Delta_\star(y_Q,\widehat{r}_Q)).
$$
Next we use that $\omega_\star\in A_\infty(\pomfqo)$
to obtain
$$
\frac{\omega_\star \big((F\cap E_0)\cup G_\star\big)}{\omega_\star (\Delta_\star(y_Q,\widehat{r}_Q))}
\gtrsim
\left(
\frac{\sigma_\star \big((F\cap E_0)\cup G_\star\big)}{\sigma_\star (\Delta_\star(y_Q,\widehat{r}_Q))}
\right)^\theta
\gtrsim \eta^{\theta},
$$
where we have used that $(F\cap E_0)\cup G_\star\subset \Delta_\star(y_Q,\widehat{r}_Q)$ by \eqref{eqn:Pj-cork}. Then,
\begin{align*}
\P_\F \nu(F)
&\ge
\omega_\star(F\cap E_0)+\sum_{Q_k\in\widetilde\F} \frac{\sigma(F\cap Q_k)}{\sigma(Q_k)}\omega_\star(P_k)
\\
&\ge
\omega_\star(F\cap E_0)+(1-2\,\eta)\sum_{Q_k\in\widetilde\F} \omega_\star(P_k)
\\
&\ge
(1-2\,\eta)\omega_\star \big((F\cap E_0)\cup G_\star\big)
\\
&\gtrsim
(1-2\,\eta) \eta^{\theta}\omega_\star (\Delta_\star(y_Q,\widehat{r}_Q))
\\
&\gtrsim
(1-2\,\eta) \eta^{\theta}\omega_\star \left(\left(Q \cup \left(\cup_{Q_k\in\F:\,Q_k\subset Q}B(x^\star_k,r_k)\right)\right)\cap \pomfqo\right)
\end{align*}
where the last inequality follows from \eqref{eqn:Pj-cork}. Next we observe that, by Proposition \ref{prop:sawtooth-contain},
\begin{multline*}
(Q\cap E_0)\cup \left(\cup_{Q_k\in \F_1}P_k\right)
\subset
(Q\cap E_0)\cup \left(\cup_{Q_k\in \F_1} \Delta_\star(x_k^\star,r_k)\right)
\\
\subset
\left(Q\cup\left(\cup_{Q_k\in\F:\,Q_k\subset Q} B(x_k^\star,r_j)\right)\right)\cap \pomfqo.
\end{multline*}
Consequently,
\begin{multline*}
\P_\F \nu(F)
\gtrsim
(1-2\,\eta) \eta^{\theta}\omega_\star \left((Q\cap E_0)\cup \left(\cup_{Q_k\in \F_1}P_k\right)\right)
\\
\gtrsim
(1-2\,\eta) \eta^{\theta}
\bigg(\omega_\star (Q\cap E_0) +\sum_{Q_k\in \F_1} \omega_\star (P_k)\bigg)
=
(1-2\,\eta) \eta^{\theta}\P_\F\nu(Q).
\end{multline*}

Thus, in both cases we have shown as desired that $\P_\F \nu(F)/\P_\F\nu(Q)\ge C_\eta$.
\end{proof}

Next we give a version of the classical result in \cite{CF} valid in our situation. The proof of this result follows the standard arguments in \cite{GR} although one has to adapt the ideas to the dyadic and local setting considered here. We give the proof for completeness.

\begin{lemma}\label{lemma:CF-dyadic}
Let $Q_0$ be a fixed cube and  let $\hm_1$, $\hm_2$ be two dyadically doubling measures  on $Q_0$ . Assume that there exist positive constants $C_0$, $\theta_0$ such that for all $Q\in \dd_{Q_0}$ and $F\subset Q$,
\begin{equation}\label{CF-dyadic:hyp}
\frac{\hm_2(F)}{\hm_2(Q)}
\le
C_0\,\bigg(\frac{\hm_1(F)}{\hm_1(Q)}\bigg)^{\theta_0}.
\end{equation}
Then, there exist positive constants $C_1$, $\theta_1$ such that for all $Q\in \dd_{Q_0}$ and $F\subset Q$,
\begin{equation}\label{CF-dyadic:conc}
\frac{\hm_1(F)}{\hm_1(Q)}
\le
C_1\,\bigg(\frac{\hm_2(F)}{\hm_2(Q)}\bigg)^{\theta_1}.
\end{equation}
\end{lemma}

\begin{remark}\label{remark:Ainfty}
The proof shows that the desired estimate can be obtained from the following (apparently) weaker condition: there exist $0<\alpha,\beta<1$ such that for every cube $Q\in\dd_{Q_0}$,
\begin{equation}\label{CF-Ainfty:w}
F\subset Q,\quad \frac{\hm_2(F)}{\hm_2(Q)}<\alpha\quad \Longrightarrow \quad \frac{\hm_1(F)}{\hm_1(Q)}<\beta.
\end{equation}
\end{remark}

To prove this result we need a local Calder\'on-Zygmund decomposition for dyadically doubling weights. The proof is standard and we leave it to the  interested reader.

\begin{lemma}\label{dyadic-HL:w}
Given $Q_0$ and $\hm$ a dyadically doubling measure on $Q_0$ with constant $C_\hm$, we consider the local dyadic Hardy-Littlewood maximal function with respect to $\hm$:
$$
\M_\hm f(x)
=
\sup_{x\in Q\in \dd_{Q_0}} \frac1{\hm(Q)}\,\int_{Q} |f(y)|\,d\hm(y).
$$
For any $0\le f\in L^1(Q_0,\hm)$ and  $\lambda\ge \frac1{\hm(Q_0)}\,\int_{Q_0} |f(y)|\,d\hm(y)$,  there exists a collection of maximal and therefore disjoint dyadic cubes $\{Q_j\}\subset\dd_{Q_0}$ such that
\begin{equation}
E_\lambda=\{x\in Q_0: \M_\hm f(x)>\lambda\}
=
\bigcup_j Q_j, \label{CZD-Omega}
\end{equation}
\begin{equation}
f(x)\le \lambda, \quad\mbox{for $\hm$-a.e.~}x\notin E_\lambda, \label{CZD-out-Omega}
\end{equation}
\begin{equation}
\lambda
<
\frac1{\hm(Q_j)}\,\int_{Q_j} f(y)\,d\hm(y)
\le
C_\hm\,\lambda. \label{CZD-aver}
\end{equation}
\end{lemma}

\begin{proof}[Proof of Lemma \ref{lemma:CF-dyadic}]
We proceed as in \cite[Lemma B.4]{HM}. Pick $0<\alpha<1$ and $\beta= 1-\big(\frac{1-\alpha}{C_0}\big)^{1/\theta_0}$, and notice that $0<\beta<1$ since $C_0\ge 1$. Then for any $F\subset Q$, $Q\in\dd_{Q_0}$ we apply \eqref{CF-dyadic:hyp} to $Q\setminus F$ and we conclude \eqref{CF-Ainfty:w}. Next we see that this (apparently) weaker condition implies the desired conclusion. Assume momentarily that $\hm_1\ll\hm_2$. Then the Radon-Nikodym derivative $h=d\hm_1/d\hm_2$ satisfies that $h\in L^1(Q_0,\hm_2)$ and $0\le h(x)<\infty$ for $\hm_2$-a.e.~$x\in Q_0$.

Fixed $Q\in\dd_{Q_0}$ we write $\tau=C_{\hm_2}/\alpha$,
$$
\lambda_0
=
\frac1{\hm_2(Q)}\,\int_{Q} h(x)\,d\hm_2(x)
=
\frac{\hm_1(Q)}{\hm_2(Q)}
$$
and $\lambda_k=\tau^k\,\lambda_0$. Notice that $\lambda_0<\lambda_1<\lambda_2<\cdots$ since $\tau>C_{\hm_2}\ge 1$. For every $k\ge 0$ we apply Lemma \ref{dyadic-HL:w} in $Q$ to $h$ with dyadically doubling measure $\hm_2$: let $\{Q_j^k\}_j\subset \dd_Q\subset \dd_{Q_0}$ be the corresponding collection of cubes such that $E_k=E_{\lambda_k}=\cup_j Q_j^k$. Fix $Q_{j_0}^{k}$ and observe that if $Q_{j_0}^k\cap Q_j^{k+1}\neq\emptyset$ then $Q_{j}^{k+1}\subset Q_{j_0}^k$: otherwise we would have $Q_{j_0}^k\subsetneq Q_{j}^{k+1}$, by \eqref{CZD-aver} we observe that $\frac1{\hm_2(Q_j^{k+1})}\,\int_{Q_j^{k+1}} h\,d\hm_2>\lambda_{k+1}>\lambda_k$ and then $Q_{j_0}^k$ would not be maximal. Then using \eqref{CZD-Omega} and \eqref{CZD-aver} we obtain
\begin{align*}
\hm_2(Q_{j_0}^k\cap E_{k+1})
&=
\sum_{j: Q_{j}^{k+1}\subset Q_{j_0}^k} \hm_2(Q_j^{k+1})
<
\frac1{\lambda_{k+1}} \sum_{j: Q_{j}^{k+1}\subset Q_{j_0}^k} \int_{Q_j^{k+1}} h\,d\hm_2
\\
&
\le
\frac1{\lambda_{k+1}} \int_{Q_{j_0}^k}h\,d\hm_2
\le
\frac{C_{\hm_2} \lambda_k}{\lambda_{k+1}}\,\hm_2(Q_{j_0}^k)
=
\alpha\,\hm_2(Q_{j_0}^k).
\end{align*}
This estimate allows us to use \eqref{CF-Ainfty:w} which in turn gives that $\hm_1(Q_{j_0}^k\cap E_{k+1})<\beta\,\hm_1(Q_{j_0}^k)$. Next we sum on $j_0$ and conclude that
$\hm_1(E_{k+1})<\beta \,\hm_1(E_{k})$ since $E_{k+1}\subset E_k$. By iterating this expression we obtain $\hm_1(E_{k})<\beta^k \,\hm_1(E_{0})$. Similarly, $\hm_2(E_{k})<\alpha^k \,\hm_1(E_{0})$, which implies
$$
\hm_2(\cap_k E_k)=\lim_{k\to\infty} \hm_2(E_k)=0.
$$
Let $0<\epsilon<-\log \beta/\log \tau$. Then $0<\tau^\epsilon\,\beta<1$ and by \eqref{CZD-out-Omega}
\begin{align}
&\frac1{\hm_2(Q)} \int_{Q}h(x)^{1+\epsilon}\,d\hm_2(x)
\label{RH:w1-w2}
\\&
=
\frac1{\hm_2(Q)}\, \int_{Q\setminus E_0} h(x)^{1+\epsilon}\,d\hm_2(x)
+
\frac1{\hm_2(Q)}\, \sum_{k=0}^\infty \int_{E_k\setminus E_{k+1}} h(x)^{1+\epsilon}\,d\hm_2(x)
\nonumber
\\
&\le
\lambda_0^\epsilon\, \frac1{\hm_2(Q)}\,\int_{Q} h(x)\,d\hm_2(x)
+
\frac1{\hm_2(Q)}\, \sum_{k=0}^\infty \lambda_{k+1}^\epsilon\, \int_{E_k} h(x)\,d\hm_2(x)
\nonumber
\\
&=
\lambda_0^\epsilon\, \frac{\hm_1(Q)}{\hm_2(Q)}
+
\frac1{\hm_2(Q)}\, \sum_{k=0}^\infty \lambda_{k+1}^\epsilon\, \hm_1(E_k)
\nonumber
\\
&
\le
\lambda_0^\epsilon\, \frac{\hm_1(Q)}{\hm_2(Q)}
+
\lambda_0^\epsilon\, \frac{\hm_1(E_0)}{\hm_2(Q)}\, \sum_{k=0}^\infty \tau^{(k+1)\,\epsilon}\,\beta^k
\nonumber
\\
&
\le
\lambda_0^\epsilon\, \frac{\hm_1(Q)}{\hm_2(Q)}\,(1+\tau^{\epsilon}\,(1-\tau^\epsilon\,\beta)^{-1})
\nonumber
\\
&
=
\bigg(\frac{\hm_1(Q)}{\hm_2(Q)}\bigg)^{1+\epsilon}\,C_1^{1+\epsilon}.\nonumber
\end{align}
This estimate implies that for all $F\subset Q$,
\begin{align*}
\frac{\hm_1(F)}{\hm_2(Q)}
&
=
\frac{1}{\hm_2(Q)}\,\int_{Q}\,\chi_F\,h\,d\hm_2
\le
\bigg(\frac{1}{\hm_2(Q)}\,\int_{Q}\,h^{1+\epsilon}\,d\hm_2\bigg)^{\frac1{1+\epsilon}}\,
\bigg(\frac{\hm_2(F)}{\hm_2(Q)}\bigg)^{\frac1{(1+\epsilon)'}}
\\
&\le
\frac{\hm_1(Q)}{\hm_2(Q)}\,C_1\,
\bigg(\frac{\hm_2(F)}{\hm_2(Q)}\bigg)^{\frac1{(1+\epsilon)'}},
\end{align*}
which is \eqref{CF-dyadic:conc} with $\theta_1=1/(1+\epsilon)'$. Notice that $\epsilon$ and $C_1$ depend only on $\alpha$, $\beta$ and $C_{\hm_2}$.

Next we see how to proceed in the general case starting from \eqref{CF-Ainfty:w}. We define a new measure $\tilde{\hm}_2=\hm_2+\delta\,\hm_1$ with $\delta>0$. It is clear that $\hm_1\ll\tilde{\hm}_2$ and also that $\tilde{\hm}_2$ is dyadically doubling  on $Q_0$ with constant $C_{\tilde{\hm}_2}=C_{\hm_1}+C_{\hm_2}$. We claim that setting $\tilde{\beta}=1-\min\{1-\beta,\alpha/2\}$, $\tilde{\alpha}=\alpha/2$ we have  for every $Q\in \dd_{Q_0}$,
\begin{equation}\label{claim-Ainfty}
F\subset Q,\quad \frac{\tilde{\hm}_2(F)}{\tilde{\hm}_2(Q)}<\tilde{\alpha}\quad \Longrightarrow \quad \frac{\hm_1(F)}{\hm_1(Q)}<\tilde{\beta}.
\end{equation}
Assuming this, \eqref{CF-Ainfty:w} holds for $\hm_1$, $\tilde{\hm}_2$. By the previous case, since $\hm_1\ll\tilde{\hm}_2$, there exist $\tilde{\epsilon}$, $\tilde{C}_1$
such that for every $Q\in\dd_{Q_0}$, $F\subset Q$ we have
$$
\frac{\hm_1(F)}{\hm_1(Q)}
\le
\tilde{C}_1\,
\bigg(\frac{\tilde{\hm}_2(F)}{\tilde{\hm}_2(Q)}\bigg)^{\frac1{(1+\tilde{\epsilon})'}}.
$$
As mentioned above  $\tilde{\epsilon}$, $\tilde{C}_1$ depend only on $\tilde{\alpha}$, $\tilde{\beta}$, $C_{\tilde{\hm}_2}$ and these are ultimately  given in terms of $\alpha$, $\beta$, $C_{\hm_1}$, $C_{\hm_2}$. Next we see that $\hm_1\ll \hm_2$: given $F\subset Q_0$ with $\hm_2(F)=0$, the previous inequality applied to $Q=Q_0$ gives as desired
$$
0\le \frac{\hm_1(F)}{\hm_1(Q)}
\le
\tilde{C}_1\,\bigg(\frac{\delta\,\hm_1(F)}{\tilde{\hm}_2(Q_0)}\bigg)^{\frac1{(1+\tilde{\epsilon})'}}
\le
\tilde{C}_1\,\bigg(\delta\,\frac{\hm_1(F)}{\hm_2(Q_0)}\bigg)^{\frac1{(1+\tilde{\epsilon})'}}\longrightarrow 0, \quad \mbox{ as }\delta\to 0^+.
$$
Thus, we get back to the first case and obtain \eqref{RH:w1-w2} which eventually leads to \eqref{CF-dyadic:conc} with $C_1$ and $\theta_1$ as stated above.

To complete the proof we obtain \eqref{claim-Ainfty}. Given $F$ as there, it follows that $\tilde{\hm}_2(Q\setminus F)/\tilde{\hm}_2(Q)>1-\alpha/2$. We see that $\hm_1(Q\setminus F)/\hm_1(Q)>\min\{1-\beta,\alpha/2\}$, which yields as desired $\hm_1(F)/\hm_1(Q)<\tilde{\beta}$. If this were not the case then we would have $\hm_1(Q\setminus F)/\hm_1(Q)\le \alpha/2$ and also that $\hm_1(F)/\hm_1(Q)\ge \beta$. By \eqref{CF-Ainfty:w}, the latter gives $\hm_2(F)/\hm_2(Q)\ge \alpha$ and therefore $\hm_2(Q\setminus F)/\hm_2(Q)\le 1-\alpha$. Gathering these estimates we get a contradiction
$$
\frac{\tilde{\hm}_2(Q\setminus F)}{\tilde{\hm}_2(Q)}
=
\frac{\hm_2(Q\setminus F)}{\tilde{\hm}_2(Q)}
+\delta\,
\frac{\hm_1(Q\setminus F)}{\tilde{\hm}_2(Q)}
\le
\frac{\hm_2(Q\setminus F)}{\hm_2(Q)}
+
\frac{\hm_1(Q\setminus F)}{\hm_1(Q)}
\le
1-\alpha/2.
$$
\end{proof}

\begin{remark}\label{remark:CF-dyadic}
Let us observe that \eqref{RH:w1-w2} can be equivalently written as
$$
\left(\frac1{\hm_2(Q)} \int_{Q}h(x)^{1+\epsilon}\,d\hm_2(x)
\right)^{\frac1{1+\epsilon}}
\le
C_1\,\frac1{\hm_2(Q)}\,\int_{Q} h(x)\,d\hm_2(x),
$$
and this shows that $h\in RH_{1+\epsilon}^{\rm dyadic}(Q_0,\hm_2)$.
\end{remark}

\section{The UR property for Approximating domains}\label{appendix:approx}

We establish the UR property (with uniform constants) for
the approximating domains
$\Omega_N$ defined by \eqref {eq7.on}.
Recall that we have already observed that $\Omega_N$ inherits the ADR, Corkscrew
and Harnack Chain conditions from $\Omega$.

The proof is based on ideas of Guy David,
and uses the following singular integral characterization of UR sets, established in \cite{DS1}.
Suppose that
$E\subset\ree$ is $n$-dimensional ADR.
The singular integral operators that we shall consider are those of the form
$$T_{E,\eps}f(x)=T_{\eps} f(x):= \int_{E} K_\eps (x-y)\,f(y)\,dH^n(y)\,,$$
where $K_\eps(x) := K (x)\,\Phi(|x|/\eps)$,  with $0\leq \Phi\leq 1$,
$\Phi(\rho)\equiv 1$ if $ \rho\geq 2,$ $\Phi(\rho) \equiv 0$ if $\rho\leq 1$, and $\Phi \in
C^\infty(\mathbb{R})$, and where
the singular kernel $K $ is an odd function, smooth on $\ree\setminus\{0\}$,
and satisfying
\begin{eqnarray}\label{eqC.1}& |K(x)|\,\leq\, C\,|x|^{-n}\\[4pt]\label{eqC.2}
&|\nabla^m K(x)|\,\leq \,C_m\,|x|^{-n-m}\,,\qquad\forall m=1,2,3,\dots ... \,.
\end{eqnarray}
Then $E$ is UR if and only if for every such kernel $K$, we have that
\begin{equation}\label{eqC.3}
\sup_{\eps>0}\int_E |T_\eps f|^2\,dH^n\leq C_K \int_E|f|^2\,dH^n.
\end{equation}
We refer the reader to \cite{DS1} for the proof.
We shall also require  ``non-tangential''
estimates for an extension of $T_\eps$ defined as follows.  For $K$ as above,
set
\begin{equation}\label{eqC.4}
\T_E f(X):= \int_E K(X-y)\,f(y)\,dH^n(y)\,,\qquad X \in \ree\setminus E.
\end{equation}
We define non-tangential
approach regions $\Gamma_\tau(x)$ as follows.  Let
$\mathcal{W}_E$ denote the collection of cubes in the Whitney decomposition
of $\ree\setminus E$, and set $\mathcal{W}_\tau(x):= \{I\in \mathcal{W}_E: \dist(I,x) <\tau \ell(I)\}$.
Then we define
$$\Gamma_\tau(x):=\bigcup_{I\in \mathcal{W}_\tau(x)} I^*$$
(thus, roughly speaking, $\tau$ is the ``aperture'' of $\Gamma_\tau(x)$).  For
$F\in \mathcal{C}(\ree\setminus E)$ we
may then also define the non-tangential maximal function
$$N_{*,\tau} (F)(x):= \sup_{Y\in \Gamma_\tau(x)}|F(Y)|.$$
We shall sometimes write simply $N_*$ when there is no chance of confusion
in leaving implicit the dependence on the aperture $\tau$.

\begin{lemma}\label{lemmacotlar}
Suppose that $E\subset \ree$ is $n$-dimensional UR, and let
$\T_E$ be defined as in \eqref{eqC.4}.  Then for each $\tau \in (0,\infty)$,
there is a constant $C_{\tau,K}$ depending only on $n,\tau, K$ and the UR constants such that
\begin{equation}\label{eqC.6}
\int_E \left(N_{*,\tau}\left(\T_E f\right)\right)^2\, dH^n \,\leq\, C_{\tau,K} \int_E |f|^2 dH^n.
\end{equation}
\end{lemma}

Given \eqref{eqC.3}, Lemma \ref{lemmacotlar} is a
variant of the standard ``Cotlar inequality'' for maximal singular
integrals, and we omit the proof.

We are now ready to prove that $\partial\Omega_N$ is UR, uniformly in $N$. It is enough
to establish the estimate \eqref{eqC.3}, for all $K$ as above, with $E$ replaced by $\partial\Omega_N$.
On the other hand, we are given that $\partial\Omega$ is UR, whence \eqref{eqC.6} holds with
$E=\partial\Omega$.  Since $\partial\Omega_N$ is ADR,
it enjoys the dyadic grid structure promised by Lemma \ref{lemmaCh}.  We then make a partition
$\partial \Omega_N =\cup Q_j(N)$, where $Q_j(N) \in \dd_N(\partial\Omega_N)=:\dd_N(N),$
the dyadic grid on $\partial\Omega_N$ at scale $2^{-N}.$  We observe that, by the construction
of $\Omega_N$,
for each $Q_j(N)\in\dd_N(N)$, we may choose a $Q_j\in \dd_N(\partial\Omega)=:\dd_N$
with $\dist(Q_j(N),Q_j)\approx 2^{-N}.$  By the ADR property of $\partial\Omega$,
a given $Q\in \dd_N$ can serve in this way for at most a bounded number of
$Q_j(N)\in \dd_N(N)$.  Therefore, we have the bounded
overlap condition
\begin{equation}\label{eqC.7}
\sum_{Q_j(N)\in\dd_N(N)} 1_{Q_j}(x)\leq C\,,\qquad\forall x\in\partial\Omega.
\end{equation}
As usual, we set $\sigma:=H^n|_{\pom}$, and we now also let
$\sigma_N:=H^n|_{\pom_N}$. We then have that for $\tau$ large enough,
\begin{multline*}
\int_{\partial\Omega_N} |\T_{\partial\Omega} f|^2\, d\sigma_N
=\sum_{Q_j(N)\in\dd_N(N)}
\int_{Q_j(N)} |\T_{\partial\Omega} f|^2\, d\sigma_N\\[4pt]
=\sum_{Q_j(N)\in\dd_N(N)} \frac1{\sigma(Q_j)}\int_{Q_j}
\int_{Q_j(N)} |\T_{\partial\Omega} f(x)|^2\, d\sigma_N(x)\,d\sigma(x')\\[4pt]
\lesssim \sum_{Q_j(N)\in\dd_N(N)} \int_{Q_j}
\left(N_{*,\tau}\left(\T_{\partial\Omega} f\right)\right)^2\, d\sigma
 \leq C_{\tau,K} \int_{\partial\Omega}
|f|^2 d\sigma\,,
\end{multline*}
where in the last line we have used first the ADR properties of $\partial\Omega$ and
$\partial\Omega_N$, and then \eqref{eqC.7} and \eqref{eqC.6} with $E=\partial\Omega$.

We have thus established that
$\T_{\partial\Omega}:L^2(\partial\Omega)\to L^2(\partial\Omega_N)$.  Since the kernel
$K$ is odd, we therefore obtain by duality that
\begin{equation}\label{eqC.8}
\T_{\partial\Omega_N}:L^2(\partial\Omega_N)\to L^2(\partial\Omega).
\end{equation}

Now fix $\eps>0$, and $N$ large enough that $2^{-N} \ll \diam{\partial\Omega}$.
Set $\eps_N = 2^{-N}$.  We consider two cases.

\noindent{\bf Case 1}: $\eps<\eps_N$.
In this case,
\begin{multline*}\int_{\partial\Omega_N} |T_{\partial\Omega_N,\eps} f|^2\,d\sigma_N
\\[4pt]
\lesssim \,
\int_{\partial\Omega_N}|T_{\partial\Omega_N,\eps} f-T_{\partial\Omega_N,\eps_N} f|^2\,d\sigma_N
\,\,\,+\,\int_{\partial\Omega_N}|T_{\partial\Omega_N,\eps_N} f|^2\,d\sigma_N\,=:\,I+II.
\end{multline*}
Let $Q_j\in\dd_N$ denote the cube chosen relative
to $Q_j(N)\in\dd_N(N)$ as above.
Then for $x\in Q_j(N)$, and $x'\in Q_j$, we have by standard Calder\'on-Zygmund estimates
using  \eqref{eqC.1} and \eqref{eqC.2}, and the ADR property of $\pom_N$, that
\begin{equation*}
|T_{\partial\Omega_N,\eps_N} f(x)-
\T_{\partial\Omega_N} f(x')| \lesssim M^Nf(x)\,,
\end{equation*}
where $M^N$ denotes the Hardy-Littlewood maximal function on $\pom_N$. Consequently,
\begin{multline*}II=\sum_{Q_j(N)\in\dd_N(N)} \frac1{\sigma(Q_j)}\int_{Q_j}
\int_{Q_j(N)}|T_{\partial\Omega_N,\eps_N} f(x)|^2\,d\sigma_N(x)\,d\sigma(x')\\[4pt]
\lesssim \,\,\sum_{Q_j(N)\in\dd_N(N)}
\int_{Q_j(N)}\left(M^Nf(x)\right)^2\,d\sigma_N(x)\\[4pt]
+\,\sum_{Q_j(N)\in\dd_N(N)} \int_{Q_j}
|\T_{\partial\Omega_N} f(x')|^2\,d\sigma(x') \,=:\,II' + II''
\end{multline*}
The desired bound for $II'$ follows immediately, and the bound for $II''$ follows
directly from \eqref{eqC.7} and \eqref{eqC.8}.

We turn now to term $I$.
Let us note that since $\eps_N\approx \diam(Q_j(N))$, for $x\in Q_j(N)$ we have
\begin{multline*}
T_{\partial\Omega_N,\eps} f(x)-T_{\partial\Omega_N,\eps_N} f(x)
\\[4pt]=
\int_{\partial\Omega_N} K(x-y)
\,\left(\Phi\left(\frac{|x-y|}{\eps}\right)-\Phi\left(\frac{|x-y|}{\eps_N}\right)\right)\,
f(y)1_{\Delta_{N,j}}(y)\,
d\sigma_N(y)\\[4pt]
=:\,T_{\pom_N,\eps,\eps_N} \left(f\,1_{\Delta_{N,j}}\right)(x)\,,
\end{multline*}
a doubly truncated singular integral on $\pom_N$,
where $\Delta_{N,j}:=B_{N,j}\cap\pom_N$, and $B_{N,j}$ is a ball
centered at some point in $Q_j(N)$, with radius
$C\diam(Q_j(N))\approx 2^{-N}.$  By choosing $C$ large enough, we may assume that
$Q_j(N)\subset \Delta_{N,j}$.
We recall that by definition, $\pom_N$ is a union of portions of
faces of fattened Whitney cubes
$I^*$, of side length $\approx 2^{-N}$.  Since only a bounded number
of these can meet $B_{N,j}$, we have
$$\Delta_{N,j} \subset \cup_{m=1}^{M_0} F^j_m\,,$$
where $M_0$ is a uniform constant and each $F^j_m$ is either
a portion of a face of some $I^*$, or else $F^j_m=\emptyset$ (since $M_0$ is not necessarily
equal to the number of faces, but is rather an upper bound for the number of faces.)
Thus,
$$I\lesssim\sum_{Q_j(N)\in\dd_N(N)}\sum_{1\leq m,m'\leq M_0}
\int_{F^j_m}|T_{\pom_N,\eps,\eps_N} \left(f\,1_{F^j_{m'}}\right)|^2\,d\sigma_N\,.$$
The faces $F^j_{m'}$ have bounded overlaps as we sum in $j$.  Therefore, the case
$m=m'$ reduces to the classical case that $\pom_N$ is a hyperplane.  For $m\neq m'$,
there are two cases as follows.
If $\dist(F^j_m,F^j_{m'})\approx 2^{-N},$ then using \eqref{eqC.1},
we may crudely dominate
$T_{\pom_N,\eps,\eps_N}$ by the Hardy-Littlewood maximal operator.
Otherwise,  $\dist(F^j_m,F^j_{m'})\ll 2^{-N},$ in which case
$F^j_{m}$ and $F^j_{m'}$ are contained in respective faces which either lie in the same hyperplane,
or else meet at an angle
of $\pi/2$.  In the latter scenario, after a possible rotation of co-ordinates, we may view
$F^j_{m}\cup F^j_{m'}$ as lying in a Lipschitz graph with Lipschitz constant 1, so that we may
estimate $T_{\pom_N,\eps,\eps_N}$ using an extension of the Coifman-McIntosh-Meyer theorem.

\noindent{\bf Case 2}: $\eps\geq\eps_N$.  We observe that \eqref{eqC.8} also
applies to the modified operator
$\T_{\pom_N}^\eps$, obtained by replacing the kernel $K$ by the kernel
$K_\eps$, since the latter is still odd and still satisfies the Calder\'on-Zygmund
estimates \eqref{eqC.1} and \eqref{eqC.2} (uniformly in $\eps$).
The present case may then be handled
just like term $II$ above, by writing
$$T_{\pom_N,\eps}f(x)=\left(T_{\pom_N,\eps}f(x)-\T_{\pom_N}^\eps f(x')\right)
+\T_{\pom_N}^\eps f(x')\,.$$
There is no term $I$.  We leave the details to the reader.

\end{document}